\documentclass[11pt]{article}

\usepackage[utf8]{inputenc}
\usepackage[T1]{fontenc}
\usepackage{url}
\usepackage{booktabs}
\usepackage{amsfonts}
\usepackage{nicefrac}
\usepackage{microtype}

\usepackage[margin=1in,letterpaper]{geometry}
\usepackage[round]{natbib}
\usepackage{color}
\usepackage{amsmath,amsthm}
\usepackage{mathrsfs}
\usepackage{extarrows}
\usepackage{amssymb}
\usepackage{graphicx}
\usepackage{enumitem}
\usepackage[affil-sl]{authblk}

\author{Ji Xu}
\author{Daniel Hsu}
\author{Arian Maleki}
\affil{Columbia University}

\newtheorem{theorem}{Theorem}
\newtheorem{lemma}{Lemma}
\newtheorem{corollary}{Corollary}
\newtheorem{remark}{Remark}

\definecolor{purple}{rgb}{0.8,0.1,0.5}
\def\ep{\epsilon}

   \def\de{\delta}

   \def\la{\lambda}
   
   \def\a{\alpha}

   \def\n{\nonumber}
   
   \def\conp{\stackrel{p}{\rightarrow}}
   
   \def\bbE{\mathbb{E}}
   \def\bbR{\mathbb{R}}

   \def\funcP{P}
   \def\funcQ{\Gamma}
   \def\funcp{p}
   \def\funcq{\gamma}
   \def\vfuncq{\boldsymbol{\funcq}}
   \def\vfuncp{\mathsf{\funcp}}

   \newcommand\R{\ensuremath{\mathbb{R}}}
   \newcommand\E{\ensuremath{\mathbb{E}}}
   \newcommand\Y{\ensuremath{\boldsymbol{Y}}}
   \newcommand\y{\ensuremath{\boldsymbol{y}}}
   \newcommand\x{\ensuremath{\boldsymbol{x}}}

   \newcommand\w{\ensuremath{\mathsf{w}}}

   \newcommand\stt{\boldsymbol{\theta}^\star}
   \def\innertm(#1){\boldsymbol{\mu}_{#1}^\star}
   \newcommand\tm[1]{\innertm(#1)}

   \newcommand\tstt{\tilde{\boldsymbol{\theta}^\star}}
   \newcommand\tstti[1]{\tilde{\theta^\star_{#1}}}
   \def\vnu{\boldsymbol{\nu}}
   \def\veta{\boldsymbol{\eta}}

   \newcommand\pt[1]{\boldsymbol{\theta}^{\langle #1 \rangle}}
   \def\innerpm(#1,#2){\boldsymbol{\mu}_{#1}^{\langle #2 \rangle}}
   \renewcommand\pm[1]{\innerpm(#1)}
   \newcommand\pa[1]{\boldsymbol{a}^{\langle #1 \rangle}}
   \newcommand\pb[1]{\boldsymbol{b}^{\langle #1 \rangle}}
   \newcommand\pp[1]{\mathsf{\funcp}^{\langle #1 \rangle}}
   \newcommand\pq[1]{\boldsymbol{\funcq}^{\langle #1 \rangle}}
   \newcommand\tpt[1]{\tilde{\boldsymbol{\theta}}_{#1}}
   \newcommand\tpa[1]{\tilde{\boldsymbol{a}}^{\langle #1 \rangle}}
   \def\innertpai(#1,#2){\tilde{a}_{#1}^{\langle #2 \rangle}}
   \newcommand\tpai[1]{\innertpai(#1)}
   \newcommand\tpb[1]{\tilde{\boldsymbol{b}}^{\langle #1 \rangle}}
   \def\innertpbi(#1,#2){\tilde{b}_{#1}^{\langle #2 \rangle}}
   \newcommand\tpbi[1]{\innertpbi(#1)}
   \newcommand\tpq[1]{\tilde{\boldsymbol{\funcq}}^{\langle #1 \rangle}}
   \def\innertpqi(#1,#2){\tilde{\funcq}_{#1}^{\langle #2 \rangle}}
   \newcommand\tpqi[1]{\innertpqi(#1)}
   \newcommand\ta{\tilde{\boldsymbol{a}}}
   \newcommand\tb{\tilde{\boldsymbol{b}}}
   \newcommand\tq{\tilde{\boldsymbol{\funcq}}}
   \newcommand\bpa[1]{\bar{\boldsymbol{a}}^{\langle #1 \rangle}}
   \newcommand\bpb[1]{\bar{\boldsymbol{b}}^{\langle #1 \rangle}}
   \newcommand\bpp[1]{\bar{\mathsf{\funcp}}^{\langle #1 \rangle}}
   \newcommand\bpq[1]{\bar{\boldsymbol{\funcq}}^{\langle #1 \rangle}}

   \newcommand\tai[1]{\tilde{a}_{#1}}
   \newcommand\tbi[1]{\tilde{b}_{#1}}
   \newcommand\tqi[1]{\tilde{\funcq}_{#1}}
   \def\innertheta(#1,#2){\theta^{\star}_{\langle #2 \rangle,#1}}
   \newcommand\ptt[1]{\innertheta(#1)}
   \def\innerpdf(#1,#2,#3){\phi^{+}_{#1}(#2,#3)}
   \newcommand\pdf[1]{\innerpdf(#1)}
   \def\innerdpdf(#1,#2){\phi^{-}(#1,#2)}
   \newcommand\dpdf[1]{\innerdpdf(#1)}
   \def\innerpj(#1,#2){\Pi^{\langle #2\rangle}_{#1}}
   
   \def\innere(#1,#2){\boldsymbol{e}_{#1}^{\langle #2 \rangle}}
   \newcommand\e[1]{\innere(#1)}
   \newcommand\pbeta[1]{\beta^{\langle #1 \rangle}}
   \newcommand\palpha[1]{\alpha^{\langle #1 \rangle}}
   \newcommand\st[1]{\hat{\boldsymbol{\theta}}^{\langle #1 \rangle}}
   \def\innersm(#1,#2){\hat{\boldsymbol{\mu}}_{#1}^{\langle #2 \rangle}}
   \newcommand\sm[1]{\innersm(#1)}
   \newcommand\sa[1]{\hat{\boldsymbol{a}}^{\langle #1 \rangle}}
   \renewcommand\sb[1]{\hat{\boldsymbol{b}}^{\langle #1 \rangle}}
   \newcommand\sq[1]{\hat{\boldsymbol{q}}^{\langle #1 \rangle}}
   \renewcommand\sp[1]{\hat{\mathsf{p}}^{\langle #1 \rangle}}

\newcommand\sgn{\ensuremath{\operatorname{sgn}}}

\usepackage{amsmath,amsbsy,amsfonts,amssymb,amsthm,dsfont,mleftright,commath}

\def\ddefloop#1{\ifx\ddefloop#1\else\ddef{#1}\expandafter\ddefloop\fi}

\def\ddef#1{\expandafter\def\csname bf#1\endcsname{\ensuremath{\mathbf{#1}}}}
\ddefloop ABCDEFGHIJKLMNOPQRSTUVWXYZabcdefghijklmnopqrstuvwxyz\ddefloop

\def\ddef#1{\expandafter\def\csname bf#1\endcsname{\ensuremath{\pmb{\csname #1\endcsname}}}}
\ddefloop {alpha}{beta}{gamma}{delta}{epsilon}{varepsilon}{zeta}{eta}{theta}{vartheta}{iota}{kappa}{lambda}{mu}{nu}{xi}{pi}{varpi}{rho}{varrho}{sigma}{varsigma}{tau}{upsilon}{phi}{varphi}{chi}{psi}{omega}{Gamma}{Delta}{Theta}{Lambda}{Xi}{Pi}{Sigma}{varSigma}{Upsilon}{Phi}{Psi}{Omega}{ell}\ddefloop

\def\ddef#1{\expandafter\def\csname bb#1\endcsname{\ensuremath{\mathbb{#1}}}}
\ddefloop ABCDEFGHIJKLMNOPQRSTUVWXYZ\ddefloop

\def\ddef#1{\expandafter\def\csname c#1\endcsname{\ensuremath{\mathcal{#1}}}}
\ddefloop ABCDEFGHIJKLMNOPQRSTUVWXYZ\ddefloop

\def\ddef#1{\expandafter\def\csname v#1\endcsname{\ensuremath{\boldsymbol{#1}}}}
\ddefloop ABCDEFGHIJKLMNOPQRSTUVWXYZabcdefghijklmnopqrstuvwxyz\ddefloop

\def\ddef#1{\expandafter\def\csname v#1\endcsname{\ensuremath{\boldsymbol{\csname #1\endcsname}}}}
\ddefloop {alpha}{beta}{gamma}{delta}{epsilon}{varepsilon}{zeta}{eta}{theta}{vartheta}{iota}{kappa}{lambda}{mu}{nu}{xi}{pi}{varpi}{rho}{varrho}{sigma}{varsigma}{tau}{upsilon}{phi}{varphi}{chi}{psi}{omega}{Gamma}{Delta}{Theta}{Lambda}{Xi}{Pi}{Sigma}{varSigma}{Upsilon}{Phi}{Psi}{Omega}{ell}\ddefloop

\renewcommand\t{{\ensuremath{\scriptscriptstyle{\top}}}}

\DeclareMathOperator*{\argmax}{arg\,max}

\renewcommand\v{\ensuremath{\boldsymbol}}

\newcommand\parens[1]{(#1)}
\renewcommand\norm[1]{\|#1\|} 
\newcommand\braces[1]{\{#1\}}

\renewcommand\abs[1]{|#1|} 

\newcommand\dotp[1]{\langle #1 \rangle}

\newcommand\Parens[1]{\mleft(#1\mright)}
\newcommand\Norm[1]{\mleft\|#1\mright\|}

\newcommand\Brackets[1]{\mleft[#1\mright]}

\title{Global Analysis of Expectation Maximization \\ for Mixtures of Two Gaussians}

\begin{document}
\maketitle
{\def\thefootnote{}
\footnotetext{E-mail:
\texttt{jixu@cs.columbia.edu},
\texttt{djhsu@cs.columbia.edu},
\texttt{arian@stat.columbia.edu}}

\begin{abstract}
Expectation Maximization (EM) is among the most popular algorithms for
estimating parameters of statistical models.
However, EM, which is an iterative algorithm based on the maximum likelihood
principle, is generally only guaranteed to find stationary points of the
likelihood objective, and these points may be far from any maximizer.
This article addresses this disconnect between the statistical principles behind
EM and its algorithmic properties.
Specifically, it provides a global analysis of EM for specific models in which
the observations comprise an i.i.d.~sample from a mixture of two Gaussians.
This is achieved by (i) studying the sequence of parameters from idealized
execution of EM in the infinite sample limit, and fully characterizing the limit
points of the sequence in terms of the initial parameters; and then (ii) based
on this convergence analysis, establishing statistical consistency (or lack
thereof) for the actual sequence of parameters produced by EM.
\end{abstract}

\section{Introduction}

Since Fisher's 1922 paper~\citep{fisher1922mathematical}, maximum likelihood estimators (MLE) have
become one of the most popular tools in many areas of science and engineering.
The asymptotic consistency and optimality of MLEs have provided users with the
confidence that, at least in some sense, there is no better way to estimate
parameters for many standard statistical models.
Despite its appealing properties, computing the MLE is often intractable.
Indeed, this is the case for many \emph{latent variable models} $\{
f(\mathcal{Y},\vz;\veta) \}$, where the \emph{latent variables} $\vz$ are not observed.
For each setting of the parameters $\veta$, the marginal distribution of the
observed data $\mathcal{Y}$ is (for discrete $\vz$)
\[
  f(\mathcal{Y}; \veta) \ = \ \sum_{\vz} f(\mathcal{Y},\vz; \veta) \,.
\]
It is this marginalization over latent variables that typically causes the
computational difficulty. Furthermore, many algorithms based on the MLE principle are only known to find stationary
points of the likelihood objective (e.g., local maxima), and these points are
not necessarily the MLE.

\subsection{Expectation Maximization}

Among the algorithms mentioned above,
Expectation Maximization (EM) has attracted more attention for the simplicity of
its iterations, and its good performance in
practice~\citep{dempster1977ml,redner1984mixture}.
EM is an iterative algorithm for climbing the likelihood objective starting from
an initial setting of the parameters $\hat{\veta}^{\langle 0 \rangle}$.
In iteration $t$, EM performs the following steps:%
\begin{align}
  \label{eq:actualEMgeneric}
  \text{E-step:} & &
  \hat{Q}(\veta \mid \hat{\veta}^{\langle t \rangle })
  & \ \triangleq \ \sum_{\vz} f(\vz \mid \mathcal{Y}; \hat{\veta}^{\langle t \rangle}) \log f(\mathcal{Y}, \vz; \veta) \,, \\
  \text{M-step:} & &
  \hat{\veta}^{\langle t+1 \rangle}
  & \ \triangleq \ \argmax_{\veta} \hat{Q}(\veta \mid \hat{\veta}^{\langle t \rangle} ) \,,
\end{align}
In many applications, each step is intuitive and can be performed very
efficiently.

Despite the popularity of EM, as well as the numerous theoretical studies of its
behavior, many important questions about its performance---such as its
convergence rate and accuracy---have remained unanswered.
The goal of this paper is to address these questions for specific models
(described in Section~\ref{sec:maincont}) in which the observation $\mathcal{Y}$
is an i.i.d.~sample from a mixture of two Gaussians.

Towards this goal, we study an idealized execution of EM in the large sample
limit, where the E-step is modified to be computed over an infinitely large
i.i.d.~sample from a Gaussian mixture distribution in the model.
In effect, in the formula for $\hat{Q}(\veta\mid\hat{\veta}^{\langle t \rangle})$, we
replace the observed data $\mathcal{Y}$ with a random variable $\Y \sim
f(\vy;\veta^\star)$ for some Gaussian mixture parameters $\veta^\star$ and then
take its expectation.
The resulting E- and M-steps in iteration $t$ are
\begin{align}
  \label{eq:populEMgeneric}
  \text{E-step:} & &
  Q(\veta \mid \veta^{\langle t \rangle })
  & \ \triangleq \ \E_{\Y}
  \sbr{ \sum_{\vz} f(\vz \mid \vY; \veta^{\langle t \rangle}) \log f(\vY, \vz; \veta) } \,, \\
 \text{M-step:} & &
  \veta^{\langle t+1 \rangle}
  & \ \triangleq \ \argmax_{\veta} Q(\veta \mid \veta^{\langle t \rangle} )
  \,.
\end{align}
This sequence of parameters $\parens{\veta^{\langle t \rangle}}_{t\geq0}$ is fully
determined by the initial setting $\veta^{\langle 0 \rangle}$.
We refer to this idealization as \emph{Population EM}.
Not only does Population EM shed light on the dynamics of EM in the large sample
limit, but it can also reveal some of the fundamental limitations of EM.
Indeed, if Population EM cannot provide an accurate estimate for the parameters
$\veta^\star$, then intuitively, one would not expect the EM algorithm with a
finite sample size to do so either.
(To avoid confusion, we refer the original EM algorithm run with a finite sample
as \emph{Sample-based EM}.)

\subsection{Models and Main Contributions}\label{sec:maincont}
In this paper, we study EM in the context of two simple yet popular and
well-studied Gaussian mixture models.
The two models, along with the corresponding Sample-based EM and Population EM
updates, are as follows:

\paragraph{Model 1.}
The observation $\cY$ is an i.i.d.~sample from the mixture distribution $0.5
N(-\stt, \vSigma) + 0.5 N(\stt, \vSigma)$; $\vSigma$ is a known covariance
matrix in $\R^d$, and $\stt$ is the unknown parameter of interest.
\begin{enumerate}
  \item Sample-based EM iteratively updates its estimate of $\stt$ according to the following equation:
    \begin{eqnarray}\label{eq:emfirstmodelact1sample}
      \st{t+1} & = & \frac{1}{n} \sum_{i=1}^n \del[2]{ 2\w_d\del[1]{\y_i,\st{t}}-1 } \y_i
      ,
    \end{eqnarray}
    where $\y_1,\dotsc,\y_n$ are the independent draws that comprise $\cY$,
    \begin{eqnarray*}
      \w_d(\y,\vtheta)
      & \triangleq &
      \frac{\phi_d(\y - \vtheta)}{\phi_d(\y - \vtheta) + \phi_d(\y + \vtheta)}
      ,
    \end{eqnarray*}
    and $\phi_d$ is the density of a Gaussian random vector with mean $\v0$ and
    covariance $\vSigma$.

  \item Population EM iteratively updates its estimate according to the following equation:
    \begin{eqnarray}\label{eq:emfirstmodelact1population}
      \pt{t+1} & = & \E (2\w_d(\Y,\pt{t})-1) \Y
      ,
    \end{eqnarray}
    where $\Y \sim 0.5 N(-\stt, \vSigma) + 0.5 N(\stt, \vSigma)$.
\end{enumerate}

\paragraph{Model 2.}
The observation $\cY$ is an i.i.d.~sample from the mixture distribution $0.5
N(\tm{1}, \vSigma) + 0.5 N(\tm{2}, \vSigma)$.
Again, $\vSigma$ is known, and $(\tm{1},\tm{2})$ are the unknown parameters of interest.
\begin{enumerate}
  \item Sample-based EM iteratively updates its estimate of $\tm{1}$ and $\tm{2}$ at every iteration according to the following equations:
    \begin{eqnarray}\label{eq:emfirstmodelact2sample}
      \sm{1,t+1} & = & \frac{\sum_{i=1}^n \mathsf{v}_d(\y_i,\sm{1,t},\sm{2,t})\y_i}{\sum_{i=1}^n \mathsf{v}_d(\y_i,\sm{1,t},\sm{2,t})}, \\
      \sm{2,t+1} & = & \frac{\sum_{i=1}^n (1-\mathsf{v}_d(\y_i,\sm{1,t},\sm{2,t}))\y_i}{\sum_{i=1}^n (1-\mathsf{v}_d(\y_i,\sm{1,t},\sm{2,t}))},
    \end{eqnarray}
    where $\y_1,\dotsc,\y_n$ are the independent draws that comprise $\cY$, and
    \begin{eqnarray*}
      \mathsf{v}_d(\y,\vmu_1,\vmu_2)
      & \triangleq &
      \frac{\phi_d(\y - \vmu_1)}{\phi_d(\y - \vmu_1) + \phi_d(\y - \vmu_2)}
      .
    \end{eqnarray*}
  \item Population EM iteratively updates its estimates according to the
    following equations:
    \begin{eqnarray}\label{eq:emfirstmodelact2population}
      \pm{1,t+1} & = & \frac{\E \mathsf{v}_d(\Y,\pm{1,t},\pm{2,t})\Y}{\bbE \mathsf{v}_d(\Y,\pm{1,t},\pm{2,t})}, \\
      \pm{2,t+1} & = & \frac{\E (1-\mathsf{v}_d(\Y,\pm{1,t},\pm{2,t}))\Y}{\bbE (1-\mathsf{v}_d(\Y,\pm{1,t},\pm{2,t}))}
      ,
    \end{eqnarray}
    where $\Y \sim 0.5 N(\tm{1}, \vSigma) + 0.5 N(\tm{2}, \vSigma)$.
\end{enumerate}

Our main contribution in this paper is a new characterization of the stationary
points and dynamics of EM in both of the above models.
\begin{enumerate}
  \item
    We prove convergence for the sequence of iterates for Population EM from
    each model:
    the sequence $(\pt{t})_{t\geq0}$ converges to either $\stt$, $-\stt$, or
    $\v0$;
    the sequence $((\pm{1,t},\pm{2,t}))_{t\geq0}$ converges to either
    $(\tm{1},\tm{2})$, $(\tm{2},\tm{1})$, or
    $((\tm{1}+\tm{2})/2,(\tm{1}+\tm{2})/2)$.
    We also fully characterize the initial parameter settings that lead to each
    limit point.

  \item
    Using this convergence result for Population EM, we also prove that the
    limits of the Sample-based EM iterates converge in probability to the
    unknown parameters of interest, as long as Sample-based EM is initialized at
    points where Population EM would converge to these parameters as well.
\end{enumerate}
Formal statements of our results are given in Section~\ref{sec:mainresult}.

\subsection{Background and Related Work}
\label{sec:related}

The EM algorithm was formally introduced by~\citet{dempster1977ml} as
a general iterative method for computing parameter estimates from
incomplete data.
Although EM is billed as a procedure for maximum likelihood
estimation, it is known that with certain initializations, the final
parameters returned by EM may be far from the MLE, both in
parameter distance and in log-likelihood
value~\citep{wu1983convergence}.
Several works characterize local convergence of EM to stationary
points of the log-likelihood objective under certain regularity
conditions~\citep{wu1983convergence,tseng2004analysis,chretien2008em}.
However, these analyses do not distinguish between global maximizers
and other stationary points (except, e.g., when the likelihood
function is unimodal).
Thus, as an optimization algorithm for maximizing the log-likelihood
objective, the ``worst-case'' performance of EM is somewhat
discouraging.

For a more optimistic perspective on EM, one may consider a
``best-case'' analysis, where (i) the data are an iid sample from a
distribution in the given model, (ii) the sample size is sufficiently
large, and (iii) the starting point for EM is sufficiently close to
the parameters of the data generating distribution.
Conditions (i) and (ii) are ubiquitous in (asymptotic) statistical
analyses, and (iii) is a generous assumption that may be satisfied in
certain cases.
\citet{redner1984mixture} show that in such a favorable scenario, EM
converges to the MLE almost surely for a broad class of mixture
models.
Moreover, recent work of~\citet{balakrishnan2014statistical} gives
non-asymptotic convergence guarantees in certain models; importantly,
these results permit one to quantify the accuracy of a pilot estimator
required to effectively initialize EM.
Thus, EM may be used in a tractable two-stage estimation procedures
given a first-stage pilot estimator that can be efficiently computed.

Indeed, for the special case of Gaussian mixtures, researchers in
theoretical computer science and machine learning have developed
efficient algorithms that deliver the highly accurate parameter
estimates under appropriate conditions.
Several of these algorithms, starting with that of
\citet{dasgupta1999learning}, assume that the means of the mixture
components are \emph{well-separated}---roughly at distance either
$d^\alpha$ or $k^\beta$ for some $\alpha,\beta>0$ for a mixture of $k$ Gaussians
in
$\R^d$~\citep{dasgupta1999learning,arora2005learning,dasgupta2007probabilistic,vempala2004spectral,kannan2008spectral,achlioptas2005spectral,chaudhuri2008learning,brubaker2008isotropic,chaudhuri2009multiview}.
More recent work employs the method-of-moments, which permit the means of the
mixture components to be arbitrarily close, provided that the sample size is
sufficiently
large~\citep{kalai2010efficiently,belkin2010polynomial,moitra2010settling,hsu2013learning,hardt2015tight}.
In particular, \citet{hardt2015tight} characterize the information-theoretic
limits of parameter estimation for mixtures of two Gaussians, and that they are
achieved by a variant of the original method-of-moments of
\citet{pearson1894contributions}.

Most relevant to this paper are works that specifically analyze EM (or variants
thereof) for Gaussian mixture models, especially when the mixture components are
well-separated.
\citet{xu1996convergence} show favorable convergence properties (akin
to super-linear convergence near the MLE) for well-separated mixtures.
In a related but different vein, \citet{dasgupta2007probabilistic}
analyze a variant of EM with a particular initialization scheme, and
proves fast convergence to the true parameters, again for
well-separated mixtures in high-dimensions.
For mixtures of two Gaussians, it is possible to exploit symmetries to
get sharper analyses.
Indeed, \citet{chaudhuri2009learning} uses these symmetries to prove
that a variant of Lloyd's
algorithm~\citep{macqueen1967kmeans,lloyd1982least} (which may be
regarded as a hard-assignment version of EM) very quickly converges to
the subspace spanned by the two mixture component means, without any
separation assumption.
Lastly, for the specific case of our Model 1,
\citet{balakrishnan2014statistical} proves linear convergence of EM (as well as
a gradient-based variant of EM) when started in a sufficiently small
neighborhood around the true parameters; here, the size of the neighborhood
grows with the separation between the two mixture components (which must be
sufficiently large).
Their analysis also proceeds by studying Population EM, and then relating
Sample-based EM to it.
Remarkably, by focusing attention on the local region around the true
parameters, they obtain non-asymptotic bounds on the parameter estimation error.
Our work is complementary to their result in that we focus on asymptotic limits
rather than finite sample analysis.
This allows us to provide a global analysis of EM, without any separation
assumption; such an analysis cannot be deduced from
the results of \citeauthor{balakrishnan2014statistical} by taking limits.

\section{Analysis of EM for Mixtures of Two Gaussians}\label{sec:mainresult}

In this section, we present our results for Population EM and Sample-based EM
under both Model 1 and Model 2, and also discuss further implications about the
expected log-likelihood function.
Without loss of generality, we may assume that the known covariance matrix
$\vSigma$ is the identity matrix $\vI_d$.
Throughout, we denote the Euclidean norm by $\norm{\cdot}$, and the signum
function by $\sgn(\cdot)$ (where $\sgn(0) = 0$, $\sgn(z)=1$ if $z>0$, and
$\sgn(z)=-1$ if $z<0$).

\subsection{Main Results for Population EM}\label{sec:mainresultmainresultpopulation}

We present results for Population EM for both models, starting with Model
1.
\begin{theorem}\label{thm:popEMsymmfp1}
Assume $\stt\in \R^d\setminus\braces{\v0}$.
Let $(\pt{t})_{t\geq0}$ denote the Population EM iterates for Model 1, and
suppose $\langle \pt{0},\stt \rangle \neq 0$.
There exists $\kappa_\theta\in (0,1)$---depending only on $\stt$ and
$\pt{0}$---such that
\begin{eqnarray*}
  \Norm{
    \pt{t+1}-\sgn(\langle \pt{0},\stt \rangle)\stt
  }
  & \leq &
  \kappa_\theta \cdot
  \Norm{
    \pt{t}-\sgn(\langle \pt{0},\stt \rangle)\stt
  }
  \,.
\end{eqnarray*}
\end{theorem}
The proof of Theorem~\ref{thm:popEMsymmfp1}, as well as all other omitted
proofs, is given in Appendix~\ref{sec:proofs}.
Theorem~\ref{thm:popEMsymmfp1} asserts that if $\pt{0}$ is not on the hyperplane
$\braces{ \vx \in \R^d : \dotp{\vx,\stt} = 0 }$, then the sequence
$(\pt{t})_{t\geq0}$ converges to either $\stt$ or $-\stt$.

Our next result shows that if $\dotp{\pt{0},\stt} = 0$, then $(\pt{t})_{t\geq0}$
still converges, albeit to $\v0$.

\begin{theorem}\label{thm:popsymm2zero}
Let $(\pt{t})_{t\geq0}$ denote the Population EM iterates for Model 1.
If $\langle \pt{0},\stt \rangle = 0$, then
\begin{eqnarray*}
  \pt{t} & \to & \v0
  \quad \text{as $t\to\infty$}
  \,.
\end{eqnarray*}
\end{theorem}
Theorems~\ref{thm:popEMsymmfp1} and~\ref{thm:popsymm2zero} together characterize
the fixed points of Population EM for Model 1, and fully specify the conditions
under which each fixed point is reached.
The results are simply summarized in the following corollary.
\begin{corollary}\label{lemma:mainsymmetric}
  If $(\pt{t})_{t\geq0}$ denote the Population EM iterates for Model 1, then
  \begin{eqnarray*}
    \pt{t} & \to & \sgn(\dotp{\pt{0},\stt}) \stt
    \quad \text{as $t\to\infty$}
    \,.
  \end{eqnarray*}
\end{corollary}

We now discuss Population EM with Model 2.
To state our results more concisely, we use the following re-parameterization of
the model parameters and Population EM iterates:
\begin{align}
  \pa{t}
  & \ \triangleq \
  \frac{\pm{1,t}+\pm{2,t}}{2} -\frac{\tm{1}+\tm{2}}{2} \,,
  &
  \pb{t}
  & \ \triangleq \
  \frac{\pm{2,t}-\pm{1,t}}{2} \,,
  &
  \stt
  & \ \triangleq \
  \frac{\tm{2}-\tm{1}}{2}
  \,.\label{equ:reparamization}
\end{align}
If the sequence of Population EM iterates $((\pm{1,t},\pm{2,t}))_{t\geq0}$
converges to $(\tm{1},\tm{2})$, then we expect $\pb{t} \rightarrow \stt$.
Hence, we also define $\pbeta{t}$ as the angle between $\pb{t}$ and $\stt$, i.e.,
\begin{eqnarray*}
  \pbeta{t}
  & \triangleq &
  \arccos\del{
    \frac{\dotp{\pb{t},\stt}}{\norm{\pb{t}}\norm{\stt}}
  }
  \ \in \
  \intcc{0,\pi}
  \,.
\end{eqnarray*}
(This is well-defined as long as $\pb{t} \neq \v0$ and $\stt \neq \v0$.)

We first present results on Population EM with Model 2 under the initial
condition $\dotp{\pb{0},\stt} \neq 0$.

\begin{theorem}\label{thm:convergencerate1}
Assume $\stt\in \R^d\setminus\braces{\v0}$.
Let $(\pa{t},\pb{t})_{t\geq0}$ denote the (re-parameterized) Population EM
iterates for Model 2, and suppose $\dotp{\pb{0},\stt} \neq 0$.
Then $\pb{t} \neq \v0$ for all $t \geq 0$.
Furthermore, there exist $\kappa_a\in (0,1)$---depending only on $\norm{\stt}$
and $\abs{\dotp{\pb{0},\stt}/\norm{\pb{0}}}$---and $\kappa_\beta\in
(0,1)$---depending only on $\norm{\stt}$, $\dotp{\pb{0},\stt} / \norm{\pb{0}}$,
$\norm{\pa{0}}$, and $\norm{\pb{0}}$---such that
\begin{eqnarray*}
  \norm{\pa{t+1}}^2
  & \leq &
  \kappa_a^2 \cdot \norm{\pa{t}}^2
  +
  \frac{\norm{\stt}^2 \sin^2(\pbeta{t})}{4}
  \,,
  \\
  \sin(\pbeta{t+1})
  & \leq &
  \kappa_{\beta}^t \cdot \sin(\pbeta{0})
  \,.
\end{eqnarray*}
\end{theorem}

By combining the two inequalities from Theorem~\ref{thm:convergencerate1},
we conclude
\begin{eqnarray*}
  \norm{\pa{t+1}}^2
  & = &
  \kappa_a^{2t}
  \norm{\pa{0}}^2
  + \frac{\norm{\stt}^2}{4}
  \sum_{\tau=0}^t \kappa_a^{2\tau} \cdot \sin^2(\pbeta{t-\tau})
  \\
  & \leq &
  \kappa_a^{2t}
  \norm{\pa{0}}^2
  +
  \frac{\norm{\stt}^2}{4}
  \sum_{\tau=0}^t \kappa_a^{2\tau} \kappa_\beta^{2(t-\tau)} \cdot \sin^2(\pbeta{0})
  \\
  & \leq &
  \kappa_a^{2t} \norm{\pa{0}}^2 + \frac{\norm{\stt}^2}{4} t
  \del{ \max\cbr{\kappa_a, \kappa_\beta}}^t
  \sin^2(\pbeta{0})
  \,.
\end{eqnarray*}

Theorem~\ref{thm:convergencerate1} shows that the re-parameterized Population EM
iterates converge, at a linear rate, to the average of the two means
$(\tm{1}+\tm{2})/2$, as well as the line spanned by $\stt$.
The theorem, however, does not provide any information on the convergence of the
magnitude of $\pb{t}$ to the magnitude of $\stt$.
This is given in the next theorem.

\begin{theorem}\label{thm:convergenceratemagnitude}
Assume $\stt\in \R^d\setminus\braces{\v0}$.
Let $(\pa{t},\pb{t})_{t\geq0}$ denote the (re-parameterized) Population EM
iterates for Model 2, and suppose $\dotp{\pb{0},\stt} \neq 0$.
Then there exist $T_0 > 0$, $\kappa_b \in (0,1)$, and $c_b>0$---all depending
only on $\norm{\stt}$, $\abs{\dotp{\pb{0},\stt}/\norm{\pb{0}}}$,
$\norm{\pa{0}}$, and $\norm{\pb{0}}$---such that
\begin{eqnarray*}
  \Norm{\pb{t+1}-\sgn(\dotp{\pb{0},\stt})\stt}^2
  &\leq&
  \kappa_b^2 \cdot \Norm{\pb{t}-\sgn(\dotp{\pb{0},\stt})\stt}^2
  +c_b \cdot \norm{\pa{t}}
  \quad \forall t > T_0
  \,.
\end{eqnarray*}
\end{theorem}

If $\dotp{\pb{0},\stt} = 0$, then we show convergence of the (re-parameterized) Population EM
iterates to the degenerate solution $(\v0,\v0)$.

\begin{theorem}\label{thm:hyperplaneinit1}
Let $(\pa{t},\pb{t})_{t\geq0}$ denote the (re-parameterized) Population EM
iterates for Model 2.
If $\dotp{\pb{0},\stt} = 0$, then
\begin{eqnarray*}
  (\pa{t}, \pb{t})
  & \to & (\v0,\v0)
  \quad \text{as $t\to\infty$}
  \,.
\end{eqnarray*}
\end{theorem}

Theorems~\ref{thm:convergencerate1},~\ref{thm:convergenceratemagnitude},
and~\ref{thm:hyperplaneinit1} together characterize the fixed points of
Population EM for Model 2, and fully specify the conditions under which each
fixed point is reached.
The results are simply summarized in the following corollary.
\begin{corollary}\label{lemma:main}
  If $(\pa{t},\pb{t})_{t\geq0}$ denote the (re-parameterized) Population EM
  iterates for Model 2, then
  \begin{eqnarray*}
    \pa{t} & \to & \frac{\tm{1} + \tm{2}}{2}
    \quad \text{as $t\to\infty$}
    \,,
    \label{equ:a}
    \\
    \pb{t} & \to &
    \sgn\parens{\dotp{\pb{0},\tm{2}-\tm{1}}} \frac{\tm{2}-\tm{1}}{2}
    \quad \text{as $t\to\infty$}
    \,.
    \label{equ:b}
  \end{eqnarray*}
\end{corollary}

\subsection{Main Results for Sample-based EM}\label{sec:mainresultmainresultsamplebase}

Using the results on Population EM presented in the above section, we can now establish consistency of
(Sample-based) EM.
We focus attention on Model 2, as the same results for Model 1 easily follow as
a corollary.
First, we state a simple connection between the Population EM and Sample-based
EM iterates.
\begin{theorem}\label{lemma:pconv}
  Suppose Population EM and Sample-based EM for Model 2 have the same initial
  parameters:
  $\sm{1,0}=\pm{1,0}$ and $\sm{2,0}=\pm{2,0}$.
  Then for each iteration $t \geq 0$,
  \begin{equation*}
    \sm{1,t} \ \to \ \pm{1,t}
    \quad\text{and}\quad
    \sm{2,t} \ \to \ \pm{2,t}
    \quad \text{as $n\to\infty$}
    \,,
  \end{equation*}
  where convergence is in probability.
\end{theorem}

Note that Theorem~\ref{lemma:pconv} does not necessarily imply that the fixed
point of Sample-based EM (when initialized at
$(\sm{1,0},\sm{2,0})=(\pm{1,0},\pm{2,0})$) is the same as that of Population EM.
It is conceivable that as $t \to \infty$, the discrepancy between (the iterates
of) Sample-based EM and Population EM increases.
We show that this is not the case:
the fixed points of Sample-based EM indeed converge to the fixed points of
Population EM.

\begin{theorem}\label{lemma:empconverge2}
  Suppose Population EM and Sample-based EM for Model 2 have the same initial
  parameters:
  $\sm{1,0}=\pm{1,0}$ and $\sm{2,0}=\pm{2,0}$.
  If $\dotp{\pm{2,0}-\pm{1,0},\stt} \neq 0$, then
  \begin{equation*}
    \limsup_{t\to\infty}
    |\sm{1,t}-\pm{1,t}|
    \ \to \
    0
    \quad\text{and}\quad
    \limsup_{t\to\infty}
    |\sm{2,t}-\pm{2,t}|
    \ \to \
    0
    \quad \text{as $n\to\infty$}
    \,,
  \end{equation*}
  where convergence is in probability.
\end{theorem}

\subsection{Population EM and Expected Log-likelihood}
Do the results we derived in the last section regarding the performance of EM provide any information on the performance of other ascent algorithms, such as gradient ascent, that aim to maximize the log-likelihood function? To address this question, we show how our analysis can determine the stationary points of the expected log-likelihood and characterize the shape of the expected log-likelihood in a neighborhood of the stationary points.  Let $G(\veta)$ denote the expected log-likelihood, i.e.,
\[
G(\veta) \triangleq \bbE (\log f_{\veta}(\Y)) = \int f(\y; \veta^*)\log f(\y;
\veta) \dif\y,
\]
where $\veta^*$ denotes the true parameter value. Also consider the following standard regularity conditions:

\begin{description}
\item[R1] The family of probability density functions $f(\y; \veta)$ have common support.
\item[R2] $\nabla_{\veta}  \int f(\y; \veta^*)\log f(\y; \veta) \dif\y =  \int
  f(\y; \veta^*)\nabla_{\veta} \log f(\y; \veta) \dif\y$, where $\nabla_{\veta}$ denotes the gradient with respect to $\veta$.
\item[R3] $\nabla_{\veta} (\mathbb{E} \sum_{\vz}  f(\vz \mid  \Y; \veta^{\langle
  t\rangle}))\log f(\Y, \vz; \eta) = \mathbb{E} \sum_{\vz} f(\vz \mid  \Y;
  \veta^{\langle t\rangle})\nabla_{\veta} \log f(\Y, \vz; \veta) $.
\end{description}
These conditions can be easily confirmed for many models including the Gaussian mixture models. The following theorem connects the fixed points of the Population EM and the stationary points of the expected log-likelihood.

\begin{lemma}\label{lem:stationarypoints}
 Let $\bar{\veta} \in \mathbb{R}^d$ denote a stationary point of $G(\veta)$.
 Also assume that $Q(\veta \mid \veta^{\dotp{t}})$ has a unique and finite stationary point
 in terms of $\veta$ for every $\veta^{\dotp{t}}$, and this stationary point is
 its global maxima. Then, if the model satisfies conditions
 R1--R3, and the Population EM algorithm is initialized at $\bar{\veta}$, it
 will stay at $\bar{\veta}$. Conversely, any fixed point of Population EM is a stationary point of $G(\veta)$.
\end{lemma}
\begin{proof}
Let $\bar{\veta}$ denote a stationary point of $G(\veta)$. We first prove that
$\bar{\veta}$ is a stationary point of $Q(\veta \mid \bar{\veta})$.
\begin{eqnarray*}
\left. \nabla_{\veta} Q(\veta \mid \bar{\veta}) \right|_{\veta = \bar{\veta}}
&=& \int \sum_{\vz} f(\vz \mid \y; \bar{\veta}) \frac{ \left.  \nabla_{\veta}
f(\y, \vz; \veta) \right|_{\veta = \bar{\veta}}}{f(\y, \vz; \bar{\veta})}  f(\y; \veta^*)\dif\y \\
&=&\int \sum_{\vz}  \frac{ \left. \nabla_{\veta} f(\y, \vz; \veta) \right|_{\veta =
\bar{\veta}}}{f(\y; \bar{\veta})}  f(\y; \veta^*)\dif\y \\
&=& \int  \frac{ \left. \nabla_{\veta} f(\y, \veta) \right|_{\veta =
\bar{\veta}}}{f(\y; \bar{\veta})} f(\y; \veta^*)\dif\y =\v0 \,,
\end{eqnarray*}
where the last equality is using the fact that $\bar{\veta}$ is a stationary
point of $G(\veta)$. Since $Q(\veta \mid \bar{\veta})$ has a unique stationary point, and we have assumed that the unique stationary point is its global maxima, then Population EM will stay at that point.  The proof of the other direction is similar.
\end{proof}
\begin{remark}
The fact that $\veta^*$ is the global maximizer of $G(\veta)$ is well-known in
the
statistics and machine learning literature~\citep[e.g.,][]{conniffe1987expected}. Furthermore, the fact that
$\veta^*$ is a global maximizer of $Q(\veta \mid \veta^*)$ is known as the self-consistency property~\citep{balakrishnan2014statistical}.
\end{remark}

It is straightforward to confirm the conditions of Lemma \ref{lem:stationarypoints} for mixtures of Gaussians. This lemma confirms that Population EM may be trapped in every local maxima. However, less intuitively it may get stuck at local minima or saddle points as well.
Our next result characterizes the stationary points of $G(\boldsymbol{\theta})$ for Model 1.
\begin{corollary}
  $G(\vtheta)$ has only three stationary points. If $d=1$ (so $\vtheta = \theta
  \in \R$), then $0$ is a local
minima of $G(\theta)$, while $\theta^*$ and $-\theta^*$ are global maxima. If
$d>1$, then $\v0$ is a saddle point, and $\stt$ and $-\stt$ are global maxima.
\end{corollary}
The proof is a straightforward result of Lemma \ref{lem:stationarypoints} and
Corollary \ref{lemma:mainsymmetric}. The phenomenon that Population EM may stuck in local minima or saddle points also happens in Model 2.
We can employ Corollary~\ref{lemma:main} and Lemma \ref{lem:stationarypoints} to
explain the shape of the expected log-likelihood function $G$. To simplify the
notation, we consider the re-parametrization $\va \triangleq
\frac{\vmu_1+\vmu_2}{2}$ and $\vb \triangleq \frac{\vmu_2-\vmu_1}{2}$.

\begin{corollary}
$G(\va,\vb)$ has three stationary points:
\begin{equation*}
  \left( \frac{\tm{1}+\tm{2}}{2}, \frac{\tm{2}- \tm{1}}{2} \right)
  \,,
  \qquad
  \left( \frac{\tm{1}+\tm{2}}{2}, \frac{\tm{1}-\tm{2}}{2} \right)
  \,,
  \qquad
  \text{and}
  \qquad
  \left( \frac{\tm{1}+\tm{2}}{2}, \frac{\tm{1}+ \tm{2}}{2} \right)
  \,.
\end{equation*}
The first two points are global maxima.
The third point is a saddle point.
\end{corollary}

\section{Concluding Remarks}

Our analysis of Population EM and Sample-based EM shows that the EM algorithm
can, at least for the Gaussian mixture models studied in this work, compute
statistically consistent parameter estimates.
Previous analyses of EM only established such results for specific methods of
initializing EM
\citep[e.g.,][]{dasgupta2007probabilistic,balakrishnan2014statistical}; our
results show that they are not really necessary in the large sample limit.
However, in any real scenario, the large sample limit may not accurately
characterize the behavior of EM.
Therefore, these specific methods for initialization, as well as non-asymptotic
analysis, are clearly still needed to understand and effectively apply EM.

There are several interesting directions concerning EM that we hope to pursue in
follow-up work.
The first considers the behavior of EM when the dimension $d = d_n$ may grow
with the sample size $n$.
Our proof of Theorem~\ref{lemma:empconverge2} reveals that the parameter
error of the $t$-th iterate (in Euclidean norm) is of the order $\sqrt{d/n}$ as
$t\to\infty$.
Therefore, we conjecture that the theorem still holds as long as $d_n = o(n)$.
This would be consistent with results from statistical physics on the MLE for
Gaussian mixtures, which characterize the behavior when $d_n \propto n$ as
$n\to\infty$~\citep{barkai1994statistical}.

Another natural direction is to extend these results to more general Gaussian
mixture models (e.g., with unequal mixing weights or unequal covariances) and
other latent variable models.

\paragraph{Acknowledgements.}
The second named author thanks Yash Deshpande and Sham Kakade for many helpful
initial discussions.
JX and AM were partially supported by NSF grant CCF-1420328.
DH was partially supported by NSF grant DMREF-1534910 and a Sloan Fellowship.

\bibliography{paper}

\begin{thebibliography}{29}
\providecommand{\natexlab}[1]{#1}
\providecommand{\url}[1]{\texttt{#1}}
\expandafter\ifx\csname urlstyle\endcsname\relax
  \providecommand{\doi}[1]{doi: #1}\else
  \providecommand{\doi}{doi: \begingroup \urlstyle{rm}\Url}\fi

\bibitem[Achlioptas and McSherry(2005)]{achlioptas2005spectral}
D.~Achlioptas and F.~McSherry.
\newblock On spectral learning of mixtures of distributions.
\newblock In \emph{Eighteenth Annual Conference on Learning Theory}, pages
  458--469, 2005.

\bibitem[Arora and Kannan(2005)]{arora2005learning}
S.~Arora and R.~Kannan.
\newblock Learning mixtures of separated nonspherical {G}aussians.
\newblock \emph{The Annals of Applied Probability}, 15\penalty0 (1A):\penalty0
  69--92, 2005.

\bibitem[{Balakrishnan} et~al.(2014){Balakrishnan}, {Wainwright}, and
  {Yu}]{balakrishnan2014statistical}
S.~{Balakrishnan}, M.~J. {Wainwright}, and B.~{Yu}.
\newblock {Statistical guarantees for the EM algorithm: From population to
  sample-based analysis}.
\newblock \emph{ArXiv e-prints}, August 2014.

\bibitem[Barkai and Sompolinsky(1994)]{barkai1994statistical}
N.~Barkai and H.~Sompolinsky.
\newblock Statistical mechanics of the maximum-likelihood density estimation.
\newblock \emph{Physical Review E}, 50\penalty0 (3):\penalty0 1766--1769, Sep
  1994.

\bibitem[Belkin and Sinha(2010)]{belkin2010polynomial}
M.~Belkin and K.~Sinha.
\newblock Polynomial learning of distribution families.
\newblock In \emph{Fifty-First Annual IEEE Symposium on Foundations of Computer
  Science}, pages 103--112, 2010.

\bibitem[Brubaker and Vempala(2008)]{brubaker2008isotropic}
S.~C. Brubaker and S.~Vempala.
\newblock Isotropic {PCA} and affine-invariant clustering.
\newblock In \emph{Forty-Ninth Annual IEEE Symposium on Foundations of Computer
  Science}, 2008.

\bibitem[Chaudhuri and Rao(2008)]{chaudhuri2008learning}
K.~Chaudhuri and S.~Rao.
\newblock Learning mixtures of product distributions using correlations and
  independence.
\newblock In \emph{Twenty-First Annual Conference on Learning Theory}, pages
  9--20, 2008.

\bibitem[Chaudhuri et~al.(2009{\natexlab{a}})Chaudhuri, Kakade, Livescu, and
  Sridharan]{chaudhuri2009multiview}
K.~Chaudhuri, S.~M. Kakade, K.~Livescu, and K.~Sridharan.
\newblock Multi-view clustering via canonical correlation analysis.
\newblock In \emph{ICML}, 2009{\natexlab{a}}.

\bibitem[Chaudhuri et~al.(2009{\natexlab{b}})Chaudhuri, Dasgupta, and
  Vattani]{chaudhuri2009learning}
K.~Chaudhuri, S.~Dasgupta, and A.~Vattani.
\newblock Learning mixtures of gaussians using the k-means algorithm.
\newblock \emph{CoRR}, abs/0912.0086, 2009{\natexlab{b}}.

\bibitem[Chr{\'e}tien and Hero(2008)]{chretien2008em}
S.~Chr{\'e}tien and A.~O. Hero.
\newblock On {EM} algorithms and their proximal generalizations.
\newblock \emph{ESAIM: Probability and Statistics}, 12:\penalty0 308--326, May
  2008.

\bibitem[Conniffe(1987)]{conniffe1987expected}
D.~Conniffe.
\newblock Expected maximum log likelihood estimation.
\newblock \emph{Journal of the Royal Statistical Society. Series D},
  36\penalty0 (4):\penalty0 317--329, 1987.

\bibitem[Dasgupta(1999)]{dasgupta1999learning}
S.~Dasgupta.
\newblock Learning mixutres of {G}aussians.
\newblock In \emph{Fortieth Annual IEEE Symposium on Foundations of Computer
  Science}, pages 634--644, 1999.

\bibitem[Dasgupta and Schulman(2007)]{dasgupta2007probabilistic}
S.~Dasgupta and L.~Schulman.
\newblock A probabilistic analysis of {EM} for mixtures of separated, spherical
  {G}aussians.
\newblock \emph{Journal of Machine Learning Research}, 8\penalty0
  (Feb):\penalty0 203--226, 2007.

\bibitem[Dempster et~al.(1977)Dempster, Laird, and Rubin]{dempster1977ml}
A.~P. Dempster, N.~M. Laird, and D.~B. Rubin.
\newblock Maximum-likelihood from incomplete data via the {EM} algorithm.
\newblock \emph{J. Royal Statist. Soc. Ser. B}, 39:\penalty0 1--38, 1977.

\bibitem[Fisher(1922)]{fisher1922mathematical}
R.~A. Fisher.
\newblock On the mathematical foundations of theoretical statistics.
\newblock \emph{Philosophical Transactions of the Royal Society, London, A.},
  222:\penalty0 309--368, 1922.

\bibitem[Hardt and Price(2015)]{hardt2015tight}
M.~Hardt and E.~Price.
\newblock Tight bounds for learning a mixture of two gaussians.
\newblock In \emph{Proceedings of the Forty-Seventh Annual ACM on Symposium on
  Theory of Computing}, pages 753--760, 2015.

\bibitem[Hsu and Kakade(2013)]{hsu2013learning}
D.~Hsu and S.~M. Kakade.
\newblock Learning mixtures of spherical {G}aussians: moment methods and
  spectral decompositions.
\newblock In \emph{Fourth Innovations in Theoretical Computer Science}, 2013.

\bibitem[Kalai et~al.(2010)Kalai, Moitra, and Valiant]{kalai2010efficiently}
A.~T. Kalai, A.~Moitra, and G.~Valiant.
\newblock Efficiently learning mixtures of two {G}aussians.
\newblock In \emph{Forty-second ACM Symposium on Theory of Computing}, pages
  553--562, 2010.

\bibitem[Kannan et~al.(2008)Kannan, Salmasian, and Vempala]{kannan2008spectral}
R.~Kannan, H.~Salmasian, and S.~Vempala.
\newblock The spectral method for general mixture models.
\newblock \emph{SIAM Journal on Computing}, 38\penalty0 (3):\penalty0
  1141--1156, 2008.

\bibitem[Koltchinskii(2011)]{VKoltchinskii2011tailbound}
V.~Koltchinskii.
\newblock Oracle inequalities in empirical risk minimization and sparse
  recovery problems.
\newblock In \emph{\'{E}cole d$'$\'{e}t\'{e} de probabilit\'{e}s de Saint-Flour
  XXXVIII}, 2011.

\bibitem[Lloyd(1982)]{lloyd1982least}
S.~P. Lloyd.
\newblock Least squares quantization in {PCM}.
\newblock \emph{IEEE Trans. Information Theory}, 28\penalty0 (2):\penalty0
  129--137, 1982.

\bibitem[MacQueen(1967)]{macqueen1967kmeans}
J.~B. MacQueen.
\newblock Some methods for classification and analysis of multivariate
  observations.
\newblock In \emph{Proceedings of the fifth Berkeley Symposium on Mathematical
  Statistics and Probability}, volume~1, pages 281--297. University of
  California Press, 1967.

\bibitem[Moitra and Valiant(2010)]{moitra2010settling}
A.~Moitra and G.~Valiant.
\newblock Settling the polynomial learnability of mixtures of {G}aussians.
\newblock In \emph{Fifty-First Annual IEEE Symposium on Foundations of Computer
  Science}, pages 93--102, 2010.

\bibitem[Pearson(1894)]{pearson1894contributions}
K.~Pearson.
\newblock Contributions to the mathematical theory of evolution.
\newblock \emph{Philosophical Transactions of the Royal Society, London, A.},
  185:\penalty0 71--110, 1894.

\bibitem[Redner and Walker(1984)]{redner1984mixture}
R.~A. Redner and H.~F. Walker.
\newblock Mixture densities, maximum likelihood and the {EM} algorithm.
\newblock \emph{SIAM Review}, 26\penalty0 (2):\penalty0 195--239, 1984.

\bibitem[Tseng(2004)]{tseng2004analysis}
P.~Tseng.
\newblock An analysis of the {EM} algorithm and entropy-like proximal point
  methods.
\newblock \emph{Mathematics of Operations Research}, 29\penalty0 (1):\penalty0
  27--44, Feb 2004.

\bibitem[Vempala and Wang(2004)]{vempala2004spectral}
S.~Vempala and G.~Wang.
\newblock A spectral algorithm for learning mixtures models.
\newblock \emph{Journal of Computer and System Sciences}, 68\penalty0
  (4):\penalty0 841--860, 2004.

\bibitem[Wu(1983)]{wu1983convergence}
C.~F.~J. Wu.
\newblock On the convergence properties of the {EM} algorithm.
\newblock \emph{The Annals of Statistics}, 11\penalty0 (1):\penalty0 95--103,
  Mar 1983.

\bibitem[Xu and Jordan(1996)]{xu1996convergence}
L.~Xu and M.~I. Jordan.
\newblock On convergence properties of the {EM} algorithm for {Gaussian}
  mixtures.
\newblock \emph{Neural Computation}, 8:\penalty0 129--151, 1996.

\end{thebibliography}
\bibliographystyle{plainnat}

\appendix

\section{Proofs of the Main Results}
\label{sec:proofs}

\subsection{Organization}
This appendix is devoted to the proofs of our main results, and is organized as
follows.
\begin{itemize}
\item Section \ref{sec:connection1} establishes a connection between Model 1 and
  Model 2 for Population EM.
    This connection enables us to use the analysis of Population EM for Model 2 for the analysis of Population EM for Model 1.
\item Section \ref{sec:structuralprop} presents several structural properties of Population EM. These properties will be used in the proofs of our main results.

\item Section \ref{sec:notation} introduces several notations that will be used in the proofs of our main results.

\item Section \ref{sec:proof:thm:convergencerate1} presents the proof of Theorem \ref{thm:convergencerate1}.

\item Section \ref{sec:prooftheomeconvergenceratemag} presents the proof of Theorem \ref{thm:convergenceratemagnitude}.

\item Section \ref{sec:prooftheorem1symmetric} presents the proof of Theorem \ref{thm:popEMsymmfp1}.

\item Section \ref{sec:proofthm5} presents the proof of
  Theorem~\ref{thm:hyperplaneinit1}, which also implies
    Theorem~\ref{thm:popsymm2zero}.

\item Section \ref{sec:proofsampleEM-populationfiniett} presents the proof of Theorem \ref{lemma:pconv}.

\item Section \ref{sec:proofofempconverge2} presents the proof of Theorem \ref{lemma:empconverge2}.

\item Appendix \ref{sec:auxiliary} includes a few auxiliary results that are used in the proofs of our main results.

\end{itemize}

\subsection{Connection Between Models 1 and 2}\label{sec:connection1}

In this section, we draw a connection between Model 1 and Model 2. This will
enable us to conclude most of the results for Model 1 from the results we prove
for Model 2.
First, consider the re-parametrization introduced in \eqref{equ:reparamization}.
The iterations of Population EM can be written in terms of these new parameters
$\pa{t}, \pb{t}$ as
\begin{eqnarray}
\pa{t+1}&=&\frac{\pq{t+1}(1-2\pp{t+1})}{2\pp{t+1}(1-\pp{t+1})}, \label{equ:itea}\\
\pb{t+1}&=&\frac{\pq{t+1}}{2\pp{t+1}(1-\pp{t+1})}, \label{equ:iteb}
\end{eqnarray}
where
\begin{eqnarray}
\pq{t+1}&=&\bbE \w_d(\Y-\pa{t}, \pb{t})\Y\n\\
&=&\int \w_d(\y-\pa{t}, \pb{t})\y\pdf{d, \y, \stt}\dif \y
  \, , \label{equ:iteq}\\
\pp{t+1}&=&\bbE \w_d(\Y-\pa{t}, \pb{t})\n\\
&=&\int \w_d(\y-\pa{t}, \pb{t})\pdf{d, \y, \stt}\dif \y
  \, .
  \label{equ:itep}
\end{eqnarray}
Above, we use
\[
  \pdf{d, \y, \stt}
  \ \triangleq \
  \frac{1}{2}(\phi_d(\y-\stt)+\phi_d(\y+\stt))
\]
as shorthand for the Gaussian mixture density $0.5 N(-\stt, I_d)+ 0.5
N(\stt, I_d)$.

The following lemma establishes a connection between the iterations of
Population EM for Model 1 and for Model 2.
\begin{lemma}\label{lem:connectionpopEMModel1and2}
  If $\pa{0}=\v0$, then $\pa{t}=\v0$ for every $t$.
  Furthermore,
\begin{eqnarray*}
  \pb{t+1}
  \ = \
  2 \bbE \w_d(\Y, \pb{t})Y
  \ = \
  \bbE (2\w_d(\Y, \pb{t})-1)\Y
  \, .
\end{eqnarray*}
\end{lemma}
Observe that the expression for $\pb{t+1}$ in
Lemma~\ref{lem:connectionpopEMModel1and2} is the same as the Population EM
update under Model 1, given in \eqref{eq:emfirstmodelact1population}.

Lemma~\ref{lem:connectionpopEMModel1and2} tells us that Model 1 is a special
case of Model 2 if we know the mean $\parens{\tm{1}+\tm{2}}/2$ is known.
In this case, $\pb{t}$ is regarded as an estimate of $\parens{\tm{2}-\tm{1}}/2$,
in the same way that $\pt{t}$ is an estimate of $\stt$ in Model 1.
(This explains our choice of the notation $\stt \triangleq
\parens{\tm{2}-\tm{1}}/2$ in \eqref{equ:reparamization}.)

The proof of Lemma~\ref{lem:connectionpopEMModel1and2} is a simple induction
that exploits the fact that $\w_d(\Y, \pb{t})+\w_d(-\Y, \pb{t})=1$:
if $\pa{t}=\v0$, then
\begin{eqnarray*}
  \pp{t+1}
  \ = \
  \bbE \w_d(\Y, \pb{t})
  \ = \
  \frac{1}{2}(\bbE \w_d(\Y, \pb{t})+\bbE \w_d(-\Y, \pb{t}))
  \ = \ \frac{1}{2}
  \, .
\end{eqnarray*}

\subsection{Some Structural Properties of Population EM}\label{sec:structuralprop}

An important structural property of Population EM is that the updates are
orthogonally invariant.
This means that our analysis of Population EM can make use of any orthogonal
basis as the coordinate system without affecting the conclusions.
This is spelled out in the following lemma.

\begin{lemma}\label{lemma:rotate}
Let $\vA \in \mathbb{R}^{d \times d}$ denote an orthogonal matrix.
Define, $\tpa{t} \triangleq \vA \pa{t}$, $\tpb{t} \triangleq \vA\pb{t}$, $\tstt
\triangleq \vA \stt$, and $\tpq{t} = \vA \pq{t}$.
Then
\begin{eqnarray}
\tpq{t+1}&=& \int \w_d(\y-\tpa{t}, \tpb{t}) \y\pdf{d, \y, \tstt}\dif\y,
\n\\
\pp{t+1}&=& \int \w_d(\y-\tpa{t}, \tpb{t})\pdf{d, \y, \tstt}\dif\y.\n
\end{eqnarray}
and
\begin{eqnarray}\label{eq:tildeiterationab}
\tpa{t+1}&=&\frac{\tpq{t+1}(1-2\pp{t+1})}{2\pp{t+1}(1-\pp{t+1})}, \n \\
\tpb{t+1}&=&\frac{\tpq{t+1}}{2\pp{t+1}(1-\pp{t+1})}.
\end{eqnarray}
\end{lemma}
\begin{proof}
The proof is a simple change of integration variables from $\y$ to $\tilde{\y} =
\vA^{-1/2} \y$.
\end{proof}

Using the Lemma~\ref{lemma:rotate}, we can establish some simple invariances about Population EM,
which we state in the following lemmas.
\begin{lemma}\label{lemma:sign}
  For all $t\geq 0$,
  \begin{align*}
    \sgn\Parens{\dotp{\pa{t}, \pb{t}}}
    & \ = \
    \sgn\Parens{\dotp{\pa{t+1}, \pb{t+1}}}
    \ = \
    \sgn\Parens{\frac{1}{2}-\pp{t+1}}
    \, ,
    \\
    \sgn\Parens{\dotp{\pb{t}, \stt}}
    & \ = \
    \sgn\Parens{\dotp{\pb{t+1}, \stt}}
    \, .
\end{align*}
\end{lemma}

\begin{lemma}\label{lemma:symmetrizeinitial}
  The following holds for any two settings of $(\pa{0}_{(1)}, \pb{0}_{(1)})$ and
  $(\pa{0}_{(2)}, \pb{0}_{(2)})$.
  \begin{enumerate}
    \item
      If $\pa{0}_{(1)}=-\pa{0}_{(2)}$ and $\pb{0}_{(1)}=\pb{0}_{(2)}$, then
      \[
        \pa{t}_{(1)}=-\pa{t}_{(2)}
        \quad\text{and}\quad
        \pb{t}_{(1)}=\pb{t}_{(2)}
        \, ,
        \qquad
        \forall t\geq 0
        \, .
      \]
    \item
      If $\pb{0}_{(1)}=-\pb{0}_{(2)}$ and $\pa{0}_{(1)}=\pa{0}_{(2)}$, then
      \[
        \pa{t}_{(1)}=\pa{t}_{(2)}
        \quad\text{and}\quad
        \pb{t}_{(1)}=-\pb{t}_{(2)}
        \, ,
        \qquad
        \forall t\geq 0
        \, .
      \]
  \end{enumerate}
\end{lemma}
The proof of Lemma~\ref{lemma:sign} is in Appendix~\ref{sec:proof:lemma:sign}, and
the proof of Lemma~\ref{lemma:symmetrizeinitial} is in Appendix~\ref{sec:proof:lemma:symmetrizeinitial}.
These two lemmas imply that in our analysis of Population EM, we may assume without loss of
generality that
\[
  \dotp{\pa{0}, \pb{0}}
  \ \geq \
  0
  \quad\text{and}\quad
  \dotp{\pb{0}, \stt}
  \ \geq \
  0
  \, .
\]

Next, we show that effectively all of the action of Population EM takes place in a two-dimensional
subspace.

\begin{lemma}
  \label{lemma:general-projection}
  Consider any matrix $\vU \in \R^{d \times k}$ such that $\vU^\t\vU = \vI_k$ and $\stt$ and $\vd$
  are in the range of $\vU$.
  For any matrix $\vV \in \R^{d \times l}$ such that $\vV^\t\vU = \v0$, and any vector $\vc \in
  \R^d$, we have
  \[
    \vV^\t \E\Brackets{ \w_d(\Y-\vc, \vd)\Y } \ = \ \v0
  \]
  where $\Y \sim 0.5 N(-\stt, \vI_d) + 0.5 N(\stt, \vI_d)$.
\end{lemma}
\begin{proof}
  It is easy to see that because $\stt$ is in the range of $\vU$, we have that $\vU^\t\Y$ is
  independent of $\vV^\t\Y$.
  Moreover, because $\vd$ is in the range of $\vU$, we have $\vU\vU^\t\vd = \vd$.
  Thus, for any $\y \in \R^d$,
  \[
    \w_d(\y-\vc, \vd)
    \ = \
    \w_k(\vU^\t(\y-\vc), \vU^\t\vd)
    \, .
  \]
  This implies
  \begin{align*}
    \vV^\t \E \w_d(\Y-\vc, \vd)\Y
    & \ = \
    \E \w_d(\vU^\t(\Y-\vc), \vU^\t\vd)\vV^\t \Y
    \\
    & \ = \
    \E\Brackets{\w_d(\vU^\t(\Y-\vc), \vU^\t\vd)}
    \cdot \E\Brackets{\vV^\t \Y}
    \\
    & \ = \
    \v0
  \end{align*}
  where the last equality follows because $\Y$ has mean zero.
\end{proof}

\begin{lemma}\label{lemma:projection}
Let $M_0$ denote the span of $\pb{0} \neq \v0$ and $\stt \neq \v0$ for the model
$\Y \sim 0.5N(-\stt, \vI_d)+0.5N(\stt, \vI_d)$.
Then $\pb{t} \in M_0$ for all $t\geq 0$.
\end{lemma}
\begin{proof}
  This follows from Lemma~\ref{lemma:general-projection} and induction, by letting the columns of
  $\vU$ be an orthonormal basis for $M_0$, and letting the columns of $\vV$ to be a basis for the
  orthogonal complement of $M_0$.
\end{proof}
Recall that in Section~\ref{sec:mainresultmainresultpopulation}, we defined the angle between
$\pb{t}$ and $\stt$ as $\pbeta{t}$.
For our analysis, it will turn out to be useful to consistently refer to the cosine and sine of this
angle, and hence we should regard the angle as possibly being any value between $0$ and $2\pi$.
To do this, we fix an orthogonal basis $\{\e{1, t}, \cdots, \e{d, t}\}$ such that
\begin{equation}
  \dotp{\pb{t}, \e{1, t}}
  \ = \
  \|\pb{t}\|
  \quad\text{and}\quad
  \stt
  \ = \
  \dotp{\stt, \e{1, t}}\e{1, t}
  +
  \dotp{\stt, \e{2, t}}
  \e{2, t}
  \, ,
  \label{eq:rotationcondition}
\end{equation}
and define $\pbeta{t} \in [0, 2\pi]$ be the angle such that
\begin{eqnarray}
  \cos \pbeta{t}&=& \frac{\langle \pb{t}, \stt \rangle}{\|\pb{t}\|_2 \|\stt\|_2}
  \, ,
  \\
  \sin \pbeta{t}&=& \frac{\dotp{\stt, \e{2, t}}}{\|\stt\|}
  \, .
\end{eqnarray}
Using this definition of the angle $\pbeta{t}$, we can establish the following monotonicity
property.

\begin{lemma}\label{lem:planarestimates}
If $0\leq \pbeta{0}< \pi/2$, then $\pbeta{0} \geq \pbeta{1}\geq \ldots \pbeta{t} \geq \ldots \geq 0$.
\end{lemma}
\begin{proof}
We assume $\pbeta{0}>0$, the extension to $\pbeta{0}=0$ is straightforward.
Define $\alpha^{\dotp{t}}$ as the angle between $\pb{t}$ and $\pb{t+1}$ such that
\begin{eqnarray}
\cos \palpha{t} &=& \frac{\langle \pb{t}, \pb{t+1} \rangle }{\|\pb{t}\|_2 \|\pb{t+1}\|_2} \, ,
  \\
\sin \palpha{t}&=& \frac{\dotp{ \pb{t+1}, \e{2, t}}}{\|\pb{t+1}\|}
  \, .
\end{eqnarray}
The strategy of the proof is to use induction to prove that the following three statements hold for
$\forall t\geq 0$:
\begin{enumerate}
  \item[(i)] $\pbeta{t}\in (0, \frac{\pi}{2})$.
  \item[(ii)]$\palpha{t}\in (0, \pbeta{t})$.
  \item[(iii)]$\pbeta{t+1} = \pbeta{t}- \palpha{t}\in (0, \pbeta{t})$.
\end{enumerate}
It is clear that the claim of the lemma holds if (iii) holds for all $ t\geq 0$.
The inductive argument uses the following chain of arguments for step $t$:
\begin{description}
  \item[Claim 1] If (i) holds for $t$, then (ii) holds for $t$.
  \item[Claim 2] If (i) and (ii) hold for $t$, then (iii) holds for $t$.
  \item[Claim 3] If (i), (ii), and (iii) hold for $t$, then (i) holds for $t+1$.
\end{description}
Since (i) holds for $t=0$ by assumption, it suffices to prove Claims 1--3.

Claim 3 is trivially true, and Claim 2 follows from the fact that $\stt$  and all $\pb{t}$ lie in a
the same two-dimensional subspace.
So we just have to prove Claim 1.
For sake of clarity, we choose the orthogonal basis $\e{1, t}, \e{2, t}, \dotsc, \e{d, t}$ satisfying
\eqref{eq:rotationcondition} to simplify the calculation.
  Let $\vU_t$ be the orthogonal matrix whose rows are
  $\e{1,t},\e{2,t},\dotsc,\e{d,t}$, so
\[
  \vU_t\pb{t}
  \ = \
  (\norm{\pb{t}}, 0, 0, \dotsc, 0)^\t
  \, ,
  \qquad
  \vU_t\stt
  \ = \
  (\ptt{1, t}, \ptt{2, t}, 0, \dotsc, 0)^\t
  \, .
\]
Define
\begin{align*}
  \tpb{t+1}_{\langle t \rangle}
  & \ \triangleq \
  \vU_t \pb{t+1}
  \, ,
  &
  \tpb{t}
  & \ \triangleq \
  \vU_t \pb{t}
  \, ,
  &
  \tpt{t}
  & \ \triangleq \
  \vU_t \stt
  \, .
\end{align*}
Also, let  $\tpbi{{\langle t \rangle, i}, t+1}$ denote the $i^{\rm th}$ element of $\tpb{t+1}_{\langle t \rangle}$. We also define the same notations of $\tpq{t+1}_{\langle t \rangle}, \tpq{t}, \tpqi{{\langle t\rangle, i}, t+1}$ for $\pq{t}$ and $\tpa{t+1}_{\langle t \rangle}, \tpa{t}, \tpai{{\langle t\rangle, i}, t+1}$ for $\pa{t}$.
Using this coordinate system and Lemma \ref{lemma:projection}, Claim 1 is equivalent to proving that
if $\ptt{2, t}>0$, then $\palpha{t}>0$ and $\palpha{t}<\pbeta{t}$.
Therefore, in the rest of the proof, we essentially do these two steps.
\begin{enumerate}
\item $\palpha{t}> 0$: First note that $\tpb{t+1}$ and $\tpq{t+1}$ are in the same direction. Hence, to prove that $\palpha{t}>0$ we should show that $\tpqi{{\langle t\rangle, 2}, t+1}> 0$.
We have
\begin{eqnarray}\label{eq:q2dynamic1}
\tpqi{{\langle t\rangle, 2}, t+1}&=&\int \w_d(\y-\tpa{t}, \tpb{t})y_2\pdf{d, \y, \tpt{t}} \dif\y\n\\
&=&\int \w(y_1-\tpai{1, t}, \|\pb{t}\|)y_2\frac{1}{2}(\phi_d(\y-\tpt{t})+\phi_d(\y+\tpt{t})) \dif\y\n\\
&=&\frac{1}{2}\int \w(y_1-\tpai{1, t}, \|\pb{t}\|)\phi(y_1- \ptt{1, t}) \dif y_1\int y_2\phi(y_2-\ptt{2, t})
  \dif y_2\n\\
&~&+\frac{1}{2}\int \w(y_1-\tpai{1, t}, \|\pb{t}\|)\phi(y_1+ \ptt{1, t}) \dif y_1\int
  y_2\phi(y_2+\ptt{2, t}) \dif y_2 \\
&=&\ptt{2, t}\int \w(y_1-\tpai{1, t}, \|\pb{t}\|)\frac{1}{2}(\phi(y_1- \ptt{1, t})-\phi(y_1+ \ptt{1, t}))
  \dif y_1 \\
&=& \ptt{2, t}\int_{0}^{\infty}\frac{e^{2y_1\| \tpb{t}\|}-e^{-2y_1\| \tpb{t}\|}}{e^{2y_1\|
\tpb{t}\|}+e^{ 2\tpai{1, t}\| \tpb{t}\|}+e^{-2y_1\| \tpb{t}\|}+e^{-2\tpai{1, t}\|
\tpb{t}\|}}\dpdf{y_1, \ptt{1, t}}  \dif y_1 \nonumber \\
&=& \ptt{2, t} S(\tpai{1, t}, \|\pb{t}\|, \ptt{1, t}), \label{equ:iteq22}
\end{eqnarray}
where $\dpdf{y, \theta}\triangleq \frac{1}{2}(\phi(y-\theta)-\phi(y+\theta))$ is shorthand for the
difference of two Gaussian densities, and
$S: \mathbb{R}^3 \rightarrow \mathbb{R}$ is defined by
\begin{align*}
  \lefteqn{
    S(x_a, x_b, x_\theta)
  } \\
  & \ \triangleq \
  \int_{0}^{\infty}\frac{e^{2yx_b}-e^{-2yx_b}}{e^{2yx_b}+e^{-2x_ax_b}+e^{-2yx_b}+e^{2x_ax_b}}
  \cdot \frac{1}{2\sqrt{2\pi}}(e^{-(y-x_\theta)^2/2}-e^{-(y+x_\theta)^2/2})
  \dif y
  \\
  & \ = \
  \int_0^\infty \w(y-x_a, x_b)\frac{1}{2}(\phi(y-x_\theta)-\phi(y+x_\theta)) \dif y
  \, .
\end{align*}
Hence it is clear that $\ptt{2, t}>0$ implies $\palpha{t}>0$ since $S(x_a, x_b, x_\theta)>0$ for all
$x_b>0$ and $x_\theta>0$.
\item  $\palpha{t}< \pbeta{t}$: We just need to show $\palpha{t}<\pi/2$ and compare the co-tangent of $\palpha{t}$ and $\pbeta{t}$. This means that we have to show $\tpqi{{\langle t\rangle, 1}, t+1}>0$ and
    compare $\frac{ \tpqi{{\langle t\rangle, 1}, t+1}}{ \tpqi{{\langle t\rangle, 2}, t+1}}$  with $\frac{ \ptt{1, t}}{ \ptt{2, t}}$. We first calculate    $\tpqi{{\langle t\rangle, 1}, t+1}$.
\begin{eqnarray}
\tpqi{{\langle t\rangle, 1}, t+1}&=&\int \w_d(\y-\tpa{t}, \tpb{t})y_1\pdf{d, \y, \tpt{t}}\dif\y\n\\
&=&\int \w(y_1-\tpai{1, t}, \|\pb{t}\|)y_1\pdf{d, \y, \tpt{t}} \dif\y \\
&=&\int \w(y_1-\tpai{1, t}, \|\pb{t}\|)y_1\pdf{, y_1, \ptt{1, t}} \dif y_1 \\
&=&\funcQ( \tpai{1, t}, \| \tpb{t}\|, \ptt{1, t})\n\\
&=& \int_{0}^{\infty}\frac{e^{2y_1\| \tpb{t}\|}-e^{-2y_1\| \tpb{t}\|}}{e^{2y_1\| \tpb{t}\|}+e^{
2\tpai{1, t}\| \tpb{t}\|}+e^{-2y_1\| \tpb{t}\|}+e^{-2\tpai{1, t}\| \tpb{t}\|}} y_1\pdf{, y_1, \ptt{1, t}}
\dif y_1, \n\\\label{equ:iteq21}
\end{eqnarray}
where $\funcQ : \mathbb{R}^3 \rightarrow \mathbb{R}$ is defined as
\begin{eqnarray*}
\funcQ(x_{a}, x_{b}, x_{\theta})=\int \w(y-x_a, x_b)y\frac{1}{2}(\phi(y-x_\theta)+\phi(y+x_\theta))
\dif y \, .
\end{eqnarray*}
It is clear that  $\tpqi{{\langle t\rangle, 1}, t+1}>0$. For comparing the co-tangent of two angles, we need to further simplify $\tpqi{{\langle t\rangle, 1}, t+1}$. We have,
\begin{eqnarray}
\tpqi{{\langle t\rangle, 1}, t+1}&=& \funcQ( \tpai{1, t}, \| \tpb{t}\|, \ptt{1, t}) \nonumber \\
&=&\int \w(y_1-\tpai{1, t}, \|\pb{t}\|)y_1\frac{1}{2}(\phi(y_1-\ptt{1, t})+\phi(y_1+\ptt{1, t})) \dif y_1 \\
&=&\ptt{1, t} \int \w(y_1-\tpai{1, t}, \|\pb{t}\|)\frac{1}{2}(\phi(y_1-\ptt{1, t})-\phi(y_1+\ptt{1, t}))
  \dif y_1 \nonumber \\
&&+\int \w(y_1-\tpai{1, t}, \|\pb{t}\|)\frac{1}{2}\{(y_1- \ptt{1, t})\phi(y_1-\ptt{1, t})+(y_1+
\ptt{1, t})\phi(y_1+\ptt{1, t})\} \dif y_1 \nonumber \\
&=&\ptt{1, t} S(\tpai{1, t}, \|\pb{t}\|, \ptt{1, t})+\int_{-\infty}^{\infty}
\w(y_1+\ptt{1, t}-\tpai{1, t}, \|\pb{t}\|)y_1 \frac{1}{2}\phi(y_1) \dif y_1 \nonumber \\
&&+\int_{-\infty}^{\infty} \w(y_1-\ptt{1, t}-\tpai{1, t}, \|\pb{t}\|)y_1 \frac{1}{2}\phi(y_1) \dif y_1
\, . \nonumber \\
&=&  \ptt{1, t} S(\tpai{1, t}, \|\pb{t}\|, \ptt{1, t})  + R(\| \tpb{t}\|, \tpai{1, t}-\ptt{1, t})+R(\|\pb{t}\|, \tpai{1, t}+\ptt{1, t}), \label{eq:Q1SRR}
\end{eqnarray}
where $R: \mathbb{R}^2 \rightarrow \mathbb{R}$ is defined as
\begin{eqnarray*}
R(x_{b}, x)&\triangleq
  &\int_{0}^{+\infty}\frac{e^{2yx_{b}}-e^{-2yx_b}}{e^{2yx_b}+e^{2xx_b}+e^{-2yx_b}+e^{-2xx_b}}\frac{1}{2\sqrt{2\pi}}ye^{-y^2/2}
  \dif y \, , \\
&=&\int_{-\infty}^{\infty} \w(y-x, x_b)y \frac{1}{2}\phi(y) \dif y \, .
\end{eqnarray*}
Employing \eqref{eq:Q1SRR} and \eqref{equ:iteq22} we have
\begin{align*}
\cot {\palpha{t}}
& \ = \
\frac{ \tpbi{{\langle t\rangle, 1}, t+1}}{ \tpbi{{\langle t\rangle, 2}, t+1}}
\ = \
\frac{ \tpqi{{\langle t\rangle, 1}, t+1}}{ \tpqi{{\langle t\rangle, 2}, t+1}}
\\
& \ = \
\frac{ \ptt{1, t}}{ \ptt{2, t}}+\frac{R(\| \tpb{t}\|, \tpai{1, t}- \ptt{1, t})+R(\|\pb{t}\|, \tpai{1, t}+ \ptt{1, t})}{ \tpqi{{\langle t\rangle, 2}, t+1}}
  \\
  & \ > \
  \frac{\ptt{1, t}}{ \ptt{2, t}}
  \ = \
  \cot {\pbeta{t}}
  \, ,
\end{align*}
where the last inequality is due to the fact that $R(x_{b}, x) >0 $ for all $x_b>0$.
\end{enumerate}
\end{proof}

\subsection{Notations for the Remaining Proofs}\label{sec:notation}
In this section we collect the main notations that will be used in the proofs of our results. For the basic notation, we let $\phi_d(\vx)$ and $\Phi_d(\vx)$ denote the pdf and CDF for $d$-dimension standard Gaussian distribution respectively. We use $\phi(x)$ and $\Phi(x)$ as shorthand for one-dimension case. Let $\pdf{d, x, x_\theta}$ denote the pdf for $\boldsymbol{X}\sim 0.5 N(-\vx_{\theta}, I)+0.5 N(\vx_{\theta}, I)$, i.e.,
\[\pdf{d, \vx, \vx_\theta}=\frac{1}{2}(\phi_d(\vx-\vx_\theta)+\phi_d(\vx+\vx_\theta)).\]
We shorthand $\pdf{1, x, x_\theta}$ as $\pdf{, x, x_\theta}$ if $x\in \bbR$. In addition, we let $\dpdf{x, x_\theta}$ denote the difference between these two pdf, i.e.,
\[\dpdf{\vx, \vx_\theta}=\frac{1}{2}(\phi(\vx-\vx_\theta)-\phi_d(\vx+\vx_\theta)).\]
Next, we introduce the notations for important functions.
Similar to the proof of Lemma \ref{lem:planarestimates}, in many other proofs we will use the
rotation matrix $\vU_t$ that satisfies $\vU_t \pb{t} = (\|\pb{t}\|_2, 0, 0, \ldots, 0)^\t$, and $\vU_t \stt
= (\ptt{1, t}, \ptt{2, t} , 0, \ldots, 0)^\t$.  We also use the notations $\tpb{t+1}_{\langle t
\rangle} \triangleq \vU_t \pb{t+1}$, $\tpb{t} \triangleq \vU_t \pb{t}$ and $\tstt \triangleq \vU_t \stt$.
Also, let  $\tpbi{{\langle t\rangle, i}, t+1}$ denote the $i^{\rm th}$ element of $\tpb{t+1}_{\langle
t \rangle}$. We also define the same notations of $\tpq{t+1}_{\langle t
\rangle}, \tpq{t}, \tpqi{{\langle t\rangle, i}, t+1}$ for $\pq{t}$ and $\tpa{t+1}_{\langle t
\rangle}, \tpa{t}, \tpai{{\langle t\rangle, i}, t+1}$ for $\pa{t}$. By Lemma \ref{lemma:sign} and
\ref{lemma:symmetrizeinitial}, we can assume that $\dotp{\pb{t}, \stt}\geq 0$ and
$\dotp{\pb{t}, \pa{t}}\geq 0$ for all $ t\geq 0$ without loss of generality. Hence we have $$\ptt{1, t}\geq 0, \quad \text{and}\quad \tpai{1, t}\geq 0.$$
 Since for any $i \geq 3$, the $i^{\rm th}$ coordinate is zero in the span of $\tpb{t}$ and $\tstt$, we prove that $\tpbi{{\langle t\rangle, i}, t+1} =0$ for $i \geq 3$. Then according to \eqref{eq:tildeiterationab} we have
\begin{eqnarray}\label{eq:tildeiterationab2}
\tpa{t+1}&=&\frac{\tpq{t+1}(1-2\pp{t+1})}{2\pp{t+1}(1-\pp{t+1})}, \n \\
\tpb{t+1}&=&\frac{\tpq{t+1}}{2\pp{t+1}(1-\pp{t+1})}.
\end{eqnarray}
The first notation we would like to introduce $\funcP: \mathbb{R}^3 \rightarrow \mathbb{R}$
\begin{eqnarray}
\funcP(x_{a}, x_{b}, x_{\theta})&\triangleq&\int\frac{e^{{yx_{b}}-{x_{a}x_{b}}}}{e^{{yx_{b}}-{x_{a}x_{b}}}+e^{-{yx_{b}}+{x_{a}x_{b}}}}\frac{1}{2\sqrt{2\pi}}(e^{-(y-x_{\theta})^2/2}+e^{-(y+x_{\theta})^2/2})
  \dif y\n \\
&=&\int \w(y-x_a, x_b)\pdf{, y, x_\theta} \dif y \, .
  \label{eq:Pdef}
\end{eqnarray}
The importance of this function is clarified in the following calculations:
\begin{eqnarray}
\pp{t+1}&=&\bbE \w(\Y-\tpa{t}, \tpb{t})\n\\
&=&\int \w(y_1-\tpai{1, t}, \tpbi{1, t})\pdf{, y, \ptt{1, t}}\n\\
&=&\int\frac{e^{y_1\| \tpb{t}\|- \tpai{1, t}\| \tpb{t}\|}}{e^{y_1\| \tpb{t}\|- \tpai{1, t}\|
\tpb{t}\|}+e^{-y_1\| \tpb{t}\|+\tpai{1, t}\|
\tpb{t}\|}}\frac{1}{2\sqrt{2\pi}}(e^{-\frac{(y_1-\ptt{1, t})^2}{2}}+e^{-\frac{(y_1+\ptt{1, t})^2}{2}})
\dif y_1\n\\
&=&\funcP(\tpai{1, t}, \| \tpb{t}\|, \ptt{1, t}), \label{equ:itep2}
\end{eqnarray}
Our second notation is the $\funcQ: \mathbb{R}^3 \rightarrow \mathbb{R}$, where
\begin{eqnarray}
\funcQ(x_{a}, x_{b}, x_{\theta})& \triangleq &\int\frac{e^{{yx_{b}}-{x_{a}x_{b}}}}{e^{{y
x_{b}}-{x_{a}x_{b}}}+e^{-{yx_{b}}+{x_{a}x_{b}}}}y\frac{1}{2\sqrt{2\pi}}(e^{-(y-x_{\theta})^2/2}+e^{-(y+x_{\theta})^2/2})
\dif y.\n\\
&=&\int \w(y-x_a, x_b)y\pdf{, y, x_\theta} \dif y.\label{eq:Qdef}
\end{eqnarray}
Note that according to \eqref{equ:iteq21}, we have
\begin{equation}
\tpqi{{\langle t\rangle, 1}, t+1} = \funcQ( \tpai{1, t}, \| \tpb{t}\|, \ptt{1, t}).\n
\end{equation}
The next function we introduce here is $S: \mathbb{R}^3 \rightarrow \mathbb{R}$
\begin{eqnarray}
S(x_a, x_b, x_\theta)&\triangleq&
  \int_{0}^{\infty}\frac{e^{2yx_b}-e^{-2yx_b}}{e^{2yx_b}+e^{-2x_ax_b}+e^{-2yx_b}+e^{2x_ax_b}}\frac{1}{2\sqrt{2\pi}}(e^{-(y-x_\theta)^2/2}-e^{-(y+x_\theta)^2/2})
  \dif y\n\\
&=&\int \w(y-x_a, x_b) \dpdf{y, x_\theta} \dif y.\n
\end{eqnarray}
Note that according to \eqref{equ:iteq22},
\begin{equation}
\tpqi{{\langle t\rangle, 2}, t+1} = \ptt{2, t} S(\tpai{1, t}, \|\pb{t}\|, \ptt{1, t}).\label{eq:q2thetaS}
\end{equation}
Another notation we introduce here is $R: \mathbb{R}^2 \rightarrow \mathbb{R}$ defined as
\begin{eqnarray}
R(x_{b}, x)&\triangleq
  &\int_{0}^{+\infty}\frac{e^{2yx_{b}}-e^{-2yx_b}}{e^{2yx_b}+e^{2xx_b}+e^{-2yx_b}+e^{-2xx_b}}\frac{1}{2\sqrt{2\pi}}ye^{-y^2/2}
  \dif y\n\\
&=&\frac{1}{2}\int \w(y-x, x_b)y\phi(y) \dif y.\n
\end{eqnarray}
According to \eqref{eq:Q1SRR},
\begin{eqnarray}
\tpqi{{\langle t\rangle, 1}, t+1}&=& \funcQ (\tpai{1, t}, \|\pb{t}\|, \ptt{1, t})\n\\
 &=& \ptt{1, t} S(\tpai{1, t}, \| \tpb{t}\|, \ptt{1, t}) + R(\| \tpb{t}\|, \tpai{1, t} - \ptt{1, t} ) + R (\| \tpb{t}\|, \tpai{1, t}+\ptt{1, t}).\label{eq:Q1SRR1}
\end{eqnarray}

The final notation that we will be using in this paper is the function $F : \mathbb{R}^2 \rightarrow \mathbb{R}$:
\begin{eqnarray}
F(x_b, x_{\theta})&=&\int\frac{e^{(y+x_{\theta})x_{b}}-e^{-(y+x_{\theta})x_{b}}}{e^{(y+x_{\theta})x_{b}}+e^{-(y+x_{\theta})x_{b}}}(y+x_{\theta})\frac{1}{\sqrt{2\pi}}e^{-y^2/2}
  \dif y\n\\
&=&\int (2\w(y+x_\theta, x_b)-1)(y+x_{\theta})\phi(y) \dif y.\label{eq:Fdefinition}
\end{eqnarray}
To understand where this function may appear, note that
\begin{eqnarray}
\lefteqn{\funcQ(0, x_b, x_\theta) = \int \w(y, x_b)y\frac{1}{2}\pdf{, y, x_\theta} \dif y} \nonumber \\
&=& \int \w(y, x_b) y\frac{1}{2}\phi(y-x_\theta) \dif y +\int \w(y, x_b)
y\frac{1}{2}\phi(y+x_{\theta}) \dif y \nonumber \\
&=& \int \w(y+x_\theta, x_b) (y+x_\theta) \frac{1}{2}\phi(y) \dif y +\int \w(y-x_\theta, x_b)
(y-x_\theta)\frac{1}{2}\phi(y) \dif y \nonumber \\
&=& \int \w(y+x_\theta, x_b) (y+x_\theta)\frac{1}{2}\phi(y) \dif y -\int \w(-y-x_\theta, x_b)
(y+x_\theta)\frac{1}{2}\phi(y) \dif y \nonumber \\
&=&\frac{1}{2}
\int\frac{e^{(y+x_{\theta})x_{b}}-e^{-(y+x_{\theta})x_{b}}}{e^{(y+x_{\theta})x_{b}}+e^{-(y+x_{\theta})x_{b}}}(y+x_{\theta})\frac{1}{\sqrt{2\pi}}e^{-y^2/2}
\dif y = \frac{1}{2} F(x_b, x_{\theta}). \label{eq:QFrelation}
\end{eqnarray}

\subsection{Proof of Theorem \ref{thm:convergencerate1}}\label{sec:proof:thm:convergencerate1}
Without loss of generality, we assume that $\dotp{\pb{0}, \stt}>0$ and $\dotp{\pa{0}, \pb{0}}\geq 0$. We employ the notations and equations reviewed in Appendix \ref{sec:notation}. For notational simplicity we also define $R_s \triangleq R(\| \tpb{t}\|, \tpai{1, t} - \ptt{1, t} ) + R (\| \tpb{t}\|, \tpai{1, t}+ \ptt{1, t})$ and $S_s \triangleq S(\tpai{1, t}, \| \tpb{t}\|, \ptt{1, t}) $. Since $\pb{t+1}$ and $\pq{t+1}$ are in the same direction, we have
\begin{eqnarray}
\cos \pbeta{t+1} &=& \frac{\langle \tpb{t+1}, \tstt \rangle}{\| \tstt\| \| \tpb{t+1} \|} =  \frac{\langle \tpq{t+1}, \tstt \rangle}{\| \tstt\| \| \tpq{t+1} \|} \nonumber \\
&=& \frac{\tpqi{{\langle t\rangle, 1}, t+1}  \ptt{1, t}+ \tpqi{{\langle t\rangle, 2}, t+1} \ptt{2, t} }{\| \tstt \| \sqrt{ (\tpqi{{\langle t\rangle, 1}, t+1})^2+ (\tpqi{{\langle t\rangle, 2}, t+1})^2}  } \nonumber \\
&\overset{(a)}{=}& \frac{ (\ptt{1, t})^2 S_s + \ptt{1, t} R_s  + (\ptt{2, t})^2 S_s }{\|\tstt\| \sqrt{( \ptt{1, t})^2 S_s^2+ R_s^2+ 2R_s \ptt{1, t} S_s+ (\ptt{2, t})^2 S_s^2}  } \nonumber \\
&=& \frac{ \| \tstt\|^2 S_s + \ptt{1, t} R_s   }{\|\tstt\| \sqrt{\| \tstt\|^2 S_s^2+ R_s^2+ 2R_s \ptt{1, t} S_s}  }, \n
\end{eqnarray}
where Equality (a) is the result of \eqref{eq:Q1SRR1} and \eqref{eq:q2thetaS}.
Hence it is straightforward to check that
\begin{eqnarray*}
\sin \beta^{t+1}\ =\ \frac{\ptt{2, t} R_s }{ \| \tstt\| \sqrt{(\|\tstt\|^2 S_s^2+ R_s^2+ 2R_s \ptt{1, t} S_s)}}\ \leq \ \frac{\ptt{2, t}}{\| \tstt\|} \frac{R_s}{R_s + \ptt{1, t} S_s}.
\end{eqnarray*}
Note that since $\tpbi{2, t}=0$, we have $$\frac{\ptt{2, t}}{\|\tstt\|} = \sin (\pbeta{t}).$$ Hence, we have
\begin{eqnarray}\label{eq:sinbetatplus1}
\sin \pbeta{t+1} = \frac{R_s}{ R_s + \ptt{1, t} S_s} \sin \pbeta{t}.
\end{eqnarray}
Our goal is to prove that there exists $0<\kappa_\beta<1$, such that $R_s/(R_s + \ptt{1, t} S_s) \leq \kappa_\beta$ at every iteration.  Toward this goal we will prove that $\ptt{1, t} S_s >0$. First note that since according to Lemma \ref{lem:planarestimates} the angle $\pbeta{t}$ is decreasing $\ptt{1, t}$ is an increasing sequence. Hence, $\ptt{1, t} S_s \geq \ptt{1, 0} S_s$. Our goal is to show that $S_s>0$. Note that $S_s = S(\tpai{1, t}, \| \tpb{t}\|, \ptt{1, t})$ is only zero if $\|\pb{t}\|=0$ and can only go to zero if $\tpai{1, t} \rightarrow \infty$ or $\|\pb{t}\| \rightarrow \infty$. Hence, if we find a lower bound for $\inf_{t} \| \pb{t}\|$ and prove that we have an upper bound for $\sup_t \|\pa{t}\|$ and $\sup_t \| \pb{t}\|$, then we obtain a non-zero lower bound for $S_s$. The following two lemmas prove our claims:

\begin{lemma}\label{lemma:upperboundEMestimates}
For any initialization $\pa{0}, \pb{0}\in \bbR^d$, we have
\begin{eqnarray}
\|\pa{t}\|^2&\leq& \max\{\|\pa{0}\|^2, \frac{2}{\pi}+\frac{\|\stt\|^2}{2}, \frac{16}{9}+\frac{73}{36}\|\stt\|^2\}\ \triangleq\  c_{U, 1}^2, \forall t\geq 0, \n\\
\|\pb{t}\|^2&\leq& \max\{\|\pb{0}\|^2, \|\stt\|^2+\frac{1}{4c_{U, 2}^2(1-c_{U, 2})^2}(1+\|\stt\|^2)\}\ \triangleq\  c_{U, 3}^2, \forall t\geq 0, \n
\end{eqnarray}
where $c_{U, 2}=\frac{1}{4}(1-\Phi(c_{U, 1}+\|\stt\|))$. Hence,  $\{\|\pa{t}\|, \|\pb{t}\|\}_t$ belong to a compact set.
\end{lemma}
We postpone the proof of this lemma until Appendix \ref{sec:proofuboundest}, but the fact that the estimates remain bounded should not be surprising for the reader.

\begin{lemma}\label{lem:lowerboundbt}
Let $\pb{t}$ and $\pa{t}$ denote the estimates of Population EM. There exists a value $c_l>0$ depending on $\|\stt\|$, $\langle \pb{0}, \stt \rangle$, $\|\pa{0}\|$, and $\|\pb{0}\|$ such that
\[
\|\pb{t}\| \geq \min (\|\pb{0}\|, c_l)\triangleq c_{L, 1}.
\]
\end{lemma}
We postpone the proof of this claim to Appendix \ref{sec:proof:lowerboundbt}. Note that according to Lemma \ref{lemma:upperboundEMestimates} we know that $\sup_t \|\pa{t}\| \leq c_{U, 1}$ and $\sup_t \|\pb{t}\| \leq c_{U, 3}$. Hence, we define
\[
 \kappa_\beta= \max_{ c_{L, 1}\leq \| \tpb{t}\| \leq c_{U, 3}, \|\tpa{t}\| \leq c_{U, 1} } \frac{ R_s}{\ptt{1, 0} S(\tpai{1, t}, \| \tpb{t}\|, \ptt{1, t})+  R_s)} \in(0, 1),
\]
then \eqref{eq:sinbetatplus1} implies
\[
\sin \pbeta{t+1} \leq \kappa_\beta \sin \pbeta{t}.
\]
This proves the first claim in Theorem \ref{thm:convergencerate1}. Our next goal is to prove the second claim, i.e.,
\[
\|\pa{t+1} \|^2 \leq \kappa_a^2 \|\pa{t}\|^2 + \frac{\|\stt\|^2 \sin \pbeta{t}}{4}.
\]
As before we write $\|\pa{t+1} \|^2 =( \tpai{{\langle t\rangle, 1}, t+1} )^2 + ( \tpai{{\langle t\rangle, 2}, t+1} )^2$ and then bound each term separately. According to \eqref{eq:tildeiterationab2}, we have
\begin{eqnarray}\label{eq:a_1boundconvergence}
\tpai{{\langle t\rangle, 1}, t+1}\ =\ \frac{\funcQ ( \tpai{1, t}, \| \tpb{t} \|, \ptt{1, t})(1 -2\funcP(\tpai{1, t}, \| \tpb{t} \|, \ptt{1, t}))}{\funcP(\tpai{1, t}, \|\tpb{t} \|, \ptt{1, t}) (1-\funcP(\tpai{1, t}, \|\tpb{t} \|, \ptt{1, t})) }\ \overset{(\star)}{\leq}\ \kappa_a \tpai{1, t} \ \leq\ \kappa_a \|\tpa{t}\|.
\end{eqnarray}
where Inequality $(\star)$ is due to the following lemma:

\begin{lemma}\label{lemma:aconvergeupperbound}
For any $\theta\geq 0$, there exists a constant $\kappa_a\in (0, 1)$ only depending on $\theta$ and continuous for $\theta>0$ such that $$\frac{\funcQ(x_a, x_b, \theta)(1-2\funcP(x_a, x_b, \theta))}{2\funcP(x_a, x_b, \theta)(1-\funcP(x_a, x_b, \theta))}\leq \kappa_ax_a,\ \ \forall x_a\geq 0, x_b>0.$$
\end{lemma}
The proof of this Lemma is presented in Appendix \ref{proof:lemubkappa_a}. Our next step is to establish the convergence of $\tilde{a}_{t, 2}^{t+1}$.  From \eqref{eq:tildeiterationab2} and \eqref{eq:q2thetaS} we have:
\begin{eqnarray}
\tpai{{\langle t\rangle, 2}, t+1} &=& \frac{\tpqi{{\langle t\rangle, 2}, t+1} (1- 2\funcP (\tpai{1, t}, \|\pb{t}\|, \ptt{1, t} ))}{2\funcP (\tpai{1, t}, \| \tpb{t}\|, \ptt{1, t} ) (1-\funcP (\tpai{1, t}, \| \tpb{t}\|, \ptt{1, t} ))} \nonumber \\
&=& \ptt{2, t} \frac{S(\tpai{1, t}, \|\pb{t}\|, \ptt{1, t}) (1- 2\funcP (\tpai{1, t}, \|\pb{t}\|, \ptt{1, t} ))}{2\funcP (\tpai{1, t}, \| \tpb{t}\|, \ptt{1, t} ) (1-\funcP (\tpai{1, t}, \| \tpb{t}\|, \ptt{1, t} ))}.\n
\end{eqnarray}
And according to \eqref{equ:itep2}, we have
\begin{eqnarray}
\funcP(\tpai{1, t}, \|\pb{t}\|, \ptt{1, t})&=&\int \w(y-\tpai{1, t}, \|\pb{t}\|)\pdf{, y, \ptt{1, t}} \dif y \n\\
&=&\int_{y=0}^{\infty}(\w(y-\tpai{1, t}, \|\pb{t}\|)+\w(-y-\tpai{1, t}, \|\pb{t}\|))\pdf{, y, \ptt{1, t}}
  \dif y\n\\
&=&\int_{y=0}^{\infty}\frac{e^{2y\|\pb{t}\|}+2e^{-2\tpai{1, t}\|\pb{t}\|}+e^{-2y\|\pb{t}\|}}{e^{2y\|\pb{t}\|}+e^{2\tpai{1, t}\|\pb{t}\|}+e^{-2y\|\pb{t}\|}+e^{-2\tpai{1, t}\|\pb{t}\|}}\pdf{, y, \ptt{1, t}}
  \dif y\n\\
&>&S(\tpai{1, t}, \|\pb{t}\|, \ptt{1, t})\ \geq\  0.\label{eq:PbiggerthanS}
\end{eqnarray}

Hence, we have
\begin{equation}\label{eq:a_2boundconvergence}
\tpai{{\langle t\rangle, 2}, t+1} \ \leq\ \frac{\ptt{2, t}}{2}\ =\ \frac{\|\stt\| \sin \pbeta{t}}{2}.
\end{equation}
Combining \eqref{eq:a_1boundconvergence} and \eqref{eq:a_2boundconvergence} establishes the second part of our main Theorem.

\subsubsection{Proof of Lemma \ref{lemma:upperboundEMestimates}}\label{sec:proofuboundest}
In this section we use the notations and equations that are summarize in Appendix
\ref{sec:notation}. Without loss of generality, we assume that $\dotp{\pb{0}, \stt}>0$ and $\dotp{\pa{0}, \pb{0}}\geq 0$ and
it is straightforward to show that if $\|\pb{0}\|=0$, then $$\pa{t}=\pb{t}=\textbf{0}, \forall t\geq 1.$$
Hence, we assume that $\|\pb{0}\|>0$. We again rotate the coordinates with the $\vU_t$ matrix we introduced in Appendix \ref{sec:notation}. Under this coordinate systems we know that $\tpai{{\langle t\rangle, i}, t+1} = \tpbi{{\langle t\rangle, i}, t+1} =0$ for every $i \geq 3$. Hence, we will prove the boundedness of $\tpai{{\langle t\rangle, 1}, t+1}, \tpai{{\langle t\rangle, 2}, t+1}, \tpbi{{\langle t\rangle, 1}, t+1}$, and $\tpbi{{\langle t\rangle, 2}, t+1}$. We start with bounding $\tpai{{\langle t\rangle, 2}, t+1}$ and $\tpbi{{\langle t\rangle, 2}, t+1}$. According to \eqref{eq:tildeiterationab2} we have
\begin{eqnarray}
\tpai{{\langle t\rangle, 2}, t+1}&=&\frac{\tpqi{{\langle t\rangle, 2}, t+1}(1-2\pp{t+1})}{2\pp{t+1}(1-\pp{t+1})}\ \leq\ \frac{\tpqi{{\langle t\rangle, 2}, t+1}}{2\pp{t+1}}\ \overset{(b)}{=}\ \frac{\ptt{2, t} S(\tpai{1, t}, \|\pb{t}\|, \ptt{1, t})}{2\pp{t+1}},  \nonumber \\
\tpbi{{\langle t\rangle, 2}, t+1}&=&\frac{\tpqi{{\langle t\rangle, 2}, t+1} }{2\pp{t+1}(1-\pp{t+1})}\ \overset{(c)}{\leq}\ \frac{\tpqi{{\langle t\rangle, 2}, t+1} }{\pp{t+1}}\ \overset{(d)}{=}\ \frac{  \ptt{2, t} S(\tpai{1, t}, \|\pb{t}\|, \ptt{1, t})}{\pp{t+1}}, \label{eq:boundinga2b2}
\end{eqnarray}
where Equalities (b) and (d) are due to \eqref{eq:q2thetaS}. To obtain Inequality (c) we used the following chain of arguments: According to Lemma \ref{lemma:sign}, $\pp{t}\leq 0.5$ for every $t$. Hence, $2(1-\pp{t+1})\geq 1$.

With exactly same calculation showed in \eqref{eq:PbiggerthanS}, we have
$$\pp{t+1}\ =\ \funcP(\tpai{i, t}, \|\pb{t}\|, \ptt{1, t})\ >\ S(\tpai{i, t}, \|\pb{t}\|, \ptt{1, t}).$$ Together with \eqref{eq:boundinga2b2}, we obtain
\begin{eqnarray}
|\tpai{{\langle t\rangle, 2}, t+1}|&\leq& \frac{|\ptt{2, t}|}{2}\ \leq\ \frac{\|\stt\|}{2}.\n\\
|\tpbi{{\langle t\rangle, 2}, t+1}|&\leq& \frac{|\ptt{2, t}|}{2(1-\pp{t+1})}\ \leq\ \|\stt\|.\label{eq:upperboundeq3}
\end{eqnarray}
Hence the only remaining step is to bound $\tpai{{\langle t\rangle, 1}, t+1}$ and $\tpbi{{\langle t\rangle, 1}, t+1}$. To bound $\tpai{{\langle t\rangle, 1}, t+1}$ we consider two separate cases.

\begin{enumerate}
\item $\tpai{1, t}\geq\ptt{1, t}\geq 0, t\geq 0$:   First note that according to \eqref{eq:tildeiterationab2}, we have
\begin{eqnarray}\label{eq:boundinga1upper}
0\leq \tpai{{\langle t\rangle, 1}, t+1}&=&\frac{\tpqi{{\langle t\rangle, 1}, t+1}(1-2\pp{t+1})}{2\pp{t+1}(1-\pp{t+1})}\n\\
&=&\frac{\funcQ(\tpai{1, t}, \|\pb{t}\|, \ptt{1, t})(1-2\funcP(\tpai{1, t}, \|\pb{t}\|, \ptt{1, t}))}{2\funcP(\tpai{1, t}, \|\pb{t}\|, \ptt{1, t})(1-\funcP(\tpai{1, t}, \|\pb{t}\|, \ptt{1, t}))}\n\\
&\leq& \frac{\funcQ(\tpai{1, t}, \|\pb{t}\|, \ptt{1, t})}{2\funcP(\tpai{1, t}, \|\pb{t}\|, \ptt{1, t})}.
\end{eqnarray}
Hence to bound $\tpai{{\langle t\rangle, 1}, t+1}$ we require a bound for $$\frac{\funcQ(\tpai{1, t}, \|\pb{t}\|, \ptt{1, t})}{2\funcP(\tpai{1, t}, \|\pb{t}\|, \ptt{1, t})}.$$
Our next lemma provides such a bound.

\begin{lemma}\label{lem:Q1overPbound}
If $x_a\geq x_\theta\geq 0, x_b\geq 0$, we have
\[
\frac{\funcQ(x_a, x_b, x_\theta)}{2\funcP(x_a, x_b, x_\theta)}\ \leq\ \frac{x_a+\sqrt{\frac{2}{\pi}}}{2}.
\]
\end{lemma}
We prove this lemma in the Appendix \ref{sec:proofoflemmaQ1overPbound}. Combining \eqref{eq:boundinga1upper} and Lemma \ref{lem:Q1overPbound} proves
\begin{eqnarray}
(\tpai{{\langle t\rangle, 1}, t+1})^2&\leq&(\frac{\tpai{1, t}+\sqrt{\frac{2}{\pi}}}{2})^2\ \leq\ \frac{\|\pa{t}\|^2+\frac{2}{\pi}}{2}.\n
\end{eqnarray}
Therefore, combined with \eqref{eq:upperboundeq3}, we have
\begin{eqnarray}
\|\pa{t+1}\|^2&=&\|\tpa{t+1}\|\ =\ (\tpai{{\langle t\rangle, 1}, t+1})^2+(\tpai{{\langle t\rangle, 2}, t+1})^2\n\\
&\leq&\frac{\|\pa{t}\|^2+\frac{2}{\pi}}{2}+\frac{\|\stt\|^2}{4}\n\\
&\leq& \max\{(\|\pa{t}\|)^2, \frac{2}{\pi}+\frac{\|\stt\|^2}{2}\}\label{eq:upperboundeq4}.
\end{eqnarray}
\item $\tpai{1, t}<\ptt{1, t}$: Again according to \eqref{eq:tildeiterationab2} we have
\begin{eqnarray}\label{eq:boundinga1upper2}
0\leq \tpai{{\langle t\rangle, 1}, t+1}&=&\frac{\tpqi{{\langle t\rangle, 1}, t+1}(1-2\pp{t+1})}{2\pp{t+1}(1-\pp{t+1})}\n\\
&=&\frac{\funcQ(\tpai{1, t}, \|\pb{t}\|, \ptt{1, t})(1-2\funcP(\tpai{1, t}, \|\pb{t}\|, \ptt{1, t}))}{2\funcP(\tpai{1, t}, \|\pb{t}\|, \ptt{1, t})(1-\funcP(\tpai{1, t}, \|\pb{t}\|, \ptt{1, t}))}.\n
\end{eqnarray}
We know from Lemma \ref{lemma:sign} that $\pp{t+1}\leq 0.5$. In the range $0<\pp{t+1}\leq 0.5$,
$$\frac{(1-2\pp{t+1})}{2\pp{t+1}(1-\pp{t+1})}$$
is a positive decreasing function of $\pp{t+1}$. Hence, if we a lower bound for $\funcP$ may lead to an upper bound for $\tpai{{\langle t\rangle, 1}, t+1}$. The following lemma provides such an upper bound.

\begin{lemma}\label{lem:Plowerbound} If  $x_a\geq x_\theta\geq 0, x_b\geq 0$, we have
\[
\funcP(x_a, x_b, x_\theta)\geq \frac{1}{2}(1-\Phi(x_{a}-x_{\theta}))+\frac{1}{2}(1-\Phi(x_{a}+x_{\theta})).
\]
If $0\leq x_a<x_\theta, x_b\geq 0$, we have
\[
\funcP(x_{a}, x_{b}, x_{\theta})\geq\frac{1}{4}.\]
where $\Phi(x)$ is the CDF for a standard Gaussian distribution.
\end{lemma}
The proof of this lemma is presented in the Appendix \ref{sec:proofoflemmaPlowerbound}. By plugging $\funcP(\tpai{1, t}, \|\pb{t}\|, \ptt{1, t}) = 0.25$ in \eqref{eq:boundinga1upper2}, we have
\begin{eqnarray}
0\leq \tpai{{\langle t\rangle, 1}, t+1}&=&\frac{\tpqi{{\langle t\rangle, 1}, t+1}(1-2\pp{t+1})}{2\pp{t+1}(1-\pp{t+1})}\ \leq\ \frac{4}{3}\funcQ(\tpai{1, t}, \|\pb{t}\|, \ptt{1, t})\n\\
&\leq& \frac{4}{3} \int |y| \pdf{, y, \ptt{1, t}} \dif y\ \leq\ \frac{4}{3} \sqrt{\int y^2
\pdf{, y, \ptt{1, t}} \dif y }\n\\
&=&\frac{4}{3}\sqrt{1+(\ptt{1, t})^2}\ \leq\ \frac{4}{3}\sqrt{1+\|\stt\|^2}.\label{eq:upperboundeq5}
\end{eqnarray}
\end{enumerate}

Combining this with \eqref{eq:upperboundeq3}, we obtain
\begin{eqnarray}
\|\pa{t+1}\|^2&\leq& \frac{16}{9}(1+\|\stt\|^2)+\frac{\|\stt\|^2}{4}\n\\
&=& \frac{16}{9}+\frac{73}{36}\|\stt\|^2\label{eq:upperboundeq6}
\end{eqnarray}
Therefore combining \eqref{eq:upperboundeq4} and \eqref{eq:upperboundeq6}, we have
\begin{eqnarray}
\|\pa{t}\|^2\leq \max\{\|\pa{0}\|^2, \frac{2}{\pi}+\frac{\|\stt\|^2}{2}, \frac{16}{9}+\frac{73}{36}\|\stt\|^2\}\ =\ c_{U, 1}^2<\infty, \forall t\geq 0\n
\end{eqnarray}
So far we have bounded $\{\|\pa{t}\|\}_{t\geq 0}$ by $c_{U, 1}$. Also, in \eqref{eq:upperboundeq3} we obtained an upper bound for $\tpai{{\langle t\rangle, 1}, t+1}$. Our next step is to obtain an upper bound for $\tpbi{{\langle t\rangle, 1}, t+1}$. First note that
\[
 \tpbi{{\langle t\rangle, 1}, t+1}=\frac{\tpqi{{\langle t\rangle, 1}, t+1}}{2\pp{t+1}(1-\pp{t+1})}.
 \]
 Hence, we have to find an upper bound for $\funcQ$ and a lower bound for $\funcP$. Note that
\begin{eqnarray}
\frac{\partial \funcP(\tpai{1, t}, \|\pb{t}\|, \ptt{1, t})}{\partial
\tpai{1, t}}=-\int\frac{2\|\pb{t}\|}{(e^{y\|\pb{t}\|-\tpai{1, t}\|\pb{t}\|}+e^{-y\|\pb{t}\|+\tpai{1, t}\|\pb{t}\|})^2}\pdf{, y, \ptt{1, t}}
\dif y\leq0.\n
\end{eqnarray}
Therefore $\pp{t+1}=\funcP(\tpai{1, t}, \|\pb{t}\|, \ptt{1, t})$ is a decreasing function of $\tpai{1, t}$. Since $\tpai{1, t}\leq \|\pa{t}\|\leq c_{U, 1}$, with Lemma \ref{lem:Plowerbound}, we have $\forall t\geq 0$
\begin{eqnarray}
\pp{t+1}&=&\funcP(\tpai{1, t}, \|\pb{t}\|, \ptt{1, t}) \geq \funcP(c_{U, 1}, \|\pb{t}\|, \ptt{1, t})\n\\
&\geq&\min\{\frac{1}{4}, \frac{1}{2}(1-\Phi(c_{U, 1}-\ptt{1, t}))+\frac{1}{2}(1-\Phi(c_{U, 1}+\ptt{1, t})) \}\n\\
&\geq&\frac{1}{4}(1-\Phi(c_{U, 1}+\|\stt\|)) \triangleq c_{U, 2}>0.\n
\end{eqnarray}
Note that in \eqref{eq:upperboundeq5}, we derived an upper bound for $\funcQ(\tpai{1, t}, \|\pb{t}\|, \ptt{1, t})$. Therefore, we have
\begin{eqnarray}
0&\leq& \tpbi{{\langle t\rangle, 1}, t+1}\ =\ \frac{\tpqi{{\langle t\rangle, 1}, t+1}}{2\pp{t+1}(1-\pp{t+1})}\n\\
 &\leq&\frac{1}{2c_{U, 2}(1-c_{U, 2})}\funcQ(\tpai{1, t}, \|\pb{t}\|, \ptt{1, t})\ \leq\  \frac{1}{2c_{U, 2}(1-c_{U, 2})}\sqrt{1+\|\stt\|^2}\n
\end{eqnarray}
Thus with \eqref{eq:upperboundeq3}, we have
\begin{eqnarray}
\|\pb{t}\|^2\ \leq\ \max\{\|\pb{0}\|^2, \|\stt\|^2+\frac{1}{4c_{U, 2}^2(1-c_{U, 2})^2}(1+\|\stt\|^2)\}\ =\ c_{U, 3}^2\ <\ \infty, \forall t\geq 0.\n
\end{eqnarray}
This completes the proof of Lemma \ref{lemma:upperboundEMestimates}.

\subsubsection{Proof of  Lemma \ref{lem:lowerboundbt}}\label{sec:proof:lowerboundbt}

Without loss of generality we only consider the case $\langle \pb{t}, \stt \rangle >0$ and $\langle \pa{t}, \pb{t} \rangle\geq 0$. Before we start the proof we remind the reader a couple of facts that we have proved in Lemma \ref{lem:planarestimates}
and Lemma \ref{lemma:upperboundEMestimates}.
\begin{enumerate}
\item $\sup_t \|\pa{t}\| \leq c_{U, 1}$ and $\sup_t \|\pb{t}\| \leq c_{U, 3}$.
\item $\pb{1}, \pb{2}, \ldots, \pb{t}, \ldots$ and $\stt$ are all on the same two-dimensional plane:
\item The angle $\pbeta{t} \triangleq \arccos \left( \frac{ \langle \pb{t}, \stt \rangle}{\|\pb{t}\|\|\stt\|} \right)$ is non-increasing in terms of $t$.
\item $\pa{t}$ is in the same direction as $\pb{t}$ for all $ t\geq 1$.
\end{enumerate}

We use the notations we summarized in Appendix \ref{sec:notation}. Similar to that section we consider the $\vU_t$ transformed vectors $\tpa{t}$, $\tpb{t}$.
Note that we have
\begin{eqnarray}\label{eq:lbbt}
\|\pb{t+1}\|& =& \|\tpb{t+1}\|\ \geq\ \tpbi{{\langle t\rangle, 1}, t+1}\n\\
 &=& \frac{\funcQ (\tpai{1, t}, \| \tpb{t}\|, \ptt{1, t})}{ 2 \funcP(\tpai{1, t}, \| \pb{t}\|, \ptt{1, t}) (1- \funcP(\tpai{1, t}, \| \pb{t}\|, \ptt{1, t}))}\n\\
 &\geq& 2 \funcQ (\tpai{1, t}, \| \pb{t}\|, \ptt{1, t}),
\end{eqnarray}
where $\funcQ$ and $\funcP$ are defined in \eqref{eq:Qdef} and \eqref{eq:Pdef}. Hence, the goal of the rest of the proof is to show that:
\[
2 \funcQ (\tpai{1, t}, \| \pb{t}\|, \ptt{1, t}) \geq \min \{\| \pb{t}\|, c_l \}.
\]
The main idea of this part is as follows.
First note that
\[
\left. \frac{\partial \funcQ (x_a, x_b, x_\theta)}{\partial x_b} \right|_{x_b=0} = \frac{1+ |x_\theta|^2}{2}.
\]
Hence, intuitively speaking we can argue that there exists a neighborhood of $x_b=0$ on which the derivative is always larger than $0.5$. Hence, when $\|\pb{t}\|$ belongs to this neighborhood, $\| \pb{t+1}\|$ is larger than $\|\pb{t}\|$ and cannot go to zero.  Next lemma justifies this claim.

\begin{lemma}\label{lem:lowerboundderivative}
For $\ptt{1, 0}>0$ there exists a value $\delta_b$ only depending on $c_{U, 1}$, $\ptt{1, 0}$ and $\|\stt\|$ such that
\[
\inf_{0 \leq x_a \leq c_{U, 1}, 0 \leq x_b \leq \delta_b, \ptt{1, 0} \leq x_\theta \leq \|\stt\| } \frac{\partial \funcQ (x_a, x_b, x_\theta)}{\partial x_b} \geq 1/2.
\]
\end{lemma}
We present the proof of this result in the Appendix \ref{sec:proofoflemmalowerboundderivative}. We remind the reader that according to Lemma \ref{lemma:upperboundEMestimates}, $\|\pa{t}\| \leq c_{U, 1}$. Furthermore, since the angle $\pbeta{t}$ is a non increasing sequence we have $\ptt{1, t}\geq \ptt{1, 0}.$
Suppose that $\|\pb{t}\|\leq \delta_b$. Then from \eqref{eq:lbbt} we know that $\|\pb{t+1}\| \geq 2\funcQ (\tpai{1, t}, \|\pb{t}\|, \ptt{1, t})$. Also from the mean value theorem we have:
\[
|\funcQ (\tpai{1, t}, \|\pb{t}\|, \ptt{1, t}) - \funcQ(\tpai{1, t}, 0, \ptt{1, t})|\ =\  \frac{\partial \funcQ(\tpai{1, t}, x_{b}, \ptt{1, t})}{\partial x_b}|_{x_b=\xi} \|\pb{t}\|\ \geq\ \frac{1}{2}\|\pb{t}\|,
\]
where $\xi\in [0, \|\pb{t}\|]$ and to obtain the last inequality we used Lemma \ref{lem:lowerboundderivative} and the fact that $\|\pb{t}\| \leq \delta_b$.

So far we have proved that if $\|\pb{t}\|\leq \delta_b$, then $\|\pb{t+1}\| \geq \|\pb{t}\|$. But, we have not ruled out the possibility of the situation in which $\|\pb{t}\| \geq \delta_b$, but $\|\pb{t+1}\|$ is close to zero. That requires a simple continuity argument. Note that since $\funcQ$ is a continuous function of all its variables, its infimum over a compact set is achieved at certain point. Since, the value of $\funcQ(x_1, x_b, x_\theta)$ is only zero when $x_b = 0$, we conclude that the infimum is not zero. Hence, we conclude that
\[
c_l\ \triangleq\ \inf_{0 \leq x_a \leq c_{U, 1}, \delta_b \leq x_b \leq c_{U, 3}, \ptt{1, 0} \leq x_\theta \leq \|\theta\| } 2\funcQ (x_a, x_b, x_\theta)\ >\ 0.
\]
Hence, we have if $\|\pb{t}\|\geq \delta_b$, then
\[\|\pb{t+1}\|\ \geq\  2 \funcQ (\tpai{1, t}, \| \pb{t}\|, \ptt{1, t})\ \geq\ c_l.\]
Therefore combining the result of $\|\pb{t}\|\leq \delta_b$ and $\|\pb{t}\|\geq \delta_b$, we know Lemma \ref{lem:lowerboundbt} holds.

\subsubsection{Proof of Lemma \ref{lemma:aconvergeupperbound}}\label{proof:lemubkappa_a}
We consider three cases and deal with them separately: (i) $0<x_a<\theta$, (ii) $x_a\geq \theta$, (iii) $x_a=0$.

\begin{itemize}
\item[(i)]  $0<x_a<\theta$: Let $x_1\triangleq\theta+x_a>\theta-x_a\triangleq x_2>0$. We first simplify the left hand side of the inequality. Our main goal in this section is to derive sharp upper bounds for $\funcQ(x_a, x_b, \theta)$ and $1- 2\funcP(x_a, x_b, \theta)$. We start with  $\funcQ(x_a, x_b, \theta)$. Note that
\begin{eqnarray}
\lefteqn{\funcQ(x_a, x_b, \theta)=\int \w(y-x_a, x_b)y\pdf{, y, \theta} \dif y}\n\\
&=&x_a \funcP(x_a, x_b, \theta)+\int(\w(y-x_a, x_b)-\frac{1}{2}+\frac{1}{2})(y-x_a )\pdf{, y, \theta}
\dif y\n\\
&=&x_a \funcP(x_a, x_b, \theta)-\frac{1}{2}x_a +\int(\w(y-x_a, x_b)-\frac{1}{2})(y-x_a )\pdf{, y, \theta}
\dif y\n\\
&=&x_a \funcP(x_a, x_b, \theta)-\frac{1}{2}x_a
+\frac{1}{2}\int(\w(y-x_a, x_b)-\frac{1}{2})(y-x_a)(\phi(y-\theta)+\phi(y+\theta)) \dif y \n\\
&=&x_a \funcP(x_a, x_b, \theta)-\frac{1}{2}x_a +\frac{1}{4}(F(x_b, \theta-x_a)+F(x_b, \theta+x_a))\n\\
&=&x_a \funcP(x_a, x_b, \theta)-\frac{1}{2}x_a +\frac{1}{4}(F(x_b, x_1)+F(x_b, x_2)), \label{eq:Q1upper1}
\end{eqnarray}
where $F$ is defined in \eqref{eq:Fdefinition}.
Next, we find an upper bound for $\frac{1}{2}(F(x_b, x_1)+F(x_b, x_2))$. Note that $\forall x_b\geq 0, x_\theta\geq 0$, we have
\begin{eqnarray}
F(x_b, x_\theta)&=
  &\int\frac{e^{yx_{b}}-e^{-yx_{b}}}{e^{yx_{b}}+e^{-yx_{b}}}y\frac{1}{\sqrt{2\pi}}e^{-(y-x_\theta)
^2/2} \dif y\n\\
&\leq&\int_{0}^{\infty}y\frac{1}{\sqrt{2\pi}}e^{-(y-x_\theta)^2/2} \dif
y+\int_{0}^{\infty}y\frac{1}{\sqrt{2\pi}}e^{-(y+x_\theta)^2/2} \dif y\n\\
&=&x_\theta\int_{-x_\theta}^{x_\theta}\frac{1}{\sqrt{2\pi}}e^{-y^2/2} \dif
y+\int_{-x_\theta}^{\infty}y\frac{1}{\sqrt{2\pi}}e^{-y^2/2} \dif
y+\int_{x_\theta}^{\infty}y\frac{1}{\sqrt{2\pi}}e^{-y^2/2} \dif y\n\\
&\leq&x_\theta(1-2\Phi(-x_\theta))+2\phi(x_\theta)\ \triangleq\ l(x_\theta).\n
\end{eqnarray}
Therefore, if we plug in $x_\theta=x_1$ and $x_\theta=x_2$, we have
\begin{eqnarray}
\frac{1}{2}(F(x_b, x_1)+F(x_b, x_2))\ \leq\ \frac{1}{2}(l(x_1)+l(x_2))\ \leq\ l(x_1), \label{eq:Q1upper2}
\end{eqnarray}
where the last inequality holds since $l(x)$ is an increasing function. This can be proved by taking the derivative of $l(x)$:
\begin{eqnarray}
\frac{\dif l(x_\theta)}{\dif x_\theta}\ =\ 1-2\Phi(-x_\theta)+2x_\theta\phi(x_\theta)-2x_\theta\phi(x_\theta)\ =\ 1-2\Phi(-x_\theta)\ \geq\ 0.\n
\end{eqnarray}
Combining \eqref{eq:Q1upper1} and \eqref{eq:Q1upper2} we obtain
\begin{eqnarray}\label{eq:Q1upper3}
\funcQ(x_a, x_b, \theta) \leq x_a \funcP(x_a, x_b, \theta)-\frac{1}{2}x_a + 2 l(x_1).
\end{eqnarray}

Now we obtain an upper bound for $1-2\funcP(x_a, x_b, \theta)$. Note that,
\begin{eqnarray}
1-2\funcP(x_a, x_b, \theta)&=&\int(\frac{1}{2}-\frac{e^{yx_b}}{e^{yx_b}+e^{-yx_b}})\frac{1}{\sqrt{2\pi}}(e^{-(y+x_a-\theta)^2/2}+e^{-(y+x_a+\theta)^2/2})
  \dif y\n\\
&=&\int(\frac{1}{2}-\frac{e^{yx_b}}{e^{yx_b}+e^{-yx_b}})\frac{1}{\sqrt{2\pi}}(e^{-(y-x_2)^2/2}+e^{-(y+x_1)^2/2})
  \dif y\n\\
&=&\int\frac{e^{-yx_b}-e^{yx_b}}{2(e^{yx_b}+e^{-yx_b})}\frac{1}{\sqrt{2\pi}}(e^{-(y-x_2)^2/2}+e^{-(y+x_1)^2/2})
  \dif y\n\\
&=&\int\frac{e^{yx_b}-e^{-yx_b}}{2(e^{yx_b}+e^{-yx_b})}\frac{1}{\sqrt{2\pi}}(e^{-(y-x_1)^2/2}-e^{-(y-x_2)^2/2})
  \dif y\n\\
&=&K(x_1, x_b)-K(x_2, x_b), \label{eq:defK}
\end{eqnarray}
where $K(x, b)\triangleq
\int\frac{e^{yb}-e^{-yb}}{2(e^{yb}+e^{-yb})}\frac{1}{\sqrt{2\pi}}e^{-(y-x)^2/2} \dif y$. The following lemma proved in the Appendix \ref{sec:proofoflemmaconcavityK} summarizes some of the nice properties of this function, which will be used later in our proof.

\begin{lemma}\label{lem:concavityK}
$K(x, b)$ is a concave, strictly increasing function of $x$. Furthermore, $K(0, x_b) =0$.
\end{lemma}

Given \eqref{eq:Q1upper3} and \eqref{eq:defK} we can now prove the claimed upper bound in Lemma \ref{lemma:aconvergeupperbound}. We have

\begin{eqnarray}
\lefteqn{\frac{\funcQ(x_a, x_b, \theta)(1-2\funcP(x_a, x_b, \theta))}{2\funcP(x_a, x_b, \theta)(1-\funcP(x_a, x_b, \theta))} } \nonumber \\
&=&\frac{\{x_a \funcP(x_a, x_b, \theta)-\frac{1}{2}x_a +\frac{1}{4}(F(x_b , x_1)+F(x_b , x_2))\}(1-2\funcP(x_a, x_b, \theta))}{2\funcP(x_a, x_b, \theta)(1-\funcP(x_a, x_b, \theta))}\n\\
&=&x_a+\frac{\frac{1}{2}(F(x_b, x_1)+F(x_b, x_2))(1-2\funcP(x_a, x_b, \theta))-x_a}{4\funcP(x_a, x_b, \theta)(1-\funcP(x_a, x_b, \theta))}\n\\
&=&x_a+\frac{\frac{1}{2}(F(x_b, x_1)+F(x_b, x_2))(K(x_1, x_b)-K(x_2, x_b))-\frac{x_1-x_2}{2}}{4\funcP(x_a, x_b, \theta)(1-\funcP(x_a, x_b, \theta))} \nonumber \\
 &\leq& x_a+\frac{l(x_1)(K(x_1, x_b)-K(x_2, x_b))-\frac{1}{2}(x_1-x_2)}{4\funcP(x_a, x_b, \theta)(1-\funcP(x_a, x_b, \theta))}.\label{eq:ata1}
\end{eqnarray}
It is straightforward to use the concavity of $K(x, x_b)$ in terms of $x$ and prove that the function $\frac{K(x_1, x_b) -K(x_2, x_b)}{x_1-x_2}$ is a decreasing function of $x_2$. Hence, it is maximized at $x_2=0$. Since $K(0, x_b)=0$, proved in Lemma \ref{lem:concavityK}, we have
\begin{equation}\label{eq:LminusKbound}
K(x_1, x_b) - K(x_2, x_b) \leq \frac{K(x_1, x_b)}{x_1} (x_1-x_2).
\end{equation}
Combining \eqref{eq:ata1} and \eqref{eq:LminusKbound} implies:
 \begin{eqnarray}
\frac{\funcQ(x_a, x_b, \theta)(1-2\funcP(x_a, x_b, \theta))}{2\funcP(x_a, x_b, \theta)(1-\funcP(x_a, x_b, \theta))} &\leq& x_a+\frac{(l(x_1)\frac{K(x_1, x_b)}{x_1}-\frac{1}{2})(x_1-x_2)}{4\funcP(x_a, x_b, \theta)(1-\funcP(x_a, x_b, \theta))}\n\\
&=&x_a+\frac{(l(x_1)\frac{K(x_1, x_b)}{x_1}-\frac{1}{2})x_a}{2\funcP(x_a, x_b, \theta)(1-\funcP(x_a, x_b, \theta))}\n\\
&=&(1+\frac{l(x_1)\frac{K(x_1, x_b)}{x_1}-\frac{1}{2}}{2\funcP(x_a, x_b, \theta)(1-\funcP(x_a, x_b, \theta))})x_a.\label{eq:ata2}
\end{eqnarray}
Our next step is to find an upper bound for $(l(x_1)\frac{K(x_1, x_b)}{x_1}-\frac{1}{2})$. Note that
\begin{eqnarray}
K(x_1, x_b)&=&\int\frac{e^{yx_b}-e^{-yx_b}}{2(e^{yx_b}+e^{-yx_b})}\frac{1}{\sqrt{2\pi}}e^{-(y-x_1)^2/2}
  \dif y\n\\
&=&\int_{0}^{\infty}\frac{1}{2}\frac{1}{\sqrt{2\pi}}e^{-(y-x_1)^2/2} \dif
  y-\int_{0}^{\infty}\frac{e^{-yx_b}}{(e^{yx_b}+e^{-yx_b})}\frac{1}{\sqrt{2\pi}}e^{-(y-x_1)^2/2}
  \dif y\n\\
&~&-\int_{0}^{\infty}\frac{1}{2}\frac{1}{\sqrt{2\pi}}e^{-(y+x_1)^2/2} \dif
  y+\int_{0}^{\infty}\frac{e^{-yx_b}}{(e^{yx_b}+e^{-yx_b})}\frac{1}{\sqrt{2\pi}}e^{-(y+x_1)^2/2}
  \dif y\n\\
&\leq&\int_{0}^{x_1}\frac{1}{\sqrt{2\pi}}e^{-y^2/2} \dif y\ =\ \frac{1}{2}-\Phi(-x_1).\n
\end{eqnarray}
Finally to obtain an upper bound  for $\frac{1}{x}l(x)(\frac{1}{2}-\Phi(-x))$ we use the following lemma:
\begin{lemma}\label{lem:lupperbound}
Define $l(x)\triangleq x(1-2\Phi(-x))+2\phi(x)$, then for all $ x>0$, we have
\[
\frac{l(x)(\frac{1}{2}-\Phi(-x))}{x}<\frac{1}{2}.
\]
\end{lemma}
The proof of this lemma is presented in Appendix \ref{sec:proofoflemmalupperbound}. Using this lemma, we have
$$\frac{2l(x_1)(\frac{1}{2}-\Phi(-x_1))}{x_1}<1, \forall x_1=x_a+\theta\in [\theta, 2\theta].$$
By continuity of the function $\frac{l(x)(\frac{1}{2}-\Phi(-x))}{x}$, we have

$$\bar{\kappa}_a(\theta)=\sup_{x_1\in[\theta, 2\theta]}\frac{2l(x_1)(\frac{1}{2}-\Phi(-x_1))}{x_1}\ <\ 1.$$
It is straightforward to prove that $\bar{\kappa}_a(\theta) $ is a continuous function of $\theta\in (0, \infty)$.
 Since $4\funcP(x_a, x_b, \theta)(1-\funcP(x_a, x_b, \theta))\leq 1$,
we can bound \eqref{eq:ata2} in the following way:
\begin{eqnarray}
\frac{\funcQ(x_a, x_b, \theta)(1-2\funcP(x_a, x_b, \theta))}{2\funcP(x_a, x_b, \theta)(1-\funcP(x_a, x_b, \theta))}&\leq&(1+\frac{l(x_1)\frac{\frac{1}{2}-\Phi(-x_1)}{x_1}-\frac{1}{2}}{2\funcP(x_a, x_b, \theta)(1-\funcP(x_a, x_b, \theta))})x_a\n\\
&\leq&(1+\frac{\bar{\kappa}_a(\theta)-1}{4\funcP(x_a, x_b, \theta)(1-\funcP(x_a, x_b, \theta))})x_a\n\\
&\leq&\bar{\kappa}_a(\theta)x_a, \ \ \forall 0<x_a<\theta, x_b>0.\label{eq:finalboundcasei}
\end{eqnarray}
This completes the proof of part (i).

\item[(ii)] $x_a\geq \theta$: According to Lemma \ref{lemma:sign}, $1-2\funcP(x_a, x_b, \theta)\geq 0$. Together with Lemma \ref{lem:Q1overPbound} we have
\begin{eqnarray}
\frac{\funcQ(x_a, x_b, \theta)(1-2\funcP(x_a, x_b, \theta))}{2\funcP(x_a, x_b, \theta)(1-\funcP(x_a, x_b, \theta))}&\leq&\frac{(x_a+\sqrt{\frac{2}{\pi}})(1-2\funcP(x_a, x_b, \theta))}{2(1-\funcP(x_a, x_b, \theta))}\n\\
&=&x_a+\frac{\sqrt{\frac{2}{\pi}}(1-2\funcP(x_a, x_b, \theta))-x_a}{2(1-\funcP(x_a, x_b, \theta))}.\label{eq:ata4}
\end{eqnarray}
From Lemma \ref{lem:Plowerbound}, we have
\begin{eqnarray}
1-2\funcP(x_a, x_b, \theta)\leq\Phi(x_a+\theta)+\Phi(x_a-\theta)-1.\label{eq:ata5}
\end{eqnarray}
Note that
\begin{eqnarray}
\frac{\partial (\Phi(x_a+\theta)+\Phi(x_a-\theta)-1)}{\partial x_a} \ = \ \phi(x_a+\theta)+\phi(x_a-\theta) \ \leq \ \sqrt{\frac{2}{\pi}}, \forall x_a, \n
\end{eqnarray}
and
\[\Phi(0+\theta)+\Phi(0-\theta)-1=0.\]
Therefore, from \eqref{eq:ata5} and mean value theorem, we have
\begin{eqnarray}
1-2\funcP(x_a, x_b, \theta)\leq \sqrt{\frac{2}{\pi}}x_a.\n
\end{eqnarray}
Together with \eqref{eq:ata4} and $\funcP(x_a, x_b, \theta))\leq \frac{1}{2}$, we have
\begin{eqnarray}
\frac{\funcQ(x_a, x_b, \theta)(1-2\funcP(x_a, x_b, \theta))}{2\funcP(x_a, x_b, \theta)(1-\funcP(x_a, x_b, \theta))}&\leq&x_a+\frac{\sqrt{\frac{2}{\pi}}(1-2\funcP(x_a, x_b, \theta))-x_a}{2(1-\funcP(x_a, x_b, \theta))}\n\\
&\leq&x_a-\frac{1-\frac{2}{\pi}}{2(1-\funcP(x_a, x_b, \theta))}x_a\n\\
&\leq&(\frac{1}{2}+\frac{1}{\pi})x_a. \label{eq:finalboundcaseii}
\end{eqnarray}

\item[(iii)] $x_a=0$: It is straightforward to prove that $\funcP(0, x_b, \theta)=\frac{1}{2}$. Hence,
\[
\frac{\funcQ(x_a, x_b, \theta)(1-2\funcP(x_a, x_b, \theta))}{2\funcP(x_a, x_b, \theta)(1-\funcP(x_a, x_b, \theta))} =0.
\]
\end{itemize}
Combining Case (i), (ii), (iii), we conclude that if we define $$\kappa_a(\theta)=\left\{\begin{aligned}
			& \max(\bar{\kappa}_a(\theta), \frac{1}{2}+ \frac{1}{\pi}),
			& &\theta>0 \\
			& \frac{1}{2}+ \frac{1}{\pi},
			& & \theta=0\\
		\end{aligned}\right. , $$ then the statement of Lemma \ref{proof:lemubkappa_a} holds.

\subsection{Proof of Theorem \ref{thm:convergenceratemagnitude}}\label{sec:prooftheomeconvergenceratemag}

It is straightforward to use Theorem \ref{thm:convergencerate1}  and show that for every $\delta_a>0$, there exists a value of $T_{\delta_a}$ such that for every $t> T_{\delta_a}$, $\|\pa{t}\|\leq \delta_a$. For the moment suppose that the following claim is true: there exists $\de_a>0, \kappa_b, c_b$ only depending on $\|\stt\|$, $|\ptt{1, 0}|$ and the initialization $\{\|\pa{0}\|, \|\pb{0}\|\}$, such that if $\|\pa{t}\|\leq \de_a$ for some $t$, then the next iteration $\pb{t+1}$ satisfies the following equation:
\begin{eqnarray}
\|\pb{t+1}-\stt\|^2&\leq&\kappa_{b}^2\|\pb{t}-\stt\|^2+c_{b}\|\pa{t}\|.\n
\end{eqnarray}
If we combine this claim with the result of Theorem \ref{thm:convergencerate1}, we obtain Theorem \ref{thm:convergenceratemagnitude}. Hence, the problem reduces to proving the above claim.

Note that in Lemma \ref{lemma:upperboundEMestimates}, Lemma \ref{lem:lowerboundbt} and Lemma \ref{lem:planarestimates}, we have
\[\|\pa{t}\|\in[0, c_{U, 1}], \|\pb{t}\|\in [c_{L, 1}, c_{U, 3}], \ptt{1, t}\in [\ptt{1, 0}, \|\stt\|], \forall t\geq 0.\]
Therefore, it is again straightforward to see that the following lemma implies our claim:

\begin{lemma}\label{lemma:conratepre1}
For any $\pa{t}, \pb{t}, \stt\in \bbR^2$, if $\|\pa{t}\|\in [0, U_a], \|\pb{t}\|\in [L_b, U_b], \frac{\dotp{\stt, \pb{t}}}{\|\pb{t}\|}\in [L_\theta, \|\stt\|], \forall t\geq 0$, where $L_b>0, L_\theta>0$ then
there exists $\de_a\in (0, \min\{L_\theta, 1\}];\kappa_b\in (0, 1);c_b>0$ such that $\forall \|\pa{t}\|\in [0, \de_a], \|\pb{t}\|\in [L_b, U_b], \frac{\dotp{\stt, \pb{t}}}{\|\pb{t}\|}\in [L_\theta, \|\stt\|], $ the next iteration $\pb{t+1}$ satisfying
\begin{eqnarray}
\|\pb{t+1}-\stt\|^2&\leq&\kappa_{b}^2\|\pb{t}-\stt\|^2+c_{b}\|\pa{t}\|, \n
\end{eqnarray}
where
$\de_a, \kappa_b$ and $c_b$ are functions of only $U_a, L_b, U_b, L_\theta, \|\stt\|$.
\end{lemma}
To prove this lemma, we use the notations and definition that are summarized in Appendix \ref{sec:notation}. In particular, we use the rotation matrix $\vU_t$ introduced in that section and rotate all the vectors $a$, $b$, $\theta$, $q$ with this matrix.
 Note that according to Lemma \ref{lemma:projection} we know that $\pb{t+1}$ lies in the span of $\stt$ and $\pb{t}$ and hence, $\tpb{{\langle t\rangle, i}, t+1}=0$ for $i \geq 3$. Therefore we only need to consider the first two coordinates.

Our strategy of proving this lemma is to prove the following two claims for the first two coordinates:

\begin{enumerate}

\item There exists $\kappa_s \in (0, 1)$ such that $|\tpbi{{\langle t\rangle, 2}, t+1}-\ptt{2, t}| \leq \kappa_s |\ptt{2, t}|$.
\item There exists $\kappa'_b\in (0, 1)$ and $\de_a>0$ such that if $\|\pa{t}\|\leq \de_a$, then $$|\tpbi{{\langle t\rangle, 1}, t+1} - \ptt{1, t}|  \leq \kappa'_b |\|\pb{t}\|-\ptt{1, t}|+ (16 \|\theta\|+6) \|\pa{t}\|.$$
\end{enumerate}
 We will then combine the above two claims to obtain Lemma \ref{lemma:conratepre1}.
 \begin{enumerate}
 \item Proof of  $|\tpbi{{\langle t\rangle, 2}, t+1}-\ptt{2, t}| \leq \kappa_s |\ptt{2, t}|$: To prove our first claim, first note that according to \eqref{eq:tildeiterationab2} and \eqref{eq:q2thetaS} we have
\begin{eqnarray}
\tpbi{{\langle t\rangle, 2}, t+1} &=& \frac{\tpqi{{\langle t\rangle, 2}, t+1}}{ 2 \funcP (\tpai{1, t}, \|\pb{t}\|, \ptt{1, t} ) (1-\funcP(\tpai{1, t}, \|\pb{t}\|, \ptt{1, t}))} \nonumber \\
&=&\ptt{2, t}\frac{S(\tpai{1, t}, \|\pb{t}\|, \ptt{1, t} )}{2 \funcP (\tpai{1, t}, \|\pb{t}\|, \ptt{1, t} ) (1-\funcP(\tpai{1, t}, \|\pb{t}\|, \ptt{1, t}))}\n
\end{eqnarray}
Hence
\begin{eqnarray}
|\tpbi{{\langle t\rangle, 2}, t+1}-\ptt{2, t}|
&=&|\ptt{2, t}|(1-\frac{S(\tpai{1, t}, \|\pb{t}\|, \ptt{1, t} )}{2 \funcP (\tpai{1, t}, \|\pb{t}\|, \ptt{1, t} ) (1-\funcP(\tpai{1, t}, \|\pb{t}\|, \ptt{1, t}))})\n\\
&\leq&|\ptt{2, t}|(1-2S(\tpai{1, t}, \|\pb{t}\|, \ptt{1, t} ))\n
\end{eqnarray}
By definition of function $S$, it is straightforward to conclude that $S(\tpai{1, t}, \|\pb{t}\|, \ptt{1, t})$ is only zero if $\|\pb{t}\|=0$ and can only go to zero if $\tpai{1, t} \rightarrow \infty$ or $\|\pb{t}\| \rightarrow \infty$. Therefore, combined with the continuity of $S$, we conclude that
\begin{eqnarray}\label{eq:upperboundforS}
\kappa_s \triangleq \sup_{\tpai{1, t} \in [0, U_a] , \|\pb{t}\| \in [L_b, U_b], \ptt{1, t} \in [L_\theta, \|\stt\|]} 1-2S(\tpai{1, t}, \|\pb{t}\|, \ptt{1, t})<1,
\end{eqnarray}
where $\kappa_s$ only depends on $U_a, L_b, U_b, L_\theta$ and $\|\stt\|$. Hence,
\begin{equation}\label{eq:boundtheta2b2}
|\tpbi{{\langle t\rangle, 2}, t+1}-\ptt{2, t}| \leq \kappa_s |\ptt{2, t}|.
\end{equation}
\item Proof of $|\tpbi{{\langle t\rangle, 1}, t+1} - \ptt{1, t}|  \leq \kappa_b' |\|\pb{t}\|-\ptt{1, t}|+ (16 \|\theta\|+6) \|\pa{t}\|$: Note that
\begin{eqnarray}
|\tpbi{{\langle t\rangle, 1}, t+1}- \ptt{1, t}| &=& \left| \frac{\funcQ(\tpai{1, t}, \|\pb{t}\|, \ptt{1, t})}{ 2 \funcP(\tpai{1, t}, \|\pb{t}\|, \ptt{1, t}) (1- \funcP(\tpai{1, t}, \|\pb{t}\|, \ptt{1, t}))} - \ptt{1, t} \right| \nonumber \\
&=&  \left| \frac{2\funcQ(\tpai{1, t}, \|\pb{t}\|, \ptt{1, t})- \ptt{1, t} + \ptt{1, t} (2\funcP(\tpai{1, t}, \|\pb{t}\|, \ptt{1, t})-1)^2}{ 4 \funcP(\tpai{1, t}, \|\pb{t}\|, \ptt{1, t}) (1- \funcP(\tpai{1, t}, \|\pb{t}\|, \ptt{1, t}))} \right| \nonumber \\
&\leq&  \left| \frac{2\funcQ(\tpai{1, t}, \|\pb{t}\|, \ptt{1, t})- \ptt{1, t}}{ 4 \funcP(\tpai{1, t}, \|\pb{t}\|, \ptt{1, t}) (1- \funcP(\tpai{1, t}, \|\pb{t}\|, \ptt{1, t}))} \right| \nonumber \\
&&+  \left| \frac{\ptt{1, t} (2\funcP(\tpai{1, t}, \|\pb{t}\|, \ptt{1, t})-1)^2}{ 4 \funcP(\tpai{1, t}, \|\pb{t}\|, \ptt{1, t}) (1- \funcP(\tpai{1, t}, \|\pb{t}\|, \ptt{1, t}))} \right|\label{eq:con1}
\end{eqnarray}
Furthermore in \eqref{eq:Q1upper1}, we proved that
\[\funcQ(x_{a}, x_{b}, x_{\theta})=x_a \funcP(x_a, x_b, \theta)-\frac{1}{2}x_a +\frac{1}{4}(F(x_b , x_a+x_\theta)+F(x_b , x_\theta-x_a)).\]
To see the definitions of $\funcP, \funcQ, $ and $F$ you may refer to Appendix \ref{sec:notation}. Hence, we have
\begin{eqnarray}
&&|2\funcQ(\tpai{1, t}, \|\pb{t}\|, \ptt{1, t})- \ptt{1, t}|\n\\
&=&|\tpai{1, t}(2\funcP(\tpai{1, t}, \|\pb{t}\|, \ptt{1, t})-1)+\frac{1}{2}(F(\|\pb{t}\|, \ptt{1, t}-\tpai{1, t})-F(\|\pb{t}\|, \ptt{1, t}))\n\\
&~&+\frac{1}{2}(F(\|\pb{t}\|, \ptt{1, t}+\tpai{1, t})-F(\|\pb{t}\|, \ptt{1, t}))+(F(\|\pb{t}\|, \ptt{1, t})-\ptt{1, t})|.\label{eq:con2}
\end{eqnarray}
Combining \eqref{eq:con1} and \eqref{eq:con2}, we conclude that in order to obtain an upper bound for $|\tpbi{{\langle t\rangle, 1}, t+1}- \ptt{1, t}|$ we have to find the following bounds:
\begin{enumerate}
\item Obtain an upper bound for $|2\funcP(\tpai{1, t}, \|\pb{t}\|, \ptt{1, t})-1|$.

\item Obtain an upper bound for $|F(\|\pb{t}\|, x_\theta)-F(\|\pb{t}\|, \ptt{1, t})|$ for all $ \ptt{1, t}\in[L_\theta, \|\stt\|]$ and $|x_\theta-\ptt{1, t}|\leq L_\theta$.

\item Obtain an upper bound for $|F(\|\pb{t}\|, \ptt{1, t})-\ptt{1, t}|$ for all $ \ptt{1, t}\in[L_\theta, \|\stt\|]$ and $\|\pb{t}\|\in [L_b, U_b]$

\item Obtain a lower bound for $4 \funcP(\tpai{1, t}, \|\pb{t}\|, \ptt{1, t}) (1- \funcP(\tpai{1, t}, \|\pb{t}\|, \ptt{1, t}))$.
\end{enumerate}

We summarize our strategy for bounding each of these terms below:
\begin{enumerate}
\item Upper bound for $|2\funcP(\tpai{1, t}, \|\pb{t}\|, \ptt{1, t})-1|$: It is straightforward to confirm $2 \funcP(0, \|\pb{t}\|, \ptt{1, t}) = \frac{1}{2}$. Hence, we have to calculate $|2\funcP(\tpai{1, t}, \|\pb{t}\|, \ptt{1, t})-2\funcP(0, \|\pb{t}\|, \ptt{1, t})|$. According to mean value theorem
\begin{eqnarray}\label{eq:2p-1bound1}
|\funcP(\tpai{1, t}, \|\pb{t}\|, \ptt{1, t})-\funcP(0, \|\pb{t}\|, \ptt{1, t})| = \left|\frac{\partial \funcP(x_a, \|\pb{t}\|, \ptt{1, t})}{\partial x_a}|_{x_a=\xi}\right| (\tpai{1, t}), \
\end{eqnarray}
where $\xi \in [0, \tpai{1, t}]$.
Therefore we only need to bound $|\frac{\partial \funcP(x_a, \|\pb{t}\|, \ptt{1, t})}{\partial x_a}|$. Note that
\begin{eqnarray}
&~&\left|\frac{\partial \funcP(x_a, \|\pb{t}\|, \ptt{1, t}))}{\partial x_a}|_{x_a=0}\right| \nonumber \\
&=&\int\frac{2\|\pb{t}\|}{(e^{y\|\pb{t}\| -x_a \|\pb{t}\| }+e^{-y\|\pb{t}\| +x_a \|\pb{t}\|
})^2}\pdf{, y, \ptt{1, t}} \dif y|_{x_a=0}\n\\
&=&\int\frac{2\|\pb{t}\|}{(e^{y\|\pb{t}\|}+e^{-y\|\pb{t}\|})^2}\pdf{, y, \ptt{1, t}} \dif y\n\\
&=&\int\frac{2\|\pb{t}\|}{(e^{y\|\pb{t}\|}+e^{-y\|\pb{t}\|})^2}\phi(y-\ptt{1, t}) \dif y.\label{eq:derP1x_a0}
\end{eqnarray}
Next we show that $\left|\frac{\partial \funcP(x_a, \|\pb{t}\|, \ptt{1, t}))}{\partial x_a}|_{x_a=0}\right|$ is a decreasing function of $\ptt{1, t}$ and hence can be upper bounded by $\left|\frac{\partial \funcP(x_a, \|\pb{t}\|, 0))}{\partial x_a}|_{x_a=0}\right|$:
\begin{eqnarray}
&~&\frac{\partial \int\frac{2\|\pb{t}\|}{(e^{y\|\pb{t}\|}+e^{-y\|\pb{t}\|})^2}\phi(y-\ptt{1, t}) \dif y}{\partial \ptt{1, t}} \nonumber \\
&=&\int\frac{2\|\pb{t}\|}{(e^{y\|\pb{t}\|}+e^{-y\|\pb{t}\|})^2}(y-\ptt{1, t})\phi(y-\ptt{1, t}) \dif y\n\\
&=&-\int\frac{2\|\pb{t}\|}{(e^{y\|\pb{t}\|}+e^{-y\|\pb{t}\|})^2}d\phi(y-\ptt{1, t})\n\\
&=&-\int\frac{4\|\pb{t}\|^2(e^{y\|\pb{t}\|}-e^{-y\|\pb{t}\|})}{(e^{y\|\pb{t}\|}+e^{-y\|\pb{t}\|})^3}\phi(y-\ptt{1, t})
  \dif y \leq 0\n
\end{eqnarray}
Hence,
\begin{eqnarray}
\left|\frac{\partial \funcP(x_a, \|\pb{t}\|, \ptt{1, t}))}{\partial
x_a}|_{x_a=0}\right|&\leq&\int\frac{2\|\pb{t}\|}{(e^{y\|\pb{t}\|}+e^{-y\|\pb{t}\|})^2}\phi(y) \dif y\n\\
&\leq&\int_{0}^{\infty}4\|\pb{t}\| e^{-2y\|\pb{t}\|}\phi(y) \dif y
\leq\sqrt{\frac{2}{\pi}}.\label{eq:P1der00}
\end{eqnarray}
Our next goal is to show that  there exists $\de_1>0$ is a function of only $L_b, U_b, L_\theta, \|\stt\|$ such that
\begin{eqnarray}
\sup_{ \tpai{1, t}\in [0, \de_1], \|\pb{t}\|\in [L_b, U_b], \ptt{1, t}\in [L_\theta, \|\stt\|]} \left|\frac{\partial\funcP(\tpai{1, t}, \|\pb{t}\|, \ptt{1, t})}{\partial \tpai{1, t}} \right|\leq 1.\label{eq:2p-1bound2}
\end{eqnarray}
This is a simple proof by contradiction. Since we have already done similar arguments in the proof of Lemma \ref{lem:lowerboundderivative}, for the sake of brevity we skip this argument. By combining \eqref{eq:2p-1bound1} and \eqref{eq:2p-1bound2} we conclude:
\begin{eqnarray}
|1-2\funcP(\tpai{1, t}, \|\pb{t}\|, \ptt{1, t})|\leq 2\tpai{1, t}, \forall \tpai{1, t}\in [0, \de_1], \|\pb{t}\|\in [L_b, U_b], \ptt{1, t}\in [L_\theta, \|\stt\|].\label{eq:precon1}
\end{eqnarray}

\item Upper bound for $|F(\|\pb{t}\|, x_\theta)-F(\|\pb{t}\|, \ptt{1, t})|$:
Again by employing the mean value theorem, we conclude that we have to bound $\frac{\partial F(\|\pb{t}\|, x_\theta)}{\partial x_\theta}$ in a neighborhood of $x_\theta =\ptt{1, t}$ for all $ \|\pb{t}\|\in [L_b, U_b], \ptt{1, t}\in[L_\theta, \|\stt\|]$. Note that, $\forall x_\theta\geq 0$
\begin{eqnarray}
\left|\frac{\partial F(\|\pb{t}\|, x_\theta)}{\partial x_\theta}\right|
&=&\left|\int(2\w(y, \|\pb{t}\|)-1)y(y-x_\theta)\phi(y-x_\theta) \dif y\right|\n\\
&=&\left|\int(2\w(y, \|\pb{t}\|)-1)\{(y-x_\theta)^2+x_\theta(y-x_\theta)\}\phi(y-x_\theta) \dif y \right|\n\\
&=&\Big|x_\theta\int\frac{e^{y\|\pb{t}\|}-e^{-y\|\pb{t}\|}}{e^{y\|\pb{t}\|}+e^{-y\|\pb{t}\|}}
(y-x_\theta)\phi(y-x_\theta) \dif y\n\\
&~&+\int\frac{e^{y\|\pb{t}\|}-e^{-y\|\pb{t}\|}}{e^{y\|\pb{t}\|}+e^{-y\|\pb{t}\|}}(y-x_\theta)^2\phi(y-x_\theta)
\dif y\Big|\n\\
&\overset{(a)}{<}&\left|x_\theta\int\frac{4\|\pb{t}\|}{(e^{y\|\pb{t}\|}+e^{-y\|\pb{t}\|})^2}\phi(y-x_\theta)
\dif y\right|+1\n\\
&\overset{(b)}{=}&2x_\theta \left|\frac{\partial (\funcP(x_a, \|\pb{t}\|, x_\theta))}{\partial x_a}|_{x_a=0}\right|+1, \n
\end{eqnarray}
where to obtain Inequality (a) we used integration by parts and also the fact that $\frac{e^{y\|\pb{t}\|}-e^{-y\|\pb{t}\|}}{e^{y\|\pb{t}\|}+e^{-y\|\pb{t}\|}}<1$. To see why (b) holds, one may check \eqref{eq:derP1x_a0}.
By employing \eqref{eq:P1der00}, we then conclude that
\begin{eqnarray}
|\frac{\partial F(\|\pb{t}\|, x_\theta)}{\partial x_\theta}|<x_\theta \frac{4}{\sqrt{2\pi}}+1\leq 4\|\stt\|+1, \forall x_\theta\in [0, 2\|\stt\|]. \n
\end{eqnarray}
Therefore, using mean value theorem, we have $\forall |x_\theta-\ptt{1, t}|\leq L_\theta,  \ptt{1, t}\in[L_\theta, \|\stt\|]$,
$$|F(\|\pb{t}\|, x_\theta)-F(\|\pb{t}\|, \ptt{1, t})|\leq(4\|\stt\|+1)|x_\theta-\ptt{1, t}|.$$
Hence, we have $\forall \tpai{1, t}\in [0, L_\theta], \ptt{1, t}\in [L_\theta, \|\stt\|]$,
\begin{eqnarray}
&~&|\frac{1}{2}(F(\|\pb{t}\|, \ptt{1, t}-\tpai{1, t})-F(\|\pb{t}\|, \ptt{1, t}))+\frac{1}{2}(F(\|\pb{t}\|, \ptt{1, t}+\tpai{1, t})-F(\|\pb{t}\|, \ptt{1, t}))|\n\\
&\leq&
(4\|\stt\|+1)\tpai{1, t}.\label{eq:precon2}
\end{eqnarray}

\item Upper bound for $|F(\|\pb{t}\|, \ptt{1, t})-\ptt{1, t}|$:

Because the proof of this part has many algebraic steps we postpone it to the Appendix \ref{sec:proofoflemma1a0conratepre}.
\begin{lemma}
\label{lemma:1a0conratepre}
Given $x_b\in [L_b, U_b], x_\theta\in [L_\theta, \|\stt\|]$ where $0<L_b\leq L_\theta\leq \|\stt\|\leq U_b<\infty $, there exists $\kappa_b''\in (0, 1)$ is a function of only $L_b, U_b, L_\theta, \|\stt\|$ such that
$$|F(x_b, x_\theta)-x_\theta|\leq \kappa_b'' |x_b-x_\theta|, \forall x_b\in [L_b, U_b], x_\theta\in [L_\theta, \|\stt\|].$$
\end{lemma}

\item Lower bound for $4 \funcP(\tpai{1, t}, \|\pb{t}\|, \ptt{1, t}) (1- \funcP(\tpai{1, t}, \|\pb{t}\|, \ptt{1, t}))$:
Note that
\begin{eqnarray}
\frac{1}{4\funcP(0, \|\pb{t}\|, \ptt{1, t})(1-\funcP(0, \|\pb{t}\|, \ptt{1, t}))}-1=0.\n
\end{eqnarray}
Let $\ep_{\funcp}=\min\{\frac{1-\kappa_b''}{2\kappa_b''}, 1\}$ (This choice will become clear later in the proof). Using contradiction arguments similar to the ones employed in the proof of Lemma \ref{lem:lowerboundderivative}, it is straight forward to see that there exists $\de_2>0$ only depending on $L_b, U_b, L_\theta, \|\stt\|$ such that
\begin{eqnarray}
\sup_{\tpai{1, t}\in [0, \de_2], \|\pb{t}\|\in [L_b, U_b], \ptt{1, t}\in [L_\theta, \|\stt\|]} \frac{1}{4\funcP(\tpai{1, t}, \|\pb{t}\|, \ptt{1, t})(1-\funcP(\tpai{1, t}, \|\pb{t}\|, \ptt{1, t}))}-1\leq \ep_{\funcp}.\label{eq:precon4}
\end{eqnarray}
\end{enumerate}
Now combining \eqref{eq:con1}, \eqref{eq:con2}, \eqref{eq:precon1}, \eqref{eq:precon2}, \eqref{eq:precon4} and Lemma \ref{lemma:1a0conratepre} we conclude that for all $ \tpai{1, t}\in [0, \min\{\de_1, L_\theta, 1\}], \|\pb{t}\|\in[L_b, U_b], \ptt{1, t}\in [L_\theta, \|\stt\|]$,
\begin{eqnarray}
&&|2\funcQ(\tpai{1, t}, \|\pb{t}\|, \ptt{1, t})- \ptt{1, t}|\n\\
&\leq&|\tpai{1, t}(2\funcP(\tpai{1, t}, \|\pb{t}\|, \ptt{1, t})-1)|+|\frac{1}{2}(F(\|\pb{t}\|, \ptt{1, t}-\tpai{1, t})-F(\|\pb{t}\|, \ptt{1, t}))|\n\\
&~&+|\frac{1}{2}(F(\|\pb{t}\|, \ptt{1, t}+\tpai{1, t})-F(\|\pb{t}\|, \ptt{1, t}))|+|F(\|\pb{t}\|, \ptt{1, t})-\ptt{1, t}|\n\\
&\leq&2(\tpai{1, t})^2+(4\|\stt\|+1)\tpai{1, t}+\kappa_b''|\|\pb{t}\|-\ptt{1, t}|\n\\
&\leq&(4\|\stt\|+3)\tpai{1, t}+\kappa_b''|\|\pb{t}\|-\ptt{1, t}|.\n
\end{eqnarray}
Hence together with \eqref{eq:precon1} again and \eqref{eq:precon4} in \eqref{eq:con1}, we have $\forall \tpai{1, t}\in [0, \min\{\de_1, L_\theta, 1, \de_2\}], \|\pb{t}\|\in[L_b, U_b], \ptt{1, t}\in [L_\theta, \|\stt\|]$
\begin{eqnarray}
|\tpbi{{\langle t\rangle, 1}, t+1} - \ptt{1, t}|&\leq&\left| \frac{2\funcQ(\tpai{1, t}, \|\pb{t}\|, \ptt{1, t})- \ptt{1, t}}{ 4 \funcP(\tpai{1, t}, \|\pb{t}\|, \ptt{1, t}) (1- \funcP(\tpai{1, t}, \|\pb{t}\|, \ptt{1, t}))} \right| \nonumber \\
&&+  \left| \frac{\ptt{1, t} (2\funcP(\tpai{1, t}, \|\pb{t}\|, \ptt{1, t})-1)^2}{ 4 \funcP(\tpai{1, t}, \|\pb{t}\|, \ptt{1, t}) (1- \funcP(\tpai{1, t}, \|\pb{t}\|, \ptt{1, t}))} \right|\n\\
&\leq&(1+\ep_{\funcp})\left(|2\funcQ(\tpai{1, t}, \|\pb{t}\|, \ptt{1, t})- \ptt{1, t}|+\ptt{1, t} \left(2\funcP(\tpai{1, t}, \|\pb{t}\|, \ptt{1, t})-1\right)^2\right)\n\\
&\leq&(1+\ep_{\funcp})((4\|\stt\|+3)\tpai{1, t}+\kappa_b''|\|\pb{t}\|-\ptt{1, t}|+4\|\stt\|\tpai{1, t})\n\\
&\leq&2(8\|\stt\|+3)\tpai{1, t}+(\frac{1-\kappa_b''}{2\kappa_b''}+1)\kappa_b''|\|\pb{t}\|-\ptt{1, t}|\n\\
&\leq&(16\|\stt\|+6)\|\pa{t}\|+\frac{1+\kappa_b''}{2}|\|\pb{t}\|-\ptt{1, t}|.\n
\end{eqnarray}
In summary, if we set $\de_a\triangleq \min\{\de_1, L_\theta, 1, \de_2\}>0$ and $\kappa_b'\triangleq \frac{1+\kappa_b''}{2}<1$, we have $\forall \tpai{1, t}\leq \|\pa{t}\|\in [0, \de_a], \|\pb{t}\|\in[L_b, U_b], \ptt{1, t}\in [L_\theta, \|\stt\|]$,
\begin{eqnarray}\label{eq:boundtheta1b1}
|\tpbi{{\langle t\rangle, 1}, t+1} - \ptt{1, t}|\leq (16\|\stt\|+6)\|\pa{t}\|+\kappa_b'|\|\pb{t}\|-\ptt{1, t}|.\n
\end{eqnarray}
\end{enumerate}
So far we have proved in \eqref{eq:boundtheta2b2}  and \eqref{eq:boundtheta1b1} and the following bounds:
\begin{eqnarray}
|\tpbi{{\langle t\rangle, 1}, t+1} - \ptt{1, t}| &\leq& (16\|\stt\|+6)\|\pa{t}\|+\kappa_b'|\|\pb{t}\|-\ptt{1, t}| \nonumber \\
|\tpbi{{\langle t\rangle, 2}, t+1} - \ptt{2, t}| &\leq& \kappa_s | \ptt{2, t}|.\n
\end{eqnarray}

 Let $\kappa_b=\max\{\kappa_s, \kappa_b'\}\in (0, 1)$ and $c_b'=16\|\stt\|+6$. Then, we conclude that
\begin{eqnarray}
\|\pb{t+1} - \stt\|^2&=&|\tpbi{{\langle t\rangle, 1}, t+1} - \ptt{1, t}|^2+|\tpbi{{\langle t\rangle, 2}, t+1}-\ptt{2, t}|^2\n\\
&\leq&((16\|\stt\|+6)\|\pa{t}\|+\kappa_b'|\|\pb{t}\|-\ptt{1, t}|)^2+(\kappa_s\ptt{2, t})^2\n \\
&\leq&\kappa_b^2(|\|\pb{t}\|-\ptt{1, t}|^2+|\ptt{2, t}|^2)+(c_b')^2\|\pa{t}\|^2+2c_b'\|\pa{t}\|\kappa_b|\|\pb{t}\|-\ptt{1, t}|\n\\
&\leq&\kappa_b^2\|\pb{t}-\stt\|^2+((c_b')^2+2c_b'U_b+2c_b'\|\stt\|)\|\pa{t}\|.\n
\end{eqnarray}
Setting $c_b=(c_b')^2+2c_b'U_b+2c_b'\|\stt\|$ completes the proof of Lemma \ref{lemma:conratepre1}.

\subsection{Proof of Theorem \ref{thm:popEMsymmfp1} }\label{sec:prooftheorem1symmetric}

To prove this result we will use some of the results we have proved in the last few sections. We summarize them here.
\begin{itemize}

\item[(i)]  In Section \ref{sec:maincont} we showed that to study the dynamics of Population EM for Model 1, it is sufficient to study
\begin{equation}\label{eq:popEMsymm2l}
\pt{t+1} = \mathbb{E} ( (2\w_d(\Y, \pt{t}) -1 ) \Y),
\end{equation}
where $\w_d(\y, \pt{t})=\frac{\phi(\y - \pt{t}; I)}{\phi(\y - \pt{t}; I) + \phi(\y + \pt{t}; I)}$ and $\Y \sim \frac{1}{2} N(- \stt, \vI_d) + \frac{1}{2} N( \stt, \vI_d)$.

\item[(ii)] In Section \ref{sec:maincont} and \ref{sec:connection1} we showed that to study the dynamics of Population EM for Model 2, it is sufficient to study
\begin{eqnarray}
\pa{t+1}&=&\frac{\pq{t+1}(1-2\pp{t+1})}{2\pp{t+1}(1-\pp{t+1})}, \nonumber \\
\pb{t+1}&=&\frac{\pq{t+1}}{2\pp{t+1}(1-\pp{t+1})}.\label{equ:iteb2}
\end{eqnarray}
where
\begin{eqnarray*}
\pq{t+1}
&=&\bbE \w_d(\Y-\pa{t}, \pb{t})\Y , \\
\pp{t+1} &=&\bbE \w_d(\Y-\pa{t}, \pb{t}),
\end{eqnarray*}
where $Y \sim \frac{1}{2} N(- \stt, \vI_d) + \frac{1}{2} N( \stt, \vI_d)$.

\item [(iii)]  According to Lemma \ref{lem:connectionpopEMModel1and2}, \eqref{equ:iteb2} reduces to \eqref{eq:popEMsymm2l}, if $\pa{0}=\v0$. In other words, if we set $\pa{0}=\v0$, then $\pa{t}=\v0$ and $\pb{t} = \pt{t}$. This in turn implies that if we analyze the convergence dynamics of \eqref{equ:iteb2}, we immediately obtain the convergence of \eqref{eq:popEMsymm2l} by setting $\pa{0}=\v0$.

\item[(iv)] If $\pbeta{t}$ denotes the angle between $\stt$ and $\pb{t}$, then we proved in Appendix \ref{sec:proof:thm:convergencerate1} that for \eqref{equ:iteb2} we have
\begin{eqnarray}
|\sin \pbeta{t+1}| \leq \kappa_\beta |\sin \pbeta{t}|.\n
\end{eqnarray}
The same is true for $\pa{0}=\v0$ initialization.
\item[(v)]According to Lemma \ref{lemma:upperboundEMestimates} and Lemma \ref{lem:lowerboundbt}, we have $\|\pb{t}\|\in [c_{L, 1}, c_{U, 3}]$. According to Lemma \ref{lem:planarestimates}, we have $\ptt{1, t}\in [\ptt{1, 0}, \|\stt\|]$. The same is true for $\pa{0}=\v0$ initialization.
\end{itemize}

Note that we can employ Theorem \ref{thm:convergenceratemagnitude} to claim that (by setting $\pa{t}=\v0$) if $\langle \pb{t}, \stt \rangle> 0$, then for the symmetric case, there exists $T_0$ such that for every $t>T_0$,
\[
|\pb{t+1} - \stt| \leq \kappa_b^2 |\pb{t}- \stt|.
\]
However, our claim in Theorem \ref{thm:popEMsymmfp1} is stronger. In fact we would like to show that for the symmetric case the geometric convergence starts at iteration 1. We use the notations and equations developed in Appendix \ref{sec:notation}. In particular, we use the rotation matrix $\vU_t$ introduced there and rotate all the vectors $\pb{t}$ and $\pq{t}$ with $\vU_t$. Note that according to Lemma \ref{lemma:projection} we know that $\pb{t+1}$ lies in the span of $\stt$ and $\pb{t}$ and hence, $\tpbi{{\langle t\rangle, i}, t+1}=0$ for $i \geq 3$. Therefore we only need to consider the first two coordinates.
According to \eqref{eq:tildeiterationab2} and \eqref{eq:q2thetaS}, we have
\begin{eqnarray}\label{eq:bt2generic}
|\tpbi{{\langle t\rangle, 2}, t+1} - \ptt{2, t}| &=& \frac{|\tpqi{{\langle t\rangle, 2}, t+1}|}{ 2 \funcP (0, \|\pb{t}\|, \ptt{1, t} ) (1-\funcP(0, \|\pb{t}\|, \ptt{1, t}))} \nonumber \\
&=& |\ptt{2, t}| \left|\frac{ S(0, \|\pb{t}\|, \ptt{1, t} ) }{ 2 \funcP (0, \|\pb{t}\|, \ptt{1, t} ) (1-\funcP(0, \|\pb{t}\|, \ptt{1, t})) }-1\right| \nonumber \\
&=& |\ptt{2, t}| |2S(0, \|\pb{t}\|, \ptt{1, t} )-1|,
\end{eqnarray}
where the last equality is due to the fact that $\funcP(0, \|\pb{t}\|, \ptt{1, t})=\frac{1}{2}$. Also, according to \eqref{eq:upperboundforS} we have
\begin{equation}
\kappa_s \triangleq \sup_{\tpai{1, t}\in [0, U_a] , \|\pb{t}\| \in [L_b, U_b], \ptt{1, t} \in [L_\theta, \|\stt\|]} 1-2S(\tpai{1, t}, \|\pb{t}\|, \ptt{1, t}) <1.\n
\end{equation}
Hence, let $U_a=0, L_b=c_{L, 1}, U_b=c_{L, 3}$ and $L_{\theta}=\ptt{1, 0}$, we conclude that
\begin{equation}\label{eq:secondcoordinate2}
|\tpbi{{\langle t\rangle, 2}, t+1}- \ptt{2, t}| \leq \kappa_s |\ptt{2, t}|.
\end{equation}
Similarly, employing \eqref{eq:tildeiterationab2} for the first coordinate and using the fact that $\funcP(0, \|\pb{t}\|, \ptt{1, t}) = \frac{1}{2}$, we obtain
\begin{eqnarray}\label{eq:firstcoordinateF1}
|\tpbi{{\langle t\rangle, 1}, t+1}-\ptt{1, t} | = | 2 \funcQ (0, \|\pb{t}\|, \ptt{1, t}) - \ptt{1, t}|= |F(\|\pb{t}\|, \ptt{1, t})- \ptt{1, t} |.
\end{eqnarray}
Note that the last equality is due to \eqref{eq:QFrelation}. According to Lemma \ref{lemma:1a0conratepre}, we know that there exists $\kappa_b''\in (0, 1)$ which is a function of only $L_b, U_b, L_\theta, \|\stt\|$ such that
\begin{equation}\label{eq:firstcoordinateF2}
|F(x_b, x_\theta)-x_\theta|\leq \kappa_b'' |x_b-x_\theta|, \forall x_b\in [L_b, U_b], x_\theta\in [L_\theta, \|\stt\|].
\end{equation}
Let  $L_b=c_{L, 1}, U_b=c_{L, 3}$ and $L_{\theta}=\ptt{1, 0}$. Combining \eqref{eq:secondcoordinate2}, \eqref{eq:firstcoordinateF1} and \eqref{eq:firstcoordinateF2} completes the proof.

\subsection{Proof of Theorem \ref{thm:hyperplaneinit1}} \label{sec:proofthm5}

\subsubsection{Main Steps of the Proof}
The proof for the case $\langle \pb{0}, \stt \rangle =0$ is very different from the proof of the case  $\langle \pb{0}, \stt \rangle \neq 0$. It seems that the convergence of the algorithm to its stationary point may not be geometric and hence proof ideas we developed for the case $\langle \pb{0}, \stt\rangle \neq 0$ are not applicable here. Hence, we prove Theorem \ref{thm:convergenceratemagnitude} using the following strategy:

\begin{itemize}

\item [(i)]We first characterize all the stationary points of Population EM. Let $(\va, \vb)$ denote the stationary points and we show that $\va =\v0$ and $\vb \in \{- \stt, \v0, \stt \}$. This is discussed in Appendix \ref{ssec:fixedpointsEM}.

\item[(ii)] We then show that any accumulation point of $\{(\pa{t}, \pb{t})\}$ is one of the stationary points. Let $(\va^\infty, \vb^\infty)$ denote any accumulation point. This is discussed in (i) in Appendix \ref{ssec:proofconvergenceEMgeneral}.

\item [(iii)] We show that if $\langle \pb{0}, \stt \rangle =0$, $\vb^\infty$ can not converge to $-\stt$ or $\stt$. Hence, the algorithm has to converge to $\v0$. Since $\va^\infty=\v0$ for all stationary points, we have $\{\pa{t}, \pb{t}\}$ converges to $(\v0, \v0)$. This is discussed in (ii) in Appendix \ref{ssec:proofconvergenceEMgeneral}

\end{itemize}

\subsubsection{Characterizing the Fixed Points of Population EM}\label{ssec:fixedpointsEM}

 First note that if we write the iterations of Population EM in terms of $\pa{t}$ and $\pb{t}$ we obtain
\begin{eqnarray}
\pa{t+1}&=&\frac{\pq{t+1}(1-2\pp{t+1})}{2\pp{t+1}(1-\pp{t+1})}, \n \\
\pb{t+1}&=&\frac{\pq{t+1}}{2\pp{t+1}(1-\pp{t+1})}, \n
\end{eqnarray}
where
\begin{eqnarray}
\pq{t+1}&=&\int \w_d(\y-\pa{t}, \pb{t})\y\pdf{d, \y, \stt}\dif\y, \n\\
\pp{t+1}&=&\int \w_d(\y-\pa{t}, \pb{t})\pdf{d, \y, \stt}\dif\y.\n
\end{eqnarray}
If $(\pq{t}, \pp{t}, \pa{t}, \pb{t})$ converges to $(\vfuncq, \vfuncp, \va, \vb)$, then it is straightforward to show that
\begin{eqnarray}
\va&=&\frac{\vfuncq(1-2\funcp)}{2\funcp(1-\funcp)}\label{equ:defva}, \\
\vb&=&\frac{\vfuncq}{2\funcp(1-\funcp)}\label{equ:defvb}, \\
\vfuncq
&=&\int \w_d(\y-\va, \vb)\y\pdf{d, \y, \stt}\dif\y, \label{equ:defvq}\\
\vfuncp&=&\int \w_d(\y-\va, \vb)\pdf{d, \y, \stt}\dif\y.\label{equ:defvp}
\end{eqnarray}
Hence, the main step of the proof is to characterize the solutions of these four equations. We first consider the one-dimensional setting in which $Y \in \mathbb{R}$ and prove the following two facts:
\begin{itemize}
\item[(i)] The only feasible solution for $a$ is zero.
\item[(ii)] We then set $a=0$ and show that the only possible solutions for $b$ are $-\theta^*, 0, \theta^*$.
\end{itemize}
We should prove the above two by considering the following four different cases: (1) $a\geq 0, b\geq 0$, (2)  $a\geq 0, b\leq 0$, (3) $a\leq 0, b\geq 0$, (4) $a\leq 0, b\leq 0$. Since the four cases are similar we focus on the first case only, i.e., $a\geq 0, b\geq 0$. To prove that the only possible solution of $a$ is zero, note that
\eqref{equ:defva} can be written as
\begin{equation}\label{eq:amustbezero}
a = \frac{\funcQ(a, b, \theta^*)(1-2\funcP(a, b, \theta^*))}{2\funcP(a, b, \theta^*)(1-\funcP(a, b, \theta^*))} \overset{(1)}{\leq} \kappa_a a.
\end{equation}
where $\kappa_a <1$. Note that Inequality (1) is a result of Lemma \ref{lemma:aconvergeupperbound}. Note that \eqref{eq:amustbezero} implies that $a$ must be zero.

 The only remaining step is to examine the solutions for $b$. It is straightforward to prove that $\funcP(0, b, \theta^*) =\frac{1}{2}$. Hence, we can simplify \eqref{equ:defvb} to
\begin{eqnarray}\label{eq:fixedpointb}
b\ =\ 2 \funcQ(0, b, \theta^*) \ =\  F(b, \theta^*),
\end{eqnarray}
where the last equality is due to \eqref{eq:QFrelation}. The following lemma enables us to characterize the solutions of \eqref{eq:fixedpointb}.
\begin{lemma}\label{lem:Fpropert1}
$F(x_b, x_\theta)$ is a concave function of $x_b \geq 0 $. Furthermore, we have the following: (i) $F(0, x_\theta)=0$, (ii) $F(x_\theta, x_\theta) = x_\theta$.
\end{lemma}
The proof of this lemma is presented in the Appendix \ref{sec:proofoflemmaFpropert1}. It is straightforward to use the above properties and show that $F(x_b, x_\theta)$ has in fact the shape that is exhibited in Figure \ref{fig:funcF}, which proves our claim in the one dimensional setting.

\begin{figure}
\begin{center}
\includegraphics[width=8cm]{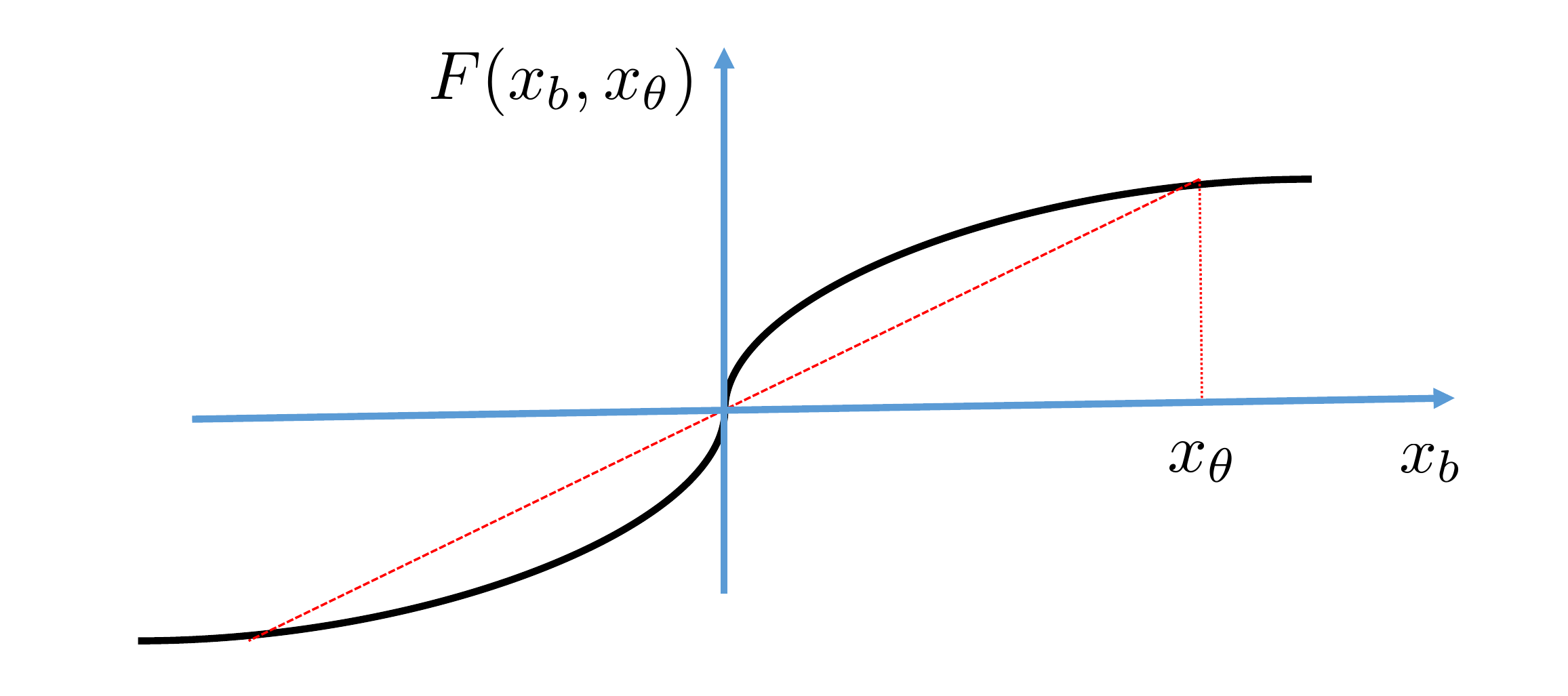}
\vspace{.1cm}
\caption{$F(x_b, x_\theta)$ as a function of $x_b$ when $x_\theta$ is fixed. The black line represents the curve of function $F$ and red line represents the diagonal line: $y=x$. }
\label{fig:funcF}
\end{center}
\end{figure}
To extend the proof to higher dimensions, we rotate the coordinates. Suppose that the fixed point is $\va, \vb, \vfuncp, \vfuncq$. Let $\tilde{M}$ denote a rotation for which the following two hold: (i)  $\tb \triangleq \tilde{M} \vb = (\|\vb\|_2, 0, 0, \ldots, 0)$ and (ii) $\tstt\triangleq \tilde{M} \stt = (\tstti{1}, \tstti{2}, 0 , \ldots, 0)$. Define $\tq = \tilde{M}\vfuncq$ and $\ta = \tilde{M}\va$. Lemma \ref{lemma:rotate} shows that if we let $\tb, \ta, \tq, \vfuncp$ denote the corresponding fixed point in the new coordinates, then they satisfy the same equations, i.e.,
\begin{eqnarray}
\ta&=&\frac{\tq(1-2\vfuncp)}{2\vfuncp(1-\vfuncp)}\label{equ:defta}, \\
\tb&=&\frac{\tq}{2\vfuncp(1-\vfuncp)}\label{equ:deftb}, \\
\tq
&=&\int \w_d(\y-\ta, \tb)\y\pdf{d, \y, \stt}\dif\y, \label{equ:deftq}\\
\vfuncp&=&\int \w_d(\y-\ta, \tb)\pdf{d, \y, \stt}\dif\y.\label{equ:deftp}
\end{eqnarray}
First, it is straightforward to employ \eqref{equ:defta} and \eqref{equ:deftb} and confirm that $\forall i \geq 2$ we have
\begin{eqnarray}
\tqi{i}&=&2\tbi{i}\vfuncp(1-\vfuncp)\ =\ 0, \n\\
\tai{i}&=& \tbi{i}(1-2\vfuncp)\ =\ 0.\n
\end{eqnarray}
Hence, with $\tstt_i=0, \forall i\geq 3$, we only need to consider the first two coordinates. Our goal is to prove the following two statements:
\begin{itemize}
\item[(i)] If $\tstti{2}\neq 0$, then $\tbi{1} =0$ and $\tai{1} =0$. In other words, both $\va$ and $\vb$ are zero.
\item[(ii)] If $\tstti{2} =0$, then the problem will be reduced to the one-dimensional problem that we have already discussed. Hence, it is straightforward to characterize the fixed points.
\end{itemize}

Here we focus on case (i), i.e., $\tstti{2} \neq 0$. Since $\tqi{2}=0$, we have
\begin{eqnarray}
0=\tqi{2}&=&\int \w_d(\y-\ta, \tb)y_2\pdf{d, \y, \tstt}\dif\y\n\\
&=&\frac{1}{2}\int \w(y_1-\tai{1}, \|\vb\|)\phi(y_1-\tstti{1}) \dif y_1\int y_2\phi(y_2-\tstti{2})
  \dif y_2\n\\
&~&+\frac{1}{2}\int \w(y_1-\tai{1}, \|\vb\|)\phi(y_1+\tstti{1}) \dif y_1\int y_2\phi(y_2+\tstti{2})
  \dif y_2\n\\
&=&\tstti{2}\int \w(y_1-\tai{1}, \|\vb\|)\dpdf{y_1, \tstti{1}} \dif y_1\n\\
&=&\tstti{2}\int_{0}^{+\infty}(\w(y_1-\tai{1}, \|\vb\|)-\w(-y_1-\tai{1}, \|\vb\|))\dpdf{y_1, \tstti{1}}
  \dif y_1\n\\
&=&\tstti{2}\int_{0}^{+\infty}\frac{e^{2y_1\|\vb\|}-e^{-2y_1\|\vb\|}}{e^{2y_1\|\vb\|}+e^{-2y_1\|\vb\|}+e^{2\|\ta\|\|\vb\|}+e^{-2\|\ta\|\|\vb\|}}\frac{1}{2\sqrt{2\pi}}(e^{-(y_1-\tstti{1})^2/2}-e^{-(y_1+\tstti{1})^2/2})
  \dif y_1.\n
\end{eqnarray}
It is straightforward to see that since $\tstti{2} \neq 0$, then $\tstti{1}=0$. Hence, from \eqref{equ:defta} and the definitions of $\funcQ$ and $\funcP$ functions given in Appendix \ref{sec:notation} we have
\begin{eqnarray}
\|\va\|=\|\ta\|= |\tai{1}|=|\tbi{1}||1-2p|=\frac{\funcQ(\|\ta\|, \|\vb\|, 0)(1-2\funcP(\|\ta\|, \|\vb\|, 0))}{2\funcP(\|\ta\|, \|\vb\|, 0)(1-\funcP(\|\ta\|, \|\vb\|, 0))}\leq \kappa_a\|\ta\|, \n
\end{eqnarray}
where the last inequality is due to Lemma \ref{lemma:aconvergeupperbound}. We know that $\kappa_a<1$.
Therefore we have $\va=\v0$ and
\begin{eqnarray}
\|\vb\|\ =\ \tbi{1}&=&\frac{\tqi{1}}{2\vfuncp(1-\vfuncp)}\n\\
&=&2\tqi{1}\n\\
&=&2\int \w(y-\ta, \tb)y_1 \pdf{d, \y, \tstt}\dif\y\n\\
&=&2\int\frac{e^{y_1\|\vb\|}}{e^{y_1\|\vb\|}+e^{-y_1\|\vb\|}}y_1\frac{1}{\sqrt{2\pi}}e^{-y_1^2/2}
  \dif y\n\\
&=&0.\n
\end{eqnarray}
Thus the only solution is $$(\va, \vb)=(\v0, \v0).$$

\subsubsection{Proof of Convergence for Population EM}\label{ssec:proofconvergenceEMgeneral}
We can break the proof into the following steps. Let $\{ \pa{t}, \pb{t}\}_{t=1}^\infty$ denote all the estimates of the Population EM algorithm.
\begin{itemize}
\item[(i)] We first prove that every accumulation point of  $\{ \pa{t}, \pb{t}\}_{t=1}^\infty$ satisfies the fixed point equations:
\begin{eqnarray}
\va&=&\frac{\vfuncq(1-2\vfuncp)}{2\vfuncp(1-\vfuncp)}, \n\\
\vb&=&\frac{\vfuncq}{2\vfuncp(1-\vfuncp)}, \n\\
\vfuncq
&=&\mathbb{E} \w_d(\Y-\va, \vb) \Y \n\\
\vfuncp&=&\mathbb{E} \w_d(\Y-\va, \vb), \n
\end{eqnarray}
where $\Y \sim \frac{1}{2} N(-\stt, I)+ \frac{1}{2} N(\stt, I)$. This proof is presented in Appendix \ref{sec:Laccumulationpoitnsproof}.

We have already proved that these fixed point equations only have the following solutions : $\va=\v0$ and $\vb \in \{ -\stt, \v0, \stt\}$. We proved in Lemma \ref{lemma:sign} that $\sgn(\langle \pb{t}, \stt \rangle)= \sgn(\langle \pb{t+1}, \stt \rangle)$. Hence, we conclude that $\langle \pb{t}, \stt \rangle =0$ for every $t$. The only possible fixed point is hence $(\v0, \v0)$. In summary, in the first step we prove that it only has one accumulation point, that is $(\v0, \v0)$.

\item[(ii)] Next we prove that $\{ \pa{t}, \pb{t}\}_{t=1}^\infty$ is a convergent sequence. Suppose that the sequence does not converge to $(\v0, \v0)$, then there exists an $\epsilon$ such that for every $T$, there exists a $t>T$ such that
\[
\|(\pa{t}, \pb{t})\|_2 > \epsilon.
\]
We construct a subsequence of our sequence in the following way: Set $T=1$ and pick $t_1>T$ such that $\|(\pa{t_1}, \pb{t_1})\|_2 > \epsilon$. Now, set $T= t_1+1$, and pick $t_2> T$ such that $\|(\pa{t_2}, \pb{t_2})\|_2 > \epsilon$. Continue the process until we construct a sequence $\{(\pa{t_n}, \pb{t_n})\}_{n=1}^{\infty}$. According to Lemma \ref{lemma:upperboundEMestimates} $\{(\pa{t_n}, \pb{t_n})\}_{n=1}^{\infty}$ is in a compact set and has a convergent subsequence. But according to part (i) the converging subsequence of this sequence must converge to $(\v0, \v0)$ which is in contradiction with the construction of the sequence $\{(\pa{t_n}, \pb{t_n})\}_{n=1}^{\infty}$. Hence $\{ \pa{t}, \pb{t}\}_{t=1}^\infty$ must be a convergent sequence and converges to $(\v0, \v0)$

\end{itemize}

\subsection{Proof of Theorem \ref{lemma:pconv}}\label{sec:proofsampleEM-populationfiniett}
Let $\sa{t}=\frac{\sm{1, t}+\sm{2, t}}{2}$ and $\sb{t}=\frac{\sm{2, t}-\sm{1, t}}{2}$. Then the iteration functions based on $(\sa{t}, \sb{t})$ are the following:
\begin{eqnarray}
\sa{t+1}&=&\frac{\sq{t+1}(1-2\sp{t+1})}{2\sp{t+1}(1-\sp{t+1})}+\frac{\bar{\y}}{2(1-\sp{t+1})}, \label{equ:iteha}\\
\sb{t+1}&=&\frac{\sq{t+1}}{2\sp{t+1}(1-\sp{t+1})}-\frac{\bar{\y}}{2(1-\sp{t+1})}.\label{equ:itehb}
\end{eqnarray}
where
\begin{eqnarray}
\bar{\y}&=&\frac{1}{n}\sum_{i=1}^n\y_i, \n\\
\sq{t+1}&=&\frac{1}{n}\sum_{i=1}^n \w_d(\y_i-\sa{t}, \sb{t})\y_i, \label{equ:itehq}\\
\sp{t+1}&=&\frac{1}{n}\sum_{i=1}^n \w_d(\y_i-\sa{t}, \sb{t}), \label{equ:itehp}
\end{eqnarray}
where
\begin{eqnarray}
\w_d(\y-\vx_a, \vx_b)&=&\frac{\phi_d(\y-\vx_a-\vx_b)}{\phi_d(\y-\vx_a+\vx_b)+\phi_d(\y-\vx_a-\vx_b)}\n\\
&=&\frac{e^{\dotp{\y-\vx_a, \vx_b}}}{e^{\dotp{\y-\vx_a, \vx_b}}+e^{-\dotp{y-\vx_a, \vx_b}}}.\n
\end{eqnarray}
Therefore $\sq{t}$ and $\sp{t}$ are the empirical versions of $\pq{t}$ and $\pp{t}$ respectively. Our first goal is, for each $i \in \{1,2\}$, to compare the Population EM sequence $(\pm{i,t})_{t\geq0}$ to the Sample-based EM sequence $(\sm{i,t})_{t\geq0}$, provided that the initial values $\pm{i,0}$ and $\sm{i,0}$ are the same.  We prove that
\begin{eqnarray}
\sa{t}{\rightarrow}\pa{t} \ \text{in probability}, \quad \text{and}\quad \sb{t}{\rightarrow}\pb{t}\ \text{in probability, \quad as}\ n\rightarrow \infty.\label{equ:conhab}
\end{eqnarray}

We prove by induction. For $t=0$, it is clear that \eqref{equ:conhab} holds because both Population EM and Sample-based EM start with the same initialization. For $t=1$, by Weak Large Law Numbers (WLLN), we have
\begin{eqnarray}
\sq{1}&=&\frac{1}{n}\sum_{i=1}^n\w_d(\y_i-\sa{0}, \sb{0})\y_i\conp \bbE\w_d(\y-\sa{0}, \sb{0})y\ =\ \pq{1}, \n\\
\sp{1}&=&\frac{1}{n}\sum_{i=1}^n\w_d(\y_i-\sa{0}, \sb{0})\conp \bbE \w_d(\y-\sa{0}, \sb{0})\ =\ \pp{1}, \n\\
\bar{\y}&=&\frac{1}{n}\sum_{i=1}^n\y_i\conp \bbE \y\ =\ \v0.\n
\end{eqnarray}
Since $\pp{1}\in (0, 1)$, by employing the continuous mapping theorem, we have
\begin{eqnarray}
\sa{1}&=&\frac{\sq{1}(1-2\sp{1})}{2\sp{1}(1-\sp{1})}+\frac{\bar{\y}}{2(1-\sp{1})}\ \rightarrow\  \frac{\pq{1}(1-2\pp{1})}{2\pp{1}(1-\pp{1})}=\pa{1}\ \text{in probability}, \n\\
\sb{1}&=&\frac{\sq{1}}{2\sp{1}(1-\sp{1})}+\frac{\bar{\y}}{2(1-\sp{1})}\ \rightarrow\  \frac{\pq{1}}{2\pp{1}(1-\pp{1})}\ =\ \pb{1}\ \text{in probability}.\n
\end{eqnarray}
Therefore  \eqref{equ:conhab} holds for $t=1$. Now we assume that \eqref{equ:conhab} holds for $t\geq 1$, and our goal is to prove it for $t+1$. Note that
\begin{eqnarray}
\left\|\frac{\partial \w_d(\y-\vx_a, \x_b)}{\partial \x_a} \right\|&=& \left\|-\frac{2\vx_b}{(e^{\dotp{\y, \vx_b}-\dotp{\vx_a, \x_b}}+e^{-\dotp{\y, \x_b}+\dotp{\x_a, \x_b}})^2} \right\|\n\\
&\leq&\frac{\|\x_b\|}{2}, \n\\
\left\|\frac{\partial \w_d(\y-\x_a, \x_b)}{\partial \x_b}\right\|&=&\left\|\frac{2(\y-\x_a)}{(e^{\dotp{\y-\x_a, \x_b}}+e^{-\dotp{\y-\x_a, \x_b}})^2} \right\| \n\\
&\leq&\|\frac{\y-\x_a}{2}\|\n\\
&\leq&\frac{\|\y\|+\|\x_a\|}{2}.\n
\end{eqnarray}
Therefore we have
\begin{eqnarray}
&&|\pp{t+1}-\sp{t+1}|=|\bbE \w_d(\y-\pa{t}, \pb{t})-\frac{1}{n}\sum_{i=1}^n\w_d(\y_i-\sa{t}, \sb{t})|\n\\
&\leq&|\bbE \w_d(\y-\pa{t}, \pb{t})-\frac{1}{n}\sum_{i=1}^n\w_d(\y_i-\pa{t}, \pb{t})|+\left|\frac{1}{n}\sum_{i=1}^n\w_d(\y_i-\pa{t}, \pb{t})-\frac{1}{n}\sum_{i=1}^n\w_d(\y_i-\sa{t}, \sb{t})\right|\n\\
&\leq&|\bbE \w_d(\y-\pa{t}, \pb{t})-\frac{1}{n}\sum_{i=1}^n\w_d(\y_i-\pa{t}, \pb{t})|\n\\
&&+\left|\frac{1}{n}\sum_{i=1}^n\left(\frac{\|\pb{t}_{\xi}\|}{2}\|\sa{t}-\pa{t}\|+\frac{\|\y_i\|+\|\pa{t}_{\xi}\|}{2}\|\sb{t}-\pb{t}\|\right)\right|\n\\
&\leq&|\bbE \w_d(\y-\pa{t}, \pb{t})-\frac{1}{n}\sum_{i=1}^n\w_d(\y_i-\pa{t}, \pb{t})|+\frac{1}{2}\frac{\sum_{i=1}^n\|\y_i\|}{n}\|\sb{t}-\pb{t}\|\n\\
&&+\left(\frac{\|\pb{t}_{\xi}\|}{2}\|\sa{t}-\pa{t}\|+\frac{\|\pa{t}_{\xi}\|}{2}\|\sb{t}-\pb{t}\|\right), \n
\end{eqnarray}
where $$\pa{t}_{\xi}=\xi \pa{t}+(1-\xi)\sa{t}, \quad \text{and}\quad \pb{t}_{\xi}=\xi \pb{t}+(1-\xi)\sb{t}, \ \text{for some}\ \xi\in [0, 1].$$
By WLLN, induction assumption and $$\|\pa{t}_{\xi}\|\leq 2\|\pa{t}\|+\|\pa{t}-\sa{t}\|, \quad \text{and}\quad \|\pb{t}_{\xi}\|\leq 2\|\pb{t}\|+\|\pb{t}-\sb{t}\|, $$
we have
\begin{eqnarray}
|\pp{t+1}-\sp{t+1}|&\leq&|\bbE\w_d(\y-\pa{t}, \pb{t})-\frac{1}{n}\sum_{i=1}^n\w_d(\y_i-\pa{t}, \pb{t})|+\frac{2\|\pb{t}\|+\|\pb{t}-\sb{t}\|}{2}\|\sa{t}-\pa{t}\|\n\\
&~&+\frac{2\|\pa{t}\|+\|\pa{t}-\sa{t}\|}{2}\|\sb{t}-\pb{t}\|+\frac{1}{2}\frac{\sum_{i=1}^n\|\y_i\|}{n}\|\sb{t}-\pb{t}\| \rightarrow 0\ \text{in probability}.\n
\end{eqnarray}
Similarly, we have
\begin{eqnarray}
\lefteqn{\|\pq{t+1}-\sq{t+1}\|=\|\bbE\w_d(\y-\pa{t}, \pb{t})\y-\frac{1}{n}\sum_{i=1}^n\w_d(\y_i-\sa{t}, \sb{t})\y_i\|}\n\\
&\leq&\|\bbE \w_d(\y-\pa{t}, \pb{t})\y-\frac{1}{n}\sum_{i=1}^n\w_d(\y_i-\pa{t}, \pb{t})\y_i\|\n\\
&&+\|\frac{1}{n}\sum_{i=1}^n\w_d(\y_i-\pa{t}, \pb{t})y_i-\frac{1}{n}\sum_{i=1}^n\w_d(\y_i-\sa{t}, \sb{t})\y_i\|\n\\
&\leq&\|\bbE\w_d(\y-\pa{t}, \pb{t})\y-\frac{1}{n}\sum_{i=1}^n\w_d(\y_i-\pa{t}, \pb{t})\y_i\|\n\\
&&+\left\|\frac{1}{n}\sum_{i=1}^n(\frac{\|\pb{t}_{\xi}\|}{2}\|\sa{t}-\pa{t}\|+\frac{\|\y_i\|+\|\pa{t}_{\xi}\|}{2}\|\sb{t}-\pb{t}\|)\y_i\right\|\n\\
&\leq&\|\bbE\w_d(\y-\pa{t}, \pb{t})\y-\frac{1}{n}\sum_{i=1}^n\w_d(\y_i-\pa{t}, \pb{t})\y_i\|+\left\|\frac{1}{2n}\sum_{i=1}^n\|\y_i\|\y_i\right\|\|\sb{t}-\pb{t}\|\n\\
&~&+(\|\pa{t}\|\|\pb{t}-\sb{t}\|+\|\pb{t}\|\|\pa{t}-\sa{t}\|+\|\pa{t}-\sa{t}\|\|\pb{t}-\sb{t}\|)\|\frac{1}{n}\sum_{i=1}^n\y_i\|.\n
\end{eqnarray}
By WLLN and induction assumption, we have
\begin{eqnarray}
\|\pq{t+1}-\sq{t+1}\|\rightarrow 0 \ \ \text{in probability}.\n
\end{eqnarray}
Therefore with $\pp{t+1}\in (0, 1)$, we have
\begin{eqnarray}
\sa{t+1}&=&\frac{\sq{t+1}(1-2\sp{t+1})}{2\sp{t+1}(1-\sp{t+1})}+\frac{\bar{y}}{2(1-\sp{t+1})}\rightarrow \frac{\pq{t+1}(1-2\pp{t+1})}{2\pp{t+1}(1-\pp{t+1})}\ =\ \pa{t+1}\ \text{in probability}, \n\\
\sb{t+1}&=&\frac{\sq{t+1}}{2\sp{t+1}(1-\sp{t+1})}+\frac{\bar{y}}{2(1-\sp{t+1})}\rightarrow \frac{\pq{t+1}}{2\pp{t+1}(1-\pp{t+1})}\ =\ \pb{t+1}\ \text{in probability}.\n
\end{eqnarray}
Hence \eqref{equ:conhab} holds for $t+1$. With induction, we completes the proof of this lemma.

\subsection{Proof of Theorem \ref{lemma:empconverge2}}\label{sec:proofofempconverge2}

\subsubsection{Roadmap of the Proof}
The main idea of the proof is simple. We first show that if we initialize Sample-based EM in a way that $\sa{0}$ is small enough and $\sb{0}$ is in small neighborhood of $\stt$, then the sampled based EM will converge to a point whose distance from $\stt$ is $O(\sqrt{d/n})$ with probability converging to 1 as $n \rightarrow \infty$. Let's call this neighborhood of $(\va, \vb)$, $\mathcal{N}_{\v0, \stt}$.

According to Theorem \ref{thm:convergenceratemagnitude} and \ref{thm:convergencerate1} we know that Population EM converges to the true parameter under quite general initialization. Hence, there exists an iteration $T_0$ at which the estimate of Population EM is in $\mathcal{N}_{\v0, \stt}$. We know from Theorem \ref{lemma:pconv} that at iteration $T_0$, $\sa{T_0} \rightarrow \pa{T_0}$ and $\sb{T_0} \rightarrow \pb{T_0}$ in probability. Hence, with probability converging to $1$, ($\sa{T_0}, \sb{T_0}) \in \mathcal{N}_{\v0, \stt}$, and hence $(\sa{t}, \sb{t})$ converge to a point that is at a distance $O(\sqrt{d/n})$ from $(\v0, \stt)$. In other words, if $\hat{\va}^{\infty}$ and $\hat{\vb}^{\infty}$ the limiting estimates, then
\begin{eqnarray}
\|\hat{\va}^{\infty}\| = O(\sqrt{d/n}), \nonumber \\
\|\hat{\vb}^{\infty} - \stt\| = O(\sqrt{d/n}), \n
\end{eqnarray}
with probability converging to $1$, which is equivalent to what we wanted to prove.

As is clear from the above discussion, the only challenging part is to prove that if  $(\sa{0}, \sb{0})$ is in small neighborhood of $(\v0, \stt)$, then the sampled-based EM will converge to a point whose distance from  $(\v0, \stt)$ is $O(\sqrt{d/n})$. The proof of this fact is our main goal in the rest of this proof.

We remind the reader that according to Theorems \ref{thm:convergencerate1} and \ref{thm:convergenceratemagnitude} the estimates of Population EM satisfy the following equations (if initialized properly):
\begin{eqnarray}
\|\pa{t}\|&\rightarrow & 0, \n\\
\|\pb{t}-\stt\|&\rightarrow& 0, \label{eq:populationconvergenceresult}
\end{eqnarray}
Also, we know from the arguments provided in the proof of Theorem \ref{lemma:pconv} that $\sa{t}$ and $\sb{t}$ converge to $\pa{t}$ and $\pb{t}$ in probability. Hence, we expect to have a similar equations for $\sa{t}$ and $\sb{t}$, except for probably an error term that will vanish as $n \rightarrow \infty$. The only issue that may happen is that the errors that are introduced in each iteration may accumulate and will let to a non-vanishing error for $t \rightarrow \infty$. Our first lemma shows that this does not happen.

\begin{lemma}\label{lem:expavoidserroracc}
Suppose that there exist $\kappa_a\in (0, 1), \kappa_b\in (0, 1)$ and $c_b>0$ such that for all $ t'\geq 1$, we have
\begin{eqnarray}
\|\sa{t'}\|&\leq& \kappa_a\|\sa{t'-1}\|+\ep_a, \label{eq:condition1}\\
\|\sb{t'}-\stt\|&\leq& \kappa_b\|\sb{t'-1}-\stt\|+\sqrt{c_b\|\sa{t'-1}\|}+\ep_b, \label{eq:condition2}
\end{eqnarray}
for some $\ep_a, \ep_b>0$. Then we have $\forall t\geq 0, $
\begin{eqnarray}
\|\sa{t}\|&\leq& (\kappa_a)^t\|\sa{0}\|+\frac{1}{1-\kappa_a}\ep_a, \label{eq:result1}\\
\|\sb{t}-\stt\|&\leq& (\kappa_b)^t\|\sb{0}-\stt\|+t\sqrt{c_b\|\sa{0}\|}(\max\{\sqrt{\kappa_a}, \kappa_b\})^t+\frac{1}{1-\kappa_b}\sqrt{\frac{c_b}{1-\kappa_a}\ep_a}+\frac{1}{1-\kappa_b}\ep_b\n\\
\label{eq:result2}
\end{eqnarray}
\end{lemma}
The proof of this lemma will be presented to in Appendix \ref{sec:proofexpoavoidserroracc}. According to this lemma as long as the errors that are introduced in each iteration are bounded by $\epsilon_a$ and $\epsilon_b$, the overall error will also remain bounded and are, in the worst case, proportional to $\sqrt{\epsilon_a}$ and $\epsilon_b$. Hence, if $\epsilon_\a \rightarrow 0$ and $\epsilon_b \rightarrow 0$ as $n \rightarrow \infty$, the overall errors will go to zero too. Hence, proving that \eqref{eq:condition1} and \eqref{eq:condition2} hold for $\epsilon_a \rightarrow 0$ and $\epsilon \rightarrow 0$ will complete the proof Theorem \ref{lemma:empconverge2}.

\begin{lemma}\label{lemma:empconverge}
There exists constants $\kappa_a\in (\frac{\sqrt{3}}{2}, 1), \kappa_{b}\in (0, 1);c_b>0$ and
$$\de_a\in (0, \min\{1, \frac{\sqrt{3}}{2}\|\stt\|, \frac{(1-\kappa_b)^2(1-(\kappa_a)^2)\|\stt\|^2}{4c_b}\})$$ only depending on $\stt$, such that if the initialization $(\sa{0}, \sb{0})$ satisfies
$$\|\sa{0}\|\leq \de_a, \quad \text{and}\quad \|\sb{0}-\stt\|\leq \sqrt{1-(\kappa_a)^2}\|\stt\|, $$
then $\forall t\geq 0$, we have
\begin{eqnarray}
\|\sa{t+1}\|&\leq& \kappa_a\|\sa{t}\|+\ep_a, \n\\
\|\sb{t+1}-\stt\|&\leq& \kappa_b\|\sb{t}-\stt\|+\sqrt{c_b\|\sa{t}\|}+\ep_b, \n
\end{eqnarray}
with probability at least $1-3\de$. The value of the other constants are the following
\begin{eqnarray}
c_\theta&=&4(\|\stt\|+2)\sqrt{\frac{3d+\ln(1/\de)}{n}}, \n\\
 C_\theta&=&3\|\stt\|c_\theta, \n\\
\ep_a=\ep_b&=&\frac{9C_\theta+c_\theta}{\rho(2-\rho)}+\frac{12C_\theta}{\rho(2-\rho)}+\frac{c_\theta}{2-\rho}, \n\\
\end{eqnarray}
where $\rho=\sup_{\|\vx_a\|\leq 1, \|\vx_b\|\leq \frac{3}{2}\|\stt\|}\max\{\funcP(\vx_a, \vx_b, \stt), 1-\funcP(\vx_a, \vx_b, \stt)\}\in (0, 1)$. The function $\funcP$ is defined in Appendix \ref{sec:notation}.
In addition, assume $n$ is large enough to satisfy the following conditions:
\begin{eqnarray}
C_\theta&<&\frac{\rho}{2}, \n\\
\ep_a=\ep_b&\leq&\min\{ (1-\kappa_a)\de_a, \frac{1}{2}(1-\kappa_b)\sqrt{1-(\kappa_a)^2}\|\stt\|\}.\n\\\label{eq:conditiononn}
\end{eqnarray}
\end{lemma}

\begin{proof}
Note that
\begin{eqnarray}
\w_d(\y-\x_a, \x_b)&\triangleq&\frac{e^{\dotp{\y-\x_a, \x_b}}}{e^{\dotp{\y-\x_a, \x_b}}+e^{-\dotp{\y-\x_a, \x_b}}}.\n
\end{eqnarray}

We showed the following equations in Appendix \ref{sec:proofsampleEM-populationfiniett}:

\begin{eqnarray}
\sa{t+1}&=&\frac{\sq{t+1}(1-2\sp{t+1})}{2\sp{t+1}(1-\sp{t+1})}+\frac{\bar{\y}}{2(1-\sp{t+1})}, \n\\
\sb{t+1}&=&\frac{\sq{t+1}}{2\sp{t+1}(1-\sp{t+1})}-\frac{\bar{\y}}{2(1-\sp{t+1})}.\n
\end{eqnarray}
where
\begin{eqnarray}
\bar{\y}&=&\frac{1}{n}\sum_{i=1}^n\y_i, \n\\
\sq{t+1} &=&\frac{1}{n}\sum_{i=1}^n\w_d(\y_i-\sa{t}, \sb{t})\y_i, \n\\
\sp{t+1} &=&\frac{1}{n}\sum_{i=1}^n\w_d(\y_i-\sa{t}, \sb{t}), \n
\end{eqnarray}

We will show in Appendix \ref{sssec:proofequconc} that with probability at least $1-3\de$, we have
\begin{eqnarray}
\|\frac{1}{n}\sum_{i=1}^n\y_i\|\leq 4(\|\stt\|+2)\sqrt{\frac{3d+\ln(1/\de)}{n}}= c_\theta, \label{eq:probcon1}
\end{eqnarray}
\begin{eqnarray}
&~&\sup_{\|\x_b\|\leq \frac{3}{2}\|\stt\|, \|\x_a\|\leq 1}|\frac{1}{n}\sum_{i=1}^n\w_d(\y_i-\x_a, \x_b)-\bbE_Y\w_d(\y-\x_a, \x_b)|\leq  C_\theta, \label{eq:uitselfconc}\\
&~&\sup_{\|\x_b\|\leq \frac{3}{2}\|\stt\|, \|\x_a\|\leq 1}\|\frac{1}{n}\sum_{i=1}^n(\w_d(\y_i-\x_a, \x_b)-\frac{1}{2})\y_i-\bbE (\w_d(\y-\x_a, \x_b)-\frac{1}{2})\y\| \leq \frac{9}{2}C_\theta.\n\\\label{equ:epcon}
\end{eqnarray}
Note that by setting $\delta = \frac{1}{n}$, we see that $c_\theta \rightarrow 0$, $C_\theta \rightarrow 0$, and $\delta \rightarrow 0$ simultaneously.
In the rest of the proof we assume that \eqref{eq:probcon1}, \eqref{eq:uitselfconc} and \eqref{equ:epcon} hold. Let
$$\bpq{t+1}=\bbE\w_d(\y-\sa{t}, \sb{t})\y, \quad \  \bpp{t+1}=\bbE \w_d(\y-\sa{t}, \sb{t}), $$
and
$$\bpa{t+1}=\frac{\bpq{t+1}(1-2\bpp{t+1})}{2\bpp{t+1}(1-\bpp{t+1})}, \quad \ \bpb{t+1}=\frac{\bpq{t+1}}{2\bpp{t+1}(1-\bpp{t+1})}.$$

The following lemma that will be proved in Appendix \ref{sec:proof:lem:expconvalast} is a key step in our analysis:
\begin{lemma}\label{lemma:connecpopsamp}
There exists $\kappa_a\in (0, 1)$ such that if $\|\sb{t}-\stt\|\leq \min\{\sqrt{1-(\kappa_a)^2}, \frac{1}{2}\}\|\stt\|$, then
\begin{eqnarray}
\|\bpa{t+1}\|&\leq& \kappa_a\|\sa{t}\|, \label{equ:empcon12}
\end{eqnarray}
Furthermore, there exist $\de_a'\in(0, 1)$, $\kappa_b\in (0, 1)$ and $c_b>0$ such that if $\|\sa{t}\|\in [0, \de_a']$, then
\begin{eqnarray}
\|\bpb{t+1}-\stt\|&\leq&\kappa_b\|\sb{t}-\stt\|+\sqrt{c_{b}\|\sa{t}\|}.\label{equ:empcon1}
\end{eqnarray}
Constant $\kappa_a, \kappa_b, \de_a'$ and $c_b$ only depend on $\stt$.
\end{lemma}

The above equations provide connections between $(\bpa{t+1}, \bpb{t+1})$ and $(\sa{t}, \sb{t})$. Next, we establish connection between $(\bpa{t+1}, \bpb{t+1})$ and $(\sa{t+1}, \sb{t+1})$. In the rest of the proof we assume that $\kappa_a \in (\sqrt{3}/2, 1)$. If $\kappa_a$ is less than $\sqrt{3}/2$ we set it to $\sqrt{3}/2$. This is just for making notations simpler and has no specific technical reason.

Note that from \eqref{equ:iteha}, we have
\begin{eqnarray}
\lefteqn{\|\sa{t+1}\|=\|\frac{\sq{t+1}(1-2\sp{t+1})}{2\sp{t+1}(1-\sp{t+1})}+\frac{\bar{y}}{2(1-\sp{t+1})}\|}\n\\
&\leq&\|\frac{\sq{t+1}(1-2\sp{t+1})}{2\sp{t+1}(1-\sp{t+1})}\|+\|\frac{\bar{y}}{2(1-\sp{t+1})}\|\n\\
&\leq&\left|\frac{(1-2\sp{t+1})}{2\sp{t+1}(1-\sp{t+1})}\right|\|\sq{t+1}-\bpq{t+1}\|+\|\frac{\bpq{t+1}(1-2\bpp{t+1})}{2\bpp{t+1}(1-\bpp{t+1})}\|+\|\frac{\bar{y}}{2(1-\sp{t+1})}\|\n\\
&~&+\left\|\frac{\bpb{t+1}((\bpp{t+1})^2+(1-\bpp{t+1})^2-(1-2\bpp{t+1})|\sp{t+1}-\bpp{t+1}|)}{\sp{t+1}(1-\sp{t+1})}\right\||\sp{t+1}-\bpp{t+1}|\n\\
&\leq&\left|\frac{1}{2\sp{t+1}(1-\sp{t+1})}\right|\|\sq{t+1}-\bpq{t+1}\|+\left\|\frac{3\bpb{t+1}}{\sp{t+1}(1-\sp{t+1})}\right\||\sp{t+1}-\bpp{t+1}|\n\\
&&+\|\bpa{t+1}\|+\|\frac{\bar{y}}{2(1-\sp{t+1})}\|.\label{equ:difa}
\end{eqnarray}
Furthermore, from \eqref{equ:itehb} we have
\begin{eqnarray}
\lefteqn{\|\sb{t+1}-\bpb{t+1}\|=\|\frac{\sq{t+1}}{2\sp{t+1}(1-\sp{t+1})}-\frac{\bar{y}}{2(1-\sp{t+1})}-\frac{\bpq{t+1}}{2\bpp{t+1}(1-\bpp{t+1})}\|}\n\\
&\leq&\|\frac{\sq{t+1}}{2\sp{t+1}(1-\sp{t+1})}-\frac{\bpq{t+1}}{2\bpp{t+1}(1-\bpp{t+1})}\|+\|\frac{\bar{y}}{2(1-\sp{t+1})}\|\n\\
&\leq&\left|\frac{1}{2\sp{t+1}(1-\sp{t+1})}\right|\|\sq{t+1}-\bpq{t+1}\|+\|\frac{\bar{y}}{2(1-\sp{t+1})}\|\n\\
&~&+\left\|\frac{\bpb{t+1}(1-2\bpp{t+1}+|\sp{t+1}-\bpp{t+1}|)}{\sp{t+1}(1-\sp{t+1})}\right\||\sp{t+1}-\bpp{t+1}|\n\\
&\leq&\left|\frac{1}{2\sp{t+1}(1-\sp{t+1})}\right|\|\sq{t+1}-\bpq{t+1}\|+\left\|\frac{3\bpb{t+1}}{\sp{t+1}(1-\sp{t+1})}\right\||\sp{t+1}-\bpp{t+1}|+\|\frac{\bar{y}}{2(1-\sp{t+1})}\|.\n\\\label{equ:difb}
\end{eqnarray}
Suppose for the moment that $\|\sa{t}\|\in [0, 1]$ and $\|\sb{t}-\stt\|\leq \frac{1}{2}\|\stt\|$.
It is straightforward to use \eqref{eq:uitselfconc} and \eqref{eq:conditiononn} and the definition of $\rho$ in the statement of Lemma \ref{lemma:empconverge} to prove
\begin{equation}\label{eq:pconcent}
\sp{t+1}\in (\frac{\rho}{2}, 1-\frac{\rho}{2}).
\end{equation}
By combining \eqref{eq:probcon1}-\eqref{equ:epcon}, \eqref{equ:difa}), \eqref{equ:difb}, and \eqref{eq:pconcent} we obtain
\begin{eqnarray}\label{eq:hatandtildedisc}
\|\sa{t+1}\|&\leq&\|\bpa{t+1}\|+\frac{9C_\theta+c_\theta}{\rho(2-\rho)}+\frac{12C_\theta}{\rho(2-\rho)}+\frac{c_\theta}{2-\rho}=\|\bpa{t+1}\|+\ep_a, \n\\
\|\sb{t+1}-\bpb{t+1}\|&\leq&\frac{9C_\theta+c_\theta}{\rho(2-\rho)}+\frac{12C_\theta}{\rho(2-\rho)}+\frac{c_\theta}{2-\rho}=\ep_b,
\end{eqnarray}
and hence $\|\sb{t+1}-\stt\|\leq \|\bpb{t+1}-\stt\|+\ep_b$.

Now suppose that the assumptions of Lemma \ref{lemma:connecpopsamp} hold, i.e., $\|\sa{t}\|\in [0, \de_a']$ and $\|\sb{t}-\stt\|\leq \sqrt{1-(\kappa_a)^2}\|\stt\|$. Then \eqref{eq:hatandtildedisc} implies that
\begin{eqnarray}
\|\sa{t+1}\|&\leq&\|\bpa{t+1}\|+\ep_a\leq \kappa_a\|\sa{t}\|+\ep_a, \n\\
\|\sb{t+1}-\stt\|&\leq&\|\bpb{t+1}-\stt\|+\ep_b\leq \kappa_b\|\sb{t}-\stt\|+\sqrt{c_{b}\|\sa{t}\|}+\ep_b.\label{equ:empcon2}
\end{eqnarray}
Note that \eqref{equ:empcon2} is the result we claimed in Lemma \ref{lemma:empconverge}.  However, to obtain \eqref{equ:empcon1}, which is one of the main steps in deriving \eqref{equ:empcon2} we have assumed that
\begin{eqnarray}
\|\sa{t}\|\in [0, \de_a'] \quad \text{and} \quad \|\sb{t}-\stt\|\leq  \sqrt{1-(\kappa_a)^2}\|\stt\|\n.
\end{eqnarray}
In order to prove the above equation holds for every $t$, we will prove an even stronger statement:
\begin{eqnarray}
\|\sa{t}\|\in [0, \de_a] \quad \text{and} \quad \|\sb{t}-\stt\|\leq  \sqrt{1-(\kappa_a)^2}\|\stt\|, \label{equ:empcon3}
\end{eqnarray}
where $\de_a=\min\{\de_a', \frac{(1-\kappa_b)^2(1-(\kappa_a)^2)\|\stt\|^2}{4c_b}\}$.
We use induction to prove that \eqref{equ:empcon3} holds $\forall t\geq 0$.  By the assumptions of this Lemma, the initial estimates $(\sa{0}, \sb{0})$ satisfy \eqref{equ:empcon3}. Hence the base of the induction is true. Suppose \eqref{equ:empcon3} holds for $t\geq 0$, then for $t+1$ \eqref{equ:empcon2} holds. Hence all we need to prove is that
\[
\kappa_a\|\sa{t}\|+\ep_a\leq \de_a, \]
and
\begin{equation}\label{eq:secondforind}
\kappa_b\|\sb{t}-\stt\|+\sqrt{c_{b}\|\sa{t}\|}+\ep_b\leq \sqrt{1-(\kappa_a)^2}\}\|\stt\|.
\end{equation}
For the first inequality, since the condition on $n$ in \eqref{eq:conditiononn} ensure that $\ep_a\leq (1-\kappa_a)\de_a$, together with induction assumption that $\|\sa{t}\|\leq \de_a$, we have
\[
\|\sa{t+1}\|\leq \kappa_a\|\sa{t}\|+\ep_a\leq \kappa_a\de_a+(1-\kappa_a)\de_a\leq \de_a.\]
To prove  \eqref{eq:secondforind} note that the condition on $n$ ensure that
\[\ep_b\leq \frac{1}{2}(1-\kappa_b)\sqrt{1-(\kappa_a)^2}\|\stt\|.\]
Also the condition on $\de_a$ and $\|\sa{t}\|\leq \de_a$ ensure that
\[\sqrt{c_{b}\|\sa{t}\|}\leq \frac{1}{2}(1-\kappa_b)\sqrt{1-(\kappa_a)^2}\|\stt\|.\]
Hence with induction assumption that $\|\sb{t}-\stt\|\leq \sqrt{1-(\kappa_a)^2}\|\stt\|$, we have
\begin{eqnarray}
\|\sb{t+1}-\stt\|&\leq&\kappa_b\|\sb{t}-\stt\|+\sqrt{c_{b}\|\sa{t}\|}+\ep_b\n\\
&\leq&\kappa_b\sqrt{1-(\kappa_a)^2}\|\theta\|+ \frac{1}{2}(1-\kappa_b)\sqrt{1-(\kappa_a)^2}\|\stt\|+ \frac{1}{2}(1-\kappa_b)\sqrt{1-(\kappa_a)^2}\|\stt\|\n\\
&=&\sqrt{1-(\kappa_a)^2}\|\stt\|.\n
\end{eqnarray}
Hence the second  part of \eqref{equ:empcon3} holds for $t+1$. This completes the proof.
\end{proof}

\subsubsection{Proof of Lemma \ref{lem:expavoidserroracc}} \label{sec:proofexpoavoidserroracc}

We first prove \eqref{eq:result1} for $\|\sa{t}\|$. Clearly the result holds for $t=0$. For all $t\geq 1$, using the condition \eqref{eq:condition1} on $\|\sa{t'}\|$ for all $t'\leq t$, we have
\begin{eqnarray}
\|\sa{t}\|&\leq& \kappa_a\|\sa{t-1}\|+\ep_a\n\\
&\leq& \kappa_a(\kappa_a\|\sa{t-2}\|+\ep_a)+\ep_a\n\\
&\leq& (\kappa_a)^t\|\sa{0}\|+\ep_a\sum_{i=0}^{t-1}(\kappa_a)^i\n\\
&\leq&(\kappa_a)^t\|\sa{0}\|+\frac{1}{1-\kappa_a}\ep_a.\n
\end{eqnarray}
Hence \eqref{eq:result1} holds. Next, we prove \eqref{eq:result2} for $\|\sb{t}\|$. Clearly the result holds for $t=0$. For all $t\geq 1$, using the condition \eqref{eq:condition2} on $\|\sb{t'}\|$ for all $t'\leq t$, we have
\begin{eqnarray}
\|\sb{t}-\stt\|&\leq& \kappa_b\|\sb{t-1}-\stt\|+\sqrt{c_b\|\sa{t-1}\|}+\ep_b\n\\
&\leq&\kappa_b(\kappa_b\|\sb{t-2}-\stt\|+\sqrt{c_b\|\sa{t-2}\|}+\ep_b)+\ep_b\n\\
&\leq&(\kappa_b)^t\|\sb{0}-\stt\|+\sqrt{c_b}\sum_{i=0}^{t-1}(\kappa_b)^{t-1-i}\sqrt{\|\sa{i}\|}+\ep_b\sum_{i=0}^{t-1}(\kappa_b)^i\n\\
&\leq&(\kappa_b)^t\|\sb{0}-\stt\|+\sqrt{c_b}\sum_{i=0}^{t-1}(\kappa_b)^{t-1-i}\sqrt{\|\sa{i}\|}+\frac{1}{1-\kappa_b}\ep_b.\n
\end{eqnarray}
From \eqref{eq:result1}, we have $\forall t\geq 0$,
\begin{eqnarray}
\sqrt{\|\sa{t}\|}\ \leq\ \sqrt{(\kappa_a)^t\|\sa{0}\|+\frac{1}{1-\kappa_a}\ep_a}\ \leq\  (\kappa_a)^{\frac{t}{2}}\sqrt{\|\sa{0}\|}+\sqrt{\frac{1}{1-\kappa_a}\ep_a}.\n
\end{eqnarray}
Hence we have
\begin{eqnarray}
\|\sb{t}-\stt\|&\leq&(\kappa_b)^t\|\sb{0}-\stt\|+\sqrt{c_b}\sum_{i=0}^{t-1}(\kappa_b)^{t-1-i}\sqrt{\|\sa{i}\|}+\frac{1}{1-\kappa_b}\ep_b\n\\
&\leq&(\kappa_b)^t\|\sb{0}-\stt\|+\sqrt{c_b}\sum_{i=0}^{t-1}(\kappa_b)^{t-1-i}((\sqrt{\kappa_a})^{i}\sqrt{\|\sa{0}\|}+\sqrt{\frac{1}{1-\kappa_a}\ep_a})+\frac{1}{1-\kappa_b}\ep_b\n\\
&=&(\kappa_b)^t\|\sb{0}-\stt\|+\sqrt{c_b\|\sa{0}\|}\sum_{i=0}^{t-1}(\kappa_b)^{t-1-i}(\sqrt{\kappa_a})^{i}+\sqrt{\frac{c_b}{1-\kappa_a}\ep_a}\sum_{i=0}^{t-1}(\kappa_b)^{i}+\frac{1}{1-\kappa_b}\ep_b\n\\
&\leq&(\kappa_b)^t\|\sb{0}-\stt\|+t\sqrt{c_b\|\sa{0}\|}(\max\{\sqrt{\kappa_a}, \kappa_b\})^t+\frac{1}{1-\kappa_b}\sqrt{\frac{c_b}{1-\kappa_a}\ep_a}+\frac{1}{1-\kappa_b}\ep_b.\n
\end{eqnarray}
This completes the proof of this lemma.

\subsubsection{Proof of \eqref{eq:probcon1}-\eqref{equ:epcon}}\label{sssec:proofequconc}

\begin{lemma}\label{lemma:prob}
Let $y_1, \cdots, y_n \overset{i.i.d.}{\sim} \frac{1}{2}N(\stt, \vI_d)+\frac{1}{2}N(-\stt, \vI_d)$. Then, we have
\begin{itemize}
\item[(1)] $\|\frac{1}{n}\sum_{i=1}^n\y_i\|\leq 4(\|\stt\|+1)\sqrt{\frac{2d+\ln(1/\de)}{n}}, $ with probability at least $1-\de$.
\item[(2)] $\sup_{\|\vx_b\|\leq c, \|\vx_a\|\leq 1} |\frac{1}{n}\sum_{i=1}^n\w_d(\y_i-\vx_a, \vx_b)-\bbE \w_d(\Y-\vx_a, \vx_b)|\leq 8c(\|\stt\|+2)\sqrt{\frac{d+2+\ln(1/\de)}{n}}, $
with probability at least $1-\de$.
\item[(3)] $\sup_{\|\vx_b\|\leq c, \|\vx_a\|\leq 1}\|\frac{1}{n}\sum_{i=1}^n(\w_d(\y_i-\vx_a, \vx_b)-\frac{1}{2})\y_i-\bbE(\w_d(\Y-\vx_a, \vx_b)-\frac{1}{2})\Y\|\leq 36c(\|\stt\|+2)\sqrt{\frac{d+2+\ln(1/\de)}{n}}, $ with probability at least $1-\de$.
\end{itemize}
\end{lemma}
\begin{proof}
We first prove the first claim (1). Note that $\y_i$ can be expressed by $\y_i=\zeta_i\stt+\boldsymbol{\omega}_i$, where $\zeta_i$ are i.i.d sequence of Rademacher variables and $\boldsymbol{\omega}_i$ are i.i.d  $N(\v0, \vI_d)$ Gaussian random variables. Therefore we have,
\begin{eqnarray}
\|\frac{1}{n}\sum_{i=1}^n\y_i\|^2&=&\|\frac{1}{n}\sum_{i=1}^n\zeta_i\stt+\frac{1}{n}\sum_{i=1}^n\boldsymbol{\omega}_i\|^2\n\\
&=&\frac{1}{n}\|\frac{1}{\sqrt{n}}\sum_{i=1}^n\zeta_i\stt+\frac{1}{\sqrt{n}}\sum_{i=1}^n\boldsymbol{\omega}_i\|^2.\n
\end{eqnarray}
Note that $\|\frac{1}{\sqrt{n}}\sum_{i=1}^n\boldsymbol{\omega}_i\|^2\overset{dist.}{=}\boldsymbol{\nu}, $ where $\boldsymbol{\nu}\sim \chi^2(d)$. Hence, using Cram\'{e}r-Chernoff inequality, we have probability at least $1-\frac{\de}{2}$ such that
$$\left|\|\frac{1}{\sqrt{n}}\sum_{i=1}^n\boldsymbol{\omega}_i\|^2-d\right|\leq \sqrt{8d\ln(2/\de)}\leq d+2\ln(2/\de)\ \text{for sufficiently large}\ n.$$
Moreover, for Rademacher variables $\zeta_i$, using Hoeffding's inequality, we have with probability at least $1-\frac{\de}{2}$ such that
$$|\frac{1}{\sqrt{n}}\sum_{i=1}^n\zeta_i|\leq \sqrt{2\ln(2/\de)}$$
Therefore, we have probability at least $1-\de$ such that
\begin{eqnarray}
\|\frac{1}{n}\sum_{i=1}^n\y_i\|&\leq& \frac{1}{\sqrt{n}}\|\frac{1}{\sqrt{n}}\sum_{i=1}^n\zeta_i\stt+\frac{1}{\sqrt{n}}\sum_{i=1}^n\boldsymbol{\omega}_i\|\n\\
&\leq&\frac{1}{\sqrt{n}}\sqrt{2\|\frac{1}{\sqrt{n}}\sum_{i=1}^n\zeta_i\stt\|^2+2\|\frac{1}{\sqrt{n}}\sum_{i=1}^n\boldsymbol{\omega}_i\|^2}\n\\
&\leq&\frac{1}{\sqrt{n}}\sqrt{ 2(2\ln(2/\de))\|\stt\|^2+2(2d+2\ln{2/\de})}\n\\
&=&2\sqrt{\frac{\ln(2/\de)(\|\stt\|^2+1)+d}{n}}\n\\
&\leq& 4(\|\stt\|+1)\sqrt{\frac{2d+\ln(1/\de)}{n}}.\n
\end{eqnarray}
For the second claim, define $$Z_{+}\triangleq \sup_{\|\vx_b\|\leq c, \|\vx_a\|\leq 1}\frac{1}{n}\sum_{i=1}^n\w_d(\y_i-\vx_a, \vx_b)-\bbE \w_d(\Y-\vx_a, \vx_b).$$ Then we have $\forall \|\vx_b\|\leq c, \|\vx_a\|\leq 1$
\begin{eqnarray}
\bbE e^{\la Z_{+}}&\overset{(i)}{\leq}&\bbE_{\Y, \Y'}e^{\la\sup_{\|\vx_b\|\leq c, \|\vx_a\|\leq 1}\frac{1}{n}\sum_{i=1}^n(\w_d(\y_i-\vx_a, \vx_b)-\w_d(\y_i'-\vx_a, \vx_b))} \n\\
&=&\bbE_{\Y, \Y', \xi}e^{\la\sup_{\|\vx_b\|\leq c, \|\vx_a\|\leq 1}\frac{1}{n}\sum_{i=1}^n\xi_i(\w_d(\y_i-\vx_a, \vx_b)-\w_d(\y_i'-\vx_a, \vx_b))}\n\\
&\leq&\bbE_{\Y, \Y', \xi}e^{\la\sup_{\|\vx_b\|\leq c, \|\vx_a\|\leq 1}|\frac{1}{n}\sum_{i=1}^n\xi_i(\w_d(\y_i-\vx_a, \vx_b)-\w_d(\y_i'-\vx_a, \vx_b))|}\n\\
&\leq & \bbE_{\xi} \{\bbE_{\Y} \left(e^{\la\sup_{\|\vx_b\|\leq c, \|\vx_a\|\leq 1}|\frac{1}{n}\sum_{i=1}^n\xi_i(\w_d(\y_i-\vx_a, \vx_b)-\frac{1}{2})|} \right)\n\\
 &&\times \mathbb{E}_{\Y'} \left( e^{\la\sup_{\|\vx_b\|\leq c, \|\vx_a\|\leq 1}|\frac{1}{n}\sum_{i=1}^n\xi_i(\w_d(\y_i'-\vx_a, \vx_b)-\frac{1}{2})|} \right)\}\n \\
&\leq&\bbE_{\Y, \xi}e^{2\la\sup_{\|\vx_b\|\leq c, \|\vx_a\|\leq 1}|\frac{1}{n}\sum_{i=1}^n\xi_i(\w_d(\y_i-\vx_a, \vx_b)-\frac{1}{2})|}, \n
\end{eqnarray}
Note that to obtain Inequality (i) we have used Jensen's inequality. Also, $\xi_i$ are i.i.d sequence of Rademacher variables. To simplify the final expression even further, we use the following lemma from \cite{VKoltchinskii2011tailbound}
\begin{lemma}\label{lemma:symmetrize}
Let $\mathcal{H}\in\bbR^n$ and let $\psi_i: \bbR\mapsto \bbR, i=1, \cdots, n$ be functions such that $\psi_i(0)=0$ and
$$|\psi_i(u)-\psi_i(v)|\leq|u-v|\in \bbR.$$
For all convex nondecreasing functions $\Psi:\bbR_+\mapsto \bbR_+$,
$$\bbE\Psi(\frac{1}{2}\sup_{\vh\in \mathcal{H}}|\sum_{i=1}^n\psi_i(h_i)\ep_i|)\leq \bbE\Psi(\sup_{\vh\in \mathcal{H}}|\sum_{i=1}^nh_i\ep_i|), $$
where $\ep_i$ are i.i.d.~Rademacher random variables.
\end{lemma}
\noindent Since $\w_d(\y-\vx_a, \vx_b)$ is a function of $\dotp{\y-\vx_a, \vx_b}$ and
$$|\w_d(\y-\vx_a, \vx_b)-\w_d(\y-\vx_a', \vx_b')|\leq \frac{1}{2}|\dotp{\y-\vx_a, \vx_b}-\dotp{\y-\vx_a', \vx_b'}|, $$
letting $\Psi(x)=e^{2\la x}$ and $\psi_i(x)=\frac{2e^x}{e^x+e^{-x}}-1$ with $h_i=\dotp{\y_i-\vx_a, \vx_b}$
in Lemma \ref{lemma:symmetrize},
we have
\begin{eqnarray}
\bbE e^{\la Z_{+}}&\leq&\bbE_{\Y, \xi}e^{2\la\sup_{\|\vx_b\|\leq c, \|\vx_a\|\leq 1}|\frac{1}{n}\sum_{i=1}^n\xi_i(\w_d(\y_i, \vx_a, \vx_b)-\frac{1}{2})|}\n\\
&\leq&\bbE_{Y, \xi}e^{2\la\sup_{\|\vx_b\|\leq c, \|\vx_a\|\leq 1}|\frac{1}{n}\sum_{i=1}^n\xi_i\dotp{\y_i-\vx_a, \vx_b}|}\n\\
&\stackrel{(ii)}{=}&\bbE_{\Y, \xi}e^{2\la\sup_{\|\vx_b\|\leq c, \|\vx_a\|\leq 1}\frac{1}{n}\sum_{i=1}^n\xi_i\dotp{\y_i-\vx_a, \vx_b}}\n\\
&\leq&\bbE_{\Y, \xi}e^{2\la\sup_{\|\vx_b\|\leq c}\frac{1}{n}\sum_{i=1}^n\xi_i\dotp{\y_i, \vx_b}}e^{2\la\sup_{\|\vx_b\|\leq c, \|\vx_a\|\leq 1}\dotp{\vx_a, \vx_b}\frac{1}{n}\sum_{i=1}^n\xi_i}\n\\
&\leq&\bbE_{\Y, \xi}e^{2\la c\|\frac{1}{n}\sum_{i=1}^n\xi_i\y_i\|}e^{2\la c|\frac{1}{n}\sum_{i=1}^n\xi_i|}\n\\
&\leq&(\bbE_{\Y, \xi}e^{4\la c\|\frac{1}{n}\sum_{i=1}^n\xi_i\y_i\|})^{1/2}(\bbE_{\xi}e^{4\la c|\frac{1}{n}\sum_{i=1}^n\xi_i|})^{1/2}\n\\
&\leq&(\underbrace{\bbE_{\Y}e^{4\la c\|\frac{1}{n}\sum_{i=1}^n\y_i\|}}_{part~1})^{1/2}(\underbrace{\bbE_{\xi}e^{4\la c|\frac{1}{n}\sum_{i=1}^n\xi_i|}}_{part~2})^{1/2}, \n
\end{eqnarray}
where last equality holds for the fact that the distribution of $\y_i$ is symmetric and equality (ii) holds for the fact that $\frac{1}{n}\sum_{i=1}^n\xi_i\dotp{\y_i-\vx_a, \vx_b}$ is symmetric in terms of $\vx_b$ and the constraints on $\vx_b$ is symmetric.

For part 1,  we use the notation $\{\vu_j, j=1, \cdots, M\}$ for a $1/2$-covering of the $d$-dimensional sphere, $Sp^d\triangleq \{\vv\in\bbR^d, \|\vv\|=1\}$.
Note that, for all $ \vv', \vv\in Sp^d$,
$$|\frac{1}{n}\sum_{i=1}^n\dotp{\y_i, \vv'}-\frac{1}{n}\sum_{i=1}^n\dotp{\y_i, \vv}|\leq \|\vv'-\vv\|\sup_{\|\vu\|=1}\frac{1}{n}\sum_{i=1}^n\dotp{\y_i, \vu}, $$
therefore, we have for all $ \vu\in Sp^d$
\begin{eqnarray}
\frac{1}{n}\sum_{i=1}^n\dotp{\y_i, \vu}\leq \max_{j\in[M]}\{\frac{1}{n}\sum_{i=1}^n\dotp{\vy_i, \vu_j}\}+\|\vu_j-\vu\|\sup_{\|\vu\|=1}\frac{1}{n}\sum_{i=1}^n\dotp{\vy_i, \vu}, \n
\end{eqnarray}
and hence
\begin{eqnarray}
\frac{1}{n}\|\sum_{i=1}^n\y_i\|=\sup_{\|\vu\|=1}\frac{1}{n}\sum_{i=1}^n\dotp{\vy_i, \vu}\leq 2\max_{j\in[M]}\{\frac{1}{n}\sum_{i=1}^n\dotp{\vy_i, \vu_j}\}. \label{equ:eqlast1}
\end{eqnarray}
recall that $\y_i=\zeta_i\stt+\boldsymbol{\omega}_i$. Hence, we have
\begin{eqnarray}
&&\bbE_{\Y} e^{\dotp{\y_i, \vu_j}}\ =\ \bbE_{\zeta}e^{\zeta\dotp{\stt, \vu_j}}\bbE_{\boldsymbol{\omega}} e^{\dotp{\boldsymbol{\omega}_i, \vu_j}}\n\\
&=&\frac{1}{2}(e^{\dotp{\stt, \vu_j}}+e^{-\dotp{\stt, \vu_j}})e^{\frac{1}{2}}\ \leq\ e^{\frac{\|\stt\|^2+1}{2}}, \label{equ:dotpyu}
\end{eqnarray}
where last inequality holds because of
$$\frac{1}{2}(e^{\|\stt\|}+e^{-\|\stt\|})\leq e^{\frac{\|\stt\|^2}{2}}.$$
Therefore we have
\begin{eqnarray}
\bbE_{\Y, \xi}e^{4\la c\|\frac{1}{n}\sum_{i=1}^n\y_i\|}&=&\bbE_{\Y, \xi}e^{4\la c\sup_{\|\vu\|=1}\frac{1}{n}\sum_{i=1}^n\dotp{\y_i, \vu}}\n\\
&\leq&\bbE_{\Y, \xi}e^{8\la c\max_{j\in[M]}\frac{1}{n}\sum_{i=1}^n\dotp{\y_i, \vu_j}}\n\\
&\leq&\sum_{j=1}^M\bbE_{\Y, \xi}e^{8\la c\frac{1}{n}\sum_{i=1}^n\dotp{\y_i, \vu_j}}\n\\
&\leq&e^{32\la^2c^2\frac{\|\stt\|^2+1}{n}+2d}.\label{equ:eqlast3}
\end{eqnarray}
For part 2, notice that $\frac{1}{n}\sum_{i=1}^n\xi_i$ is symmetric, we have
\begin{eqnarray}
\bbE_{\xi}e^{4\la c|\frac{1}{n}\sum_{i=1}^n\xi_i|}&\leq&2\bbE_{\xi}e^{4\la c\frac{1}{n}\sum_{i=1}^n\xi_i}\n\\
&\leq&2(\bbE_{\xi}e^{\frac{4\la c}{n}\xi})^n\n\\
&\leq&e^{\frac{8\la^2c^2}{n}+1}.\label{equ:eqlast2}
\end{eqnarray}
Therefore combining \eqref{equ:eqlast3} and \eqref{equ:eqlast2}, we have
\begin{eqnarray}
\bbE e^{\la Z_{+}}&\leq&e^{16\la^2c^2\frac{\|\stt\|^2+1}{n}+d}\times e^{\frac{4\la^2c^2}{n}+\frac{1}{2}}\n\\
&\leq&e^{16\la^2c^2\frac{\|\stt\|^2+2}{n}+d+\frac{1}{2}}.\n
\end{eqnarray}
Using Markov Inequality:
$$P(Z_+>\ep)\leq\bbE e^{\la Z_{+}-\la\ep}, \forall \ep, \la>0, $$
choosing $\la=\frac{\ep n}{32c^2(\|\stt\|^2+2)}$, we have
\begin{eqnarray}
P(Z_+>\ep)&\leq& e^{\frac{16c^2\la^2(\|\stt\|^2+2)}{n}+d+\frac{1}{2}-\la \ep}\n\\
&=&e^{-\frac{n\ep^2}{64c^2(\|\stt\|^2+2)}+d+\frac{1}{2}}.\n
\end{eqnarray}
Therefore
\begin{eqnarray}
|\sup_{\|\vx_b\|\leq c, \|\vx_a\|\leq 1}\frac{1}{n}\sum_{i=1}^n\w_d(\y_i-\vx_a, \vx_b)-\bbE\w_d(\Y-\vx_a, \vx_b)|\leq 8c(\|\stt\|+2)\sqrt{\frac{d+2+\ln(1/\de)}{n}}, \n
\end{eqnarray}
with probability at least $1-\de$.

For the last claim, we borrow a technique in the proof of corollary 2 in B.2 in \cite{balakrishnan2014statistical}. Let $$Z=\sup_{\|\vx_b\|\leq c, \|\vx_a\|\leq 1}\|\frac{1}{n}\sum_{i=1}^n(\w_d(\y_i-\vx_a, \vx_b)-\frac{1}{2})\y_i-\bbE (\w_d(\Y-\vx_a, \vx_b)-\frac{1}{2})\Y\|, $$
and
$$Z_{\vu}=\sup_{\|\vx_b\|\leq c, \|\vx_a\|\leq 1}|\frac{1}{n}\sum_{i=1}^n(\w_d(\y_i-\vx_a, \vx_b)-\frac{1}{2}))\dotp{\y_i, \vu}-\bbE (\w_d(\Y-\vx_a, \vx_b)-\frac{1}{2})\dotp{\Y, \vu}|$$
we have
\begin{eqnarray}
\bbE e^{\la Z}&=&\bbE_Ye^{\la\sup_{\|\vu\|=1}Z_{\vu}}\leq\bbE e^{2\la\max_{j\in[M]}Z_{\vu_j}}\leq\sum_{j=1}^M\bbE e^{2\la Z_{\vu_j}}\n\\
&\leq&\sum_{j=1}^M\bbE_{\Y, \xi}e^{4\la\sup_{\|\vx_b\|\leq c, \|\vx_a\|\leq 1}\frac{1}{n}\sum_{i=1}^n\xi_i(\w_d(\y_i-\vx_a, \vx_b)-\frac{1}{2})\dotp{\y_i, u_{j}}}, \n
\end{eqnarray}
where $\xi_i$ are i.i.d.~sequence of Rademacher variables and the last inequality holds for standard symmetrization result for empirical process.
Since
$$|(2\w_d(\y_i-\vx_a, \vx_b)-1)\dotp{\y_i, \vu_{j}}-(2\w_d(\y_i-\vx_a', \vx_b')-1)\dotp{\y_i, \vu_{j}}|\leq |\dotp{\y_i-\vx_a, \vx_b}-\dotp{\y_i-\vx_a', \vx_b'}|\dotp{\y_i, \vu_{j}}, $$
let $\Psi(x)=e^{2\la x}$ and $\psi_i(x)=(\frac{2e^x}{e^x-e^{-x}}-1)\dotp{\y_i, \vu_{j}}$ with $h_i=\dotp{\y_i-\vx_a, \vx_b}$
 in Lemma \ref{lemma:symmetrize}, we have
\begin{eqnarray}
&~&\bbE e^{\la Z}\leq\sum_{j=1}^M\bbE_{\Y, \xi}e^{4\la\sup_{\|\vx_b\|\leq c, \|\vx_a\|\leq 1}|\frac{1}{n}\sum_{i=1}^n\xi_i\dotp{\y_i-\vx_a, \vx_b}\dotp{\y_i, \vu_{j}}|}\n\\
&\stackrel{iii}{=}&\sum_{j=1}^M\bbE_{\Y, \xi}e^{4\la\sup_{\|\vx_b\|\leq c, \|\vx_a\|\leq 1}\frac{1}{n}\sum_{i=1}^n\xi_i\dotp{\y_i-\vx_a, \vx_b}\dotp{\y_i, \vu_{j}}}\n\\
&\leq&\sum_{j=1}^M\bbE_{\Y, \xi}e^{4\la\sup_{\|\vx_b\|\leq c}\frac{1}{n}\sum_{i=1}^n\xi_i\dotp{\y_i, \vx_b}\dotp{\y_i, \vu_{j}}} e^{4\la\sup_{\|\vx_b\|\leq c, \|\vx_a\|\leq 1}\frac{1}{n}\sum_{i=1}^n\xi_i\dotp{\vx_a, \vx_b}\dotp{\y_i, \vu_{j}}}\n\\
&\leq&\sum_{j=1}^M(\bbE_{\Y, \xi}e^{8\la\sup_{\|\vx_b\|\leq c}\frac{1}{n}\sum_{i=1}^n\xi_i\dotp{\y_i, \vx_b}\dotp{\y_i, \vu_{j}}})^{\frac{1}{2}}(\bbE_{\Y, \xi}e^{8\la\sup_{\|\vx_b\|\leq c, \|\vx_a\|\leq 1}\frac{1}{n}\sum_{i=1}^n\xi_i\dotp{\vx_a, \vx_b}\dotp{\y_i, \vu_{j}}})^{\frac{1}{2}}\n\\
&\leq&\sum_{j=1}^M(\underbrace{\bbE_{\Y, \xi}e^{8\la c\|\frac{1}{n}\sum_{i=1}^n\xi_i\y_i\y_i^\t\|_{op}}}_{part 1})^{\frac{1}{2}} (\underbrace{\bbE_{\Y, \xi}e^{8\la c\frac{1}{n}|\sum_{i=1}^n\xi_i\dotp{\y_i, \vu_{j}}|}}_{part 2})^{\frac{1}{2}}, \n
\end{eqnarray}
where $\|\cdot\|_{op}$ is $l_2$-operator norm of a matrix(maximum singular value), equality (iii) holds for the fact that $\frac{1}{n}\sum_{i=1}^n\xi_i\dotp{\y_i-\vx_a, \vx_b}\dotp{\y_i, \vu_{j}}$ is symmetric in terms of $\vx_b$ and constraints of $\vx_b$ is symmetric. The correctness of the last inequality is shown in B.2 of \cite{balakrishnan2014statistical}. For part 1, as shown in B.2 of \cite{balakrishnan2014statistical} we have
\begin{eqnarray}
\bbE_{\Y, \xi}e^{8\la c\|\frac{1}{n}\sum_{i=1}^n\xi_i\y_i\y_i^\t\|_{op}}\leq\bbE_{\Y, \xi}e^{16\la c\max_{j'\in[M]}\frac{1}{n}\sum_{i=1}^n\xi_i\dotp{\y_i, \vu_{j'}}^2}\n\\
\end{eqnarray}
Recall that $\y_i=\zeta_i\stt+\boldsymbol{\omega}_i$ and \eqref{equ:dotpyu},
we have
$$\bbE_{\Y} e^{\dotp{\y_i, \vu_j}}=\bbE_{\zeta}e^{\zeta\dotp{\stt, \vu_j}}\bbE_{\boldsymbol{\omega}} e^{\dotp{\boldsymbol{\omega}_i, \vu_j}}\leq e^{\frac{\|\stt\|^2+1}{2}}.$$
Therefore
$$\bbE e^{\la\xi\dotp{\y_i, \vu_j}^2}\leq e^{\frac{(\|\stt\|^2+1)\la^2}{2}}, \ \text{for small enough }\ \la.$$
Therefore
\begin{eqnarray}
\bbE_{\Y, \xi}e^{16\la c\max_{j'\in[M]}\frac{1}{n}\sum_{i=1}^n\xi_i\dotp{\y_i, \vu_{j'}}^2}&\leq&\sum_{j'=1}^M\bbE_{\Y, \xi}e^{16\la c\frac{1}{n}\sum_{i=1}^n\xi_i\dotp{\y_i, \vu_{j'}}^2}\n\\
&\leq&e^{\frac{(\|\stt\|^2+1)(16c\la)^2}{2n}+2d}.\n
\end{eqnarray}
For part 2, since $\xi_i\dotp{\y_i, \vu_{j}}\overset{dist.}{=}\dotp{\y_i, \vu_{j}}$, using \eqref{equ:dotpyu}, we have
\begin{eqnarray}
\bbE_{\Y, \xi}e^{8\la c\frac{1}{n}|\sum_{i=1}^n\xi_i\dotp{\y_i, \vu_{j}}|}&=&\bbE_{\Y}e^{8\la c\frac{1}{n}|\sum_{i=1}^n \dotp{\y_i, \vu_{j}}|}\n\\
&\overset{iv}{\leq}&2\bbE_{\Y}e^{8\la c\frac{1}{n}\sum_{i=1}^n\dotp{\y_i, \vu_{j}}}\n\\
&\leq&e^{\frac{(\|\stt\|^2+1)(8c\la)^2}{2n}+1}, \n
\end{eqnarray}
where inequality (iv) holds for the fact that the distribution of $\frac{1}{n}\sum_{i=1}^n\dotp{\y_i, \vu_{j}}$ is symmetric.
Therefore, combining part 1 and part 2, we have
\begin{eqnarray}
\bbE e^{\la Z}&\leq&\sum_{j=1}^Me^{\frac{(\|\stt\|^2+1)(16c\la)^2}{4}+d}e^{\frac{(\|\stt\|^2+1)(8c\la)^2}{4}+\frac{1}{2}}\n\\
&\leq&e^{\frac{81(\|\stt\|^2+1)c^2\la^2}{n}+3d+\frac{1}{2}}.\n
\end{eqnarray}
Using Markov Inequality:
$$P(Z>\ep)\leq\bbE_Ye^{\la Z-\la\ep}, \forall \ep, \la>0, $$
choosing $\la=\frac{\ep n}{32c^2(\|\stt\|^2+2)}$, we have
\begin{eqnarray}
P(Z>\ep)&\leq& e^{\frac{81c^2\la^2(\|\stt\|^2+1)}{n}+3d+\frac{1}{2}-\la \ep}\n\\
&=&e^{-\frac{n\ep^2}{324c^2(\|\stt\|^2+1)}+3d+\frac{1}{2}}.\n
\end{eqnarray}
Therefore
\begin{eqnarray}
\sup_{\|\vx_b\|\leq c, \|\vx_a\|\leq 1}\|\frac{1}{n}\sum_{i=1}^n(\w_d(\y_i-\vx_a, \vx_b)-\frac{1}{2})\y_i-\bbE \w_d(\Y-\vx_a, \vx_b)\Y\|\leq 36c(\|\stt\|+2)\sqrt{\frac{d+2+\ln(1/\de)}{n}}, \n
\end{eqnarray}
with probability at least $1-\de$.
\end{proof}

\subsubsection{Proof of Lemma \ref{lemma:connecpopsamp}}\label{sec:proof:lem:expconvalast}
Since $\bpa{t+1}$ and $\bpb{t+1}$ are the result of first iteration based on initialization $(\sa{t}, \sb{t})$ in Population EM model where initialization $(\sa{t}, \sb{t})$ satisfying the corresponding condition mentioned in the lemma. Hence to prove the lemma holds for all $t\geq 0$, it is sufficient to prove that for any initialization $(\pa{0}, \pb{0})$ satisfying the same condition, we have $\|\pa{1}\|\leq \kappa_a\|\pa{0}\|$ for \eqref{equ:empcon12} and $\|\pb{1}-\stt\|\leq \kappa_b\|\pb{0}-\stt\|+\sqrt{c_b\|\pa{0}\|}$ for \eqref{equ:empcon1}.
To achieve this goal, we use the notations and definition that are summarized in Appendix \ref{sec:notation}. We first prove the first claim:
\begin{eqnarray}
\|\pa{1}\|\leq \kappa_a\|\pa{0}\|.\label{equ:connecpopsampeq1}
\end{eqnarray}
If $\pa{0}=\v0$, we immediately have \eqref{equ:connecpopsampeq1} holds. If $\pa{0}\neq \v0$, because of
Lemma \ref{lemma:sign} and Lemma \ref{lemma:symmetrizeinitial}, we assume $\dotp{\pa{0}, \pb{0}}>0$
without loss of generality, thus $\tpai{1, 0}>0$.
Since $\tpb{1}(1-2\pp{1})=\tpa{1}$, we know they are in the same direction, thus the angle between $\pa{1}$ and $\stt$ is the same angle between $\pb{1}$ and $\stt$, i.e., $\pbeta{1}$. Furthermore, according to Lemma \ref{lem:planarestimates}, we have $\pbeta{1}\leq \pbeta{0}$. Hence we have
\begin{eqnarray}
\|\pa{1}\|=\frac{\tpai{1, 1}}{\cos{\pbeta{1}}}\leq\frac{\tpai{1, 1}}{\cos{\pbeta{0}}}. \label{equ:connecpopsampeq2}
\end{eqnarray}
Therefore, we need to bound $\tpai{1, 1}$ and $\frac{1}{\cos{\pbeta{0}}}$. According to \eqref{eq:a_1boundconvergence} and Lemma \ref{lemma:aconvergeupperbound} we have,
\[
\tpai{1, 1} = \frac{\funcQ ( \tpai{1, 0}, \|\pb{0}\|, \ptt{1, 0})(1 -2\funcP( \tpai{1, 0}, \|\pb{0}\|, \ptt{1, 0}))}{\funcP( \tpai{1, 0}, \|\pb{0}\|, \ptt{1, 0}) (1-\funcP( \tpai{1, 0}, \|\pb{0}\|, \ptt{1, 0})) } \leq \kappa_a' \tpai{1, 0} \leq \kappa_a' \|\pa{0}\|,
\]
where $\kappa_a'\in (0, 1)$ is a continuous function of $\ptt{1, 0}>0$. Since the condition of $\|\pb{0}-\stt\|\leq \frac{1}{2}\|\stt\|$ implies that $\ptt{1, 0}\geq \frac{\sqrt{3}}{2}\|\stt\|>0$, we have
\begin{eqnarray}
\kappa_a\triangleq \sup_{\ptt{1, 0}\in [\frac{\sqrt{3}}{2}\|\stt\|, \|\stt\|]}\sqrt{\kappa_a'(\ptt{1, 0})}\in (0, 1), \n
\end{eqnarray}
and $\kappa_a$ only depends on $\stt$. Now for $\frac{1}{\cos{\pbeta{0}}}$, by the condition of $\|\pb{0}-\stt\|\leq \sqrt{1-(\kappa_a)^2}\|\stt\|$,
we have $\cos{\pbeta{0}}\geq \kappa_a$. Hence, combining the two parts in \eqref{equ:connecpopsampeq2}, we have
\begin{eqnarray}
\|\pa{1}\|\leq\frac{\tpai{1, 1}}{\cos{\pbeta{0}}}\leq \frac{\kappa_a'(\ptt{1, 0})}{\kappa_a} \|\pa{0}\|\leq \kappa_a\|\pa{0}\|.\n
\end{eqnarray}
Hence \eqref{equ:connecpopsampeq1} holds. Next we prove the second claim:
\begin{eqnarray}
\|\pb{1}-\stt\|\leq \kappa_b\|\pb{0}-\stt\|+\sqrt{c_b\|\pa{0}\|}.\n
\end{eqnarray}
According to Lemma \ref{lemma:conratepre1} we can conclude there exists $\delta_a'\in (0, 1)$, $\kappa_b\in (0, 1)$ and $c_b>0$ such that
that if $\|\pa{0}\| \leq \delta_a'$, then
\begin{eqnarray}
\|\pb{1}-\stt\|&\leq&\sqrt{\kappa_{b}^2\|\pb{0}-\stt\|^2+c_{b}\|\pa{0}\|}\leq \kappa_b\|\pb{0}-\stt\|+\sqrt{c_{b}\|\pa{0}\|}, \n
\end{eqnarray}
where $\delta_a', \kappa_b$ and $c_b$ only depend on $U_a=1, L_b=\frac{1}{2}\|\stt\|, U_b=\frac{3}{2}\|\stt\|, L_\theta=\frac{\sqrt{3}}{2}\|\stt\|$ and $\|\stt\|$. Hence $\delta_a', \kappa_b$ and $c_b$ only depend on $\stt$.
This completes the proof.

\section{Proofs of Auxiliary Results}\label{sec:auxiliary}

\subsection{Proof of Lemma \ref{lemma:sign}} \label{sec:proof:lemma:sign}
From \eqref{equ:itea} and \eqref{equ:itea}, we know that
\begin{eqnarray}
\dotp{\pb{t+1}, \stt}&=&\frac{\dotp{\pq{t+1}, \stt}}{\pp{t+1}(1-\pp{t+1})}, \n\\
\dotp{\pa{t+1}, \pb{t+1}}&=&\|\pb{t+1}\|^2(1-2\pp{t+1}).\n
\end{eqnarray}
Since $\pp{t+1}\in (0, 1)$ and $\|\stt\|>0$, we have
\begin{eqnarray}
\sgn (\dotp{\pb{t+1}, \stt})&=&\sgn(\dotp{\pq{t+1}, \stt})\n\\
\sgn(\dotp{\pa{t+1}, \pb{t+1}})&=&\sgn(1-2\pp{t+1})\sgn(\|\pb{t+1}\|^2)\n\\
&=&\sgn(1-2\pp{t+1})|\sgn(\dotp{\pb{t+1}, \stt})|.\n
\end{eqnarray}
Hence if we have
\begin{eqnarray}
\sgn(\dotp{\pq{t+1}, \stt})=\sgn(\dotp{\pb{t}, \stt}), \label{eq:signcondition1}
\end{eqnarray}
and
\begin{eqnarray}
\sgn(1-2\pp{t+1})=\sgn(\dotp{\pa{t}, \pb{t}}), \label{eq:signcondition2}
\end{eqnarray}
then immediately, we have
\begin{eqnarray}
\sgn(\dotp{\pb{t+1}, \stt})&=&\sgn(\dotp{\pb{t}, \stt}), \n
\end{eqnarray}
and
\begin{eqnarray}
\sgn(\dotp{\pa{t+1}, \pb{t+1}})&=&\sgn(1-2\pp{t+1})|\sgn(\dotp{\pb{t+1}, \stt})|\n\\
&=&\sgn(\dotp{\pa{t}, \pb{t}})|\sgn(\dotp{\pb{t}, \stt})|\n\\
&=&\sgn(\dotp{\pa{t}, \pb{t}}).\n
\end{eqnarray}
Hence our next goal is to prove \eqref{eq:signcondition1} and \eqref{eq:signcondition2}. Consider the rotation matrix $\vO$ for which $\tstt \triangleq \vO \stt$ has all its coordinates except the first one equal to zero, i.e., $\tstt = (\|\stt\|, 0, \ldots, 0)^\t$. Also, let $\tpb{t} \triangleq \vO \pb{t}$, and $\tpa{t} \triangleq \vO \pa{t}$. According to Lemma \ref{lemma:rotate}, we have
\begin{eqnarray}
&& \dotp{ \tpq{t+1}, \tstt}\ =\ \dotp{ \tstt, \bbE \w_d(\Y-\tpa{t}, \tpb{t})\Y}\n\\
&=&\|\tstt\|\int \w_d(\y-\tpa{t}, \tpb{t}) y_1\pdf{d, \y, \tstt}\dif \y\n\\
&=& \|\tstt\|\int \frac{e^{-\sum_{i=2}^d y_i^2/2}}{\sqrt{2\pi}^{d-1}}(\int \w_d(\y-\tpa{t}, \tpb{t})y_1\pdf{, y_1, \|\tstt\|}\dif y_1)\dif y_2\cdots \dif y_d\n\\\label{equ:qta}
\end{eqnarray}
Hence, define $\y_{2\ldots d}=(y_2, \cdots, y_d)$ and $B(\y_{2\ldots d})\triangleq\sum_{i=2}^d y_ib^{\langle t \rangle}_i$, then to prove \eqref{eq:signcondition1}, it is sufficient to prove
\begin{eqnarray}
\sgn(\int \w_d(\y-\tpa{t}, \tpb{t})y_1\pdf{, y_1, \|\tstt\|}\dif y_1)&=&\sgn(\dotp{\pb{t}, \stt})\ =\ \sgn(\|\stt\|\tpbi{1, t})\n\\
&=&\sgn(\tpbi{1, t}), \ \ \forall \y_{2\ldots d}, B(\y_{2\ldots d}).\n
\end{eqnarray}
Note that
\begin{eqnarray}
&~&\int \w_d(\y-\tpa{t}, \tpb{t})y_1\pdf{, y_1, \|\tstt\|}\dif y_1\n\\
&=&\int_{0}^{+\infty} (\w_d((y_1, \y_{2\ldots d})^\t-\tpa{t}, \tpb{t})- \w_d((-y_1, \y_{2\ldots d})^\t-\tpa{t}, \tpb{t}))y_1\pdf{, y_1, \|\tstt\|}\dif y_1\n\\
&=&\int_{0}^{+\infty}\frac{e^{2y_1\tpbi{1, t}}-e^{-2y_1\tpbi{1, t}}}{e^{2y_1\tpbi{1, t}}+e^{-2y_1\tpbi{1, t}}+e^{2B(\y_{2\ldots d})-2\langle \tpa{t}, \tpb{t}\rangle}+e^{-2B(\y_{2\ldots d})+2\langle \tpa{t}, \tpb{t}\rangle}}y_1\pdf{, y_1, \|\tstt\|}\dif y_1.\n
\end{eqnarray}
Hence, we have
\[\sgn(\int \w_d(\y-\tpa{t}, \tpb{t})y_1\pdf{, y_1, \|\tstt\|}\dif y_1)=\sgn(\tpbi{1, t}), \ \ \forall \y_{2\ldots d}, B(\y_{2\ldots d}).\]
To prove \eqref{eq:signcondition2}, according to Lemma \ref{lemma:rotate}, we have
\begin{eqnarray}
2\pp{t+1}-1&=&\bbE (2\w_d(\Y-\tpa{t}, \tpb{t})-1)\n\\
&=&\bbE (\w_d(\Y-\tpa{t}, \tpb{t})+\w_d(-\Y-\tpa{t}, \tpb{t})-1)\n\\
&=&\bbE(\frac{e^{2\dotp{\Y, \tpb{t}}}+e^{-2\dotp{\Y, \tpb{t}}}+2e^{-2\dotp{\tpa{t}, \tpb{t}}}}{e^{2\dotp{\Y, \tpb{t}}}+e^{-2\dotp{\Y, \tpb{t}}}+e^{2\dotp{\tpa{t}, \tpb{t}}}+e^{-2\dotp{\tpa{t}, \tpb{t}}}}-1)\n\\
&=&\bbE\frac{e^{-2\dotp{\tpa{t}, \tpb{t}}}-e^{2\dotp{\tpa{t}, \tpb{t}}}}{e^{2\dotp{\Y, \tpb{t}}}+e^{-2\dotp{\Y, \tpb{t}}}+e^{2\dotp{\tpa{t}, \tpb{t}}}+e^{-2\dotp{\tpa{t}, \tpb{t}}}}.\n
\end{eqnarray}
Hence, we have
\begin{eqnarray}
\sgn(1-2\pp{t+1})&=&\sgn(\dotp{\tpa{t}, \tpb{t}}).\n
\end{eqnarray}
This completes the proof of this lemma.

\subsection{Proof of Lemma \ref{lemma:symmetrizeinitial}} \label{sec:proof:lemma:symmetrizeinitial}
\begin{proof}
We use induction to prove for claim (i) and the proof for claim (ii) is similar. Clearly (i) holds for $t=0$. If (i) holds for $t$, then for $t+1$, note that $\forall \va, \vb\in \bbR^d$
\begin{eqnarray}
\bbE \w_d(\Y-\va, \vb)\Y&=&\bbE\w_d(\Y+\va, \vb)\Y, \n\\
\bbE \w_d(\Y-\va, \vb)&=&1-\bbE\w_d(\Y+\va, \vb), \n\\
\bbE \w_d(\Y-\va, \vb)\Y&=&-\bbE\w_d(\Y-\va, -\vb)\Y, \n\\
\bbE \w_d(\Y-\va, \vb)&=&1-\bbE\w_d(\Y-\va, -\vb).\n
\end{eqnarray}
Hence by definition of $\pa{t}$ and $\pb{t}$, it is straight forward to see
\begin{eqnarray}
\pa{t+1}_{(1)}&=&-\pa{t+1}_{(2)}, \n\\
\pb{t+1}_{(1)}&=&\pb{t+1}_{(2)}.\n
\end{eqnarray}
Hence claim (i) holds for $t+1$. By induction, we know claim (i) holds for all $ t\geq 0$.
\end{proof}

\subsection{Proof of Lemma \ref{lem:Plowerbound}}\label{sec:proofoflemmaPlowerbound}
According to the definition of $P(x_a, x_b, x_\theta)$ presented in Appendix \ref{sec:notation} we have
\begin{eqnarray}
\funcP(x_{a}, x_{b}, x_{\theta})&=&\int \w(y-x_a, x_b)\pdf{, y, x_\theta}\dif y\n\\
&=&\int\w(y, x_b)\pdf{, y+x_a, x_\theta}\dif y\n\\
&=&\int_{y\geq 0}\pdf{, y+x_a, x_\theta}\dif y+\int_{y\geq 0}\w(-y, x_b)(\pdf{, y-x_a, x_\theta}-\pdf{, y+x_a, x_\theta})\dif y.\n\label{eq:plow2}\\
\end{eqnarray}
where the last equality used the fact that $\w(y, x_b)+\w(-y, x_b)=1$.
If $x_{a}\geq x_{\theta}\geq 0$, then
\begin{eqnarray}
\lefteqn{2(\pdf{, y-x_a, x_\theta}-\pdf{, y+x_a, x_\theta})}\n\\
&=&\phi(y-x_a+x_\theta)+\phi(y-x_a-x_\theta)-\phi(y+x_a+x_\theta)+\phi(y+x_a-x_\theta)\n\\
&=&\phi(y-x_a+x_\theta)-\phi(y+x_a-x_\theta)+\phi(y-x_a-x_\theta)-\phi(y+x_a+x_\theta)\n\\
&\geq& 0, \ \ \forall y\geq 0.\n
\end{eqnarray}
Hence with \eqref{eq:plow2}, the above equation implies that if $x_{a}\geq x_{\theta}\geq 0$, we have
\begin{eqnarray}
\funcP(x_{a}, x_{b}, x_{\theta})&\geq&\int_{y\geq 0}\pdf{, y+x_a, x_\theta}\dif y\n\\
&=&\frac{1}{2}(1-\Phi(x_{a}-x_{\theta}))+\frac{1}{2}(1-\Phi(x_{a}+x_{\theta})).\label{eq:plow3}
\end{eqnarray}
This completes the proof of the first part of Lemma \ref{lem:Plowerbound}. Now we discuss the second part, i.e., the case $x_{a}<x_{\theta}$. First note that $\funcP$ is a decreasing function of $x_a$, since $x_b\geq 0$ and
\begin{eqnarray}
\frac{\partial \funcP(x_a, x_b, x_\theta)}{\partial x_a}=-\int\frac{2x_b}{(e^{yx_b-x_ax_b}+e^{-yx_b+x_ax_b})^2}\pdf{, y, x_\theta}\dif y\leq0.\n
\end{eqnarray}
Therefore, we have
\begin{eqnarray}
\funcP(x_{a}, x_{b}, x_{\theta})\ \geq\ \funcP(x_{\theta}, x_{b}, x_{\theta})\ \geq\ \frac{1}{4}+\frac{1}{2}(1-\Phi(2x_{\theta}))\n
\end{eqnarray}
where the last inequality holds because $x_a=x_\theta$ satisfies the condition of \eqref{eq:plow3}. Hence, it immediately gives us that
if $x_{a}<x_{\theta}$, then $$\funcP_1(x_{a}, x_{b}, x_{\theta})\geq\frac{1}{4}.$$
This completes the proof.

\subsection{Proof of Lemma \ref{lem:Q1overPbound}}\label{sec:proofoflemmaQ1overPbound}
We warn the reader that in this proof we use the proof of Lemma \ref{lem:Plowerbound}, presented in the last section. According to the definition of  $\funcQ$ presented in Appendix \ref{sec:notation}, we have
\begin{eqnarray}
\lefteqn{\funcQ(x_{a}, x_{b}, x_{\theta})=\int \w(y-x_a, x_b)y\pdf{, y, x_\theta}\dif y}\n\\
&=&x_{a}\funcP(x_{a}, x_{b}, x_{\theta})+\int \w(y-x_a, x_b)(y-x_a)\pdf{, y, x_\theta}\dif y\n\\
&=&x_{a}\funcP(x_{a}, x_{b}, x_{\theta})+\int \w(y, x_b)y\pdf{, y+x_a, x_\theta}\dif y\n\\
&<&x_{a}\funcP(x_{a}, x_{b}, x_{\theta})+\int_{y\geq 0}y\pdf{, y+x_a, x_\theta}\dif y\n\\
&=&x_{a}\funcP(x_{a}, x_{b}, x_{\theta})+\frac{1}{2}\int_{y\geq 0}\{(y+x_{a}-x_{\theta})\phi(y+x_{a}-x_{\theta})+(y+x_{a}+x_{\theta})\phi(y+x_{a}+x_{\theta})\}\dif y\n\\
&~&-\frac{1}{2}\left\{(x_{a}-x_{\theta})\int_{y\geq 0}\phi(y+x_{a}-x_{\theta})\dif y-(x_{a}+x_{\theta})\int_{y\geq 0}\phi(y+x_{a}+x_{\theta})\dif y\right\}\n\\
&=&x_{a}\funcP(x_{a}, x_{b}, x_{\theta})+\frac{1}{2}\left\{\phi(x_{\theta}-x_{a})-(x_{a}-x_{\theta})(1-\Phi(x_{a}-x_{\theta}))\right\}\n\\
&~&+\frac{1}{2}\left\{\phi(x_{\theta}+x_{a})-(x_{a}+x_{\theta})(1-\Phi(x_{a}+x_{\theta}))\right\}\n\\
&=&x_{a}\funcP(x_{a}, x_{b}, x_{\theta})+\frac{1}{2}(W(x_a+x_\theta)+W(x_a-x_\theta)), \label{eq:equq2}
\end{eqnarray}
where $W(x)=\phi(x)-x(1-\Phi(x))$. Therefore we should find an upper bound for $W(x)$. Towards this goal we use the following lemma:
\begin{lemma}\label{lemma:gaussian}
Let $\phi(x), \Phi(x)$ denote the pdf and CDF of standard Gaussian respectively. Then we have
$$\frac{\phi(x)}{1-\Phi(x)}<
x+\sqrt{\frac{2}{\pi}}, \ \ \forall x>0.$$
\end{lemma}

The proof of this Lemma is presented in Appendix \ref{sec:proofofgaussian}. Therefore from this lemma, we have $$W(x)<\sqrt{\frac{2}{\pi}}(1-\Phi(x)).$$
Hence we can upper bound \eqref{eq:equq2} by the following inequality:
\begin{eqnarray}
\funcQ(x_{a}, x_{b}, x_{\theta})\leq x_{a}\funcP(x_{a}, x_{b}, x_{\theta})+\sqrt{\frac{2}{\pi}}\{\frac{1}{2}(1-\Phi(x_a+x_\theta))+\frac{1}{2}(1-\Phi(x_a-x_\theta))\}.\n
\end{eqnarray}
From Lemma \ref{lem:Plowerbound}, we have
\begin{eqnarray}
\funcP(x_a, x_b, x_\theta)\geq \frac{1}{2}(1-\Phi(x_{a}-x_{\theta}))+\frac{1}{2}(1-\Phi(x_{a}+x_{\theta})), \ \ \forall x_a\geq x_\theta.\n
\end{eqnarray}
Therefore, if $x_{a}\geq x_{\theta}$, then we have
\begin{eqnarray}
\funcQ(x_{a}, x_{b}, x_{\theta})&\leq&x_{a}\funcP(x_{a}, x_{b}, x_{\theta})+\sqrt{\frac{2}{\pi}}\{\frac{1}{2}(1-\Phi(x_a+x_\theta))+\frac{1}{2}(1-\Phi(x_a-x_\theta))\}\n\\
&\leq&x_{a}\funcP(x_{a}, x_{b}, x_{\theta})+\sqrt{\frac{2}{\pi}}\funcP(x_{a}, x_{b}, x_{\theta}), \n\\
\end{eqnarray}
which completes the proof.

\subsection{Proof of Lemma \ref{lemma:gaussian}}\label{sec:proofofgaussian}
It is equivalent to show that
$$r(x)\triangleq(x+\sqrt{\frac{2}{\pi}})(1-\Phi(x))-\phi(x)>0, \ \ \forall x>0.$$
Taking the first derivative of the left hand side, we have
\begin{eqnarray}
\frac{\dif r(x)}{\dif x}=1-\Phi(x)-\phi(x)(x+\sqrt{\frac{2}{\pi}})+x\phi(x)=1-\Phi(x)-\sqrt{\frac{2}{\pi}}\phi(x).\n
\end{eqnarray}
Taking the second derivative, we have
\begin{eqnarray}
\frac{\dif^2r(x)}{\dif x^2}=-\phi(x)+\sqrt{\frac{2}{\pi}}x\phi(x)=(\sqrt{\frac{2}{\pi}}x-1)\phi(x).\n
\end{eqnarray}
Hence, we have $r''(x)<0$ if $x<\sqrt{\pi/2}$ and $r''(x)>0$ if $x>\sqrt{\pi/2}$. Therefore, $r'(x)$ is first strictly decreasing then strictly increasing function of $x$ for $x\geq 0$. Since $r'(0)=1/2-1/\pi>0$, $r'(\sqrt{\pi/2})=-0.04008391$ and
\[\lim_{x\rightarrow \infty}r'(x)\ =\ \lim_{x\rightarrow \infty}1-\Phi(x)-\sqrt{\frac{2}{\pi}}\phi(x)\ =\ 0, \]
we know there exists $x_0\in (0, \sqrt{\pi/2})$ such that $r'(x)>0$ if $x<x_0$ and $r'(x)<0$ if $x>x_0$. Hence, $r(x)$ is first strictly increasing and then strictly decreasing function of $x\geq 0$. Since $r(x)=0$ and
\begin{eqnarray}
|\lim_{x\rightarrow \infty}r(x)|&\leq& \lim_{x\rightarrow \infty}(x+\sqrt{\frac{2}{\pi}})(1-\Phi(x))+\phi(x)\n\\
&\leq&2\lim_{x\rightarrow \infty}\int_{y=x}^{+\infty}x\phi(y)\dif y\n\\
&\leq& 2\lim_{x\rightarrow \infty}\int_{y=x}^{+\infty}y\phi(y)\dif y\n\\
&=&2\lim_{x\rightarrow \infty}\phi(x)=0.\n
\end{eqnarray}
Hence, we have $r(x)>0, \ \forall x>0$. This completes the proof of this Lemma.

\subsection{Proof of Lemma \ref{lem:lowerboundderivative}}\label{sec:proofoflemmalowerboundderivative}
We first calculate the derivative $\frac{\partial \funcQ(x_{a}, x_{b}, x_{\theta})}{\partial x_{b}}$ at zero:
\begin{eqnarray}
\frac{\partial \funcQ(x_{a}, x_{b}, x_{\theta})}{\partial x_{b}}|_{x_b=0}&=&\int\frac{2(y-x_{a})}{(e^{yx_{b}-x_{a}x_{b}}+e^{-yx_{b}+x_{a}x_{b}})^2}y\pdf{, y, x_\theta}\dif y|_{x_b=0}\n\\
&=&\frac{1}{2}\int\frac{y^2-x_a y}{2}(\phi(y-x_{\theta})+\phi(y+x_{\theta}))\dif y\n\\
&=&\frac{1}{2}(1+x_{\theta}^2).\n
\end{eqnarray}
This derivative is clearly larger than $0.5$. Now we prove the main result by contradiction. Suppose that the claim of the lemma is not correct. Then, for any fixed $\{c_{U, 1}, |\ptt{1, 0}|, \|\stt\|\}$, $\forall \de>0$, we have $a_{\de}\in [0, c_{U, 1}], b_{\de}\in [0, \de], \theta_{\de}\in [\ptt{1, 0}, \|\stt\|]$ such that
$$\frac{\partial \funcQ(x_a, x_b, x_\theta)}{\partial x_b}|_{x_{a}=a_{\de}, x_{b}=b_{\de}, x_{\theta}=\theta_{\de}}<\frac{1}{2}.$$
Therefore, for any sequence $\{\de_i\}$ such that $\de_i\rightarrow 0$, we have
$$a_{\de_i}\in [0, c_{U, 1}], b_{\de_i}\in [0, \de_i], \theta_{\de_i}\in [\ptt{1, 0}, \|\stt\|].$$
Since the sequence $\{a_{\delta_i}, b_{\delta_i}, \theta_{\delta_i}\}_{i=1}^{\infty}$, belong to a compact set,  there exists a subsequence $\de_{i_j}$ such that $\{(a_{\de_{i_j}}, b_{\de_{i_j}}, \theta_{\de_{i_j}})\}$ converges to a limit $(a^{\infty}, b^{\infty}, \theta^{\infty})$ satisfying $$a^{\infty}\in [0, c_{U, 1}], b^{\infty}\in[0, \lim_{j\rightarrow \infty}\de_{i_j}=0], \theta^{\infty}\in [\ptt{1, 0}, \|\stt\|].$$
By continuity of $\frac{\partial \funcQ(x_a, x_b, x_\theta)}{\partial x_b}$, we have
\begin{eqnarray}
\frac{1}{2}&\geq& \lim_{j\rightarrow \infty}\frac{\partial \funcQ(x_a, x_b, x_\theta)}{\partial x_b}|_{x_a=a_{\de_{i_j}}, x_b=b_{\de_{i_j}}, x_\theta=\theta_{\de_{i_j}}}\n\\
&=&\frac{\partial \funcQ(x_a, x_b, x_\theta)}{\partial x_b}|_{x_a=a^{\infty}, x_b=0, x_\theta=\theta^{\infty}}\n\\
&=&\frac{1}{2}(1+(\theta^{\infty})^2)\ >\ \frac{1}{2}.\n
\end{eqnarray}
This contradiction proves that Lemma \ref{lem:lowerboundderivative} is correct.

\subsection{Proof of Lemma \ref{lem:concavityK}}\label{sec:proofoflemmaconcavityK}
Firs note that if $x=0$, then
\[
K(0, x_b) = \int \frac{e^{yx_b}-e^{-yx_b}}{2(e^{yx_b}+e^{-yx_b})}\frac{1}{\sqrt{2\pi}}e^{-y^2/2}\dif y,
\]
which is the integral of an odd function and is hence equal to zero. To prove that the function is increasing and concave for $x\geq 0$, we calculate its derivatives. It is straightforward to see that
\begin{eqnarray}
\frac{\partial K(x, x_b)}{\partial x}&=&\int\frac{e^{yx_b}-e^{-yx_b}}{2(e^{yx_b}+e^{-yx_b})}(y-x)\phi(y-x)\dif y\n\\
&=&-\int\frac{e^{yx_b}-e^{-yx_b}}{2(e^{yx_b}+e^{-yx_b})}\dif \phi(y-x)\n\\
&=&\int \phi(y-x)\frac{2x_b}{(e^{yx_b}+e^{-yx_b})^2}\dif y>0, \n
\end{eqnarray}
where the last equality is the result of integration by parts. Similarly, $\forall x\geq 0$
\begin{eqnarray}
\frac{\partial^2 K(x, x_b)}{\partial x^2}&=&\int (y-x)\phi(y-x)\frac{2x_b}{(e^{yx_b}+e^{-yx_b})^2}\dif y\n\\
&=&-\int \frac{2x_b}{(e^{yx_b}+e^{-yx_b})^2}\dif \phi(y-x)\n\\
&\overset{(a)}{=}&-\int \phi(y-x)\frac{4x_b^2(e^{yx_b}-e^{-yx_b})}{(e^{yx_b}+e^{-yx_b})^3}\dif y\n\\
&=&\int_{0}^{\infty}(\phi(y+x)-\phi(y-x))\frac{4x_b^2(e^{yx_b}-e^{-yx_b})}{(e^{yx_b}+e^{-yx_b})^3}\dif y<0, \n
\end{eqnarray}
where equality (a) is an application of integration by parts.

\subsection{Proof of Lemma \ref{lem:lupperbound}}\label{sec:proofoflemmalupperbound}
Recall the definition of $l(x)$ in Appendix \ref{proof:lemubkappa_a}:
$$l(x)=x(1-2\Phi(-x))+2\phi(x).$$
Define
\begin{eqnarray}
J(x)\ =\ \frac{1}{2}(x-l(x)(1-2\Phi(-x)))\ =\ 2\phi(x)\Phi(-x)+2x\Phi(-x)-\phi(x)-2x\Phi(-x)^2.\n
\end{eqnarray}
We would like to show that $J(x) \geq 0$. Hence, we analyze the shape of the function $J(x)$ by taking the derivatives, for all $x>0$
\begin{eqnarray}
\frac{\dif J(x)}{\dif x}&=&-2\phi(x)^2+2\Phi(-x)-x\phi(x)-2\Phi(-x)^2\n\\
\frac{\dif^2 J(x)}{\dif x^2}&=&\phi(x)(4\phi(x)x-3+x^2+4\Phi(-x))\n\\
\frac{\dif J''(x)/\phi(x)}{\dif x}&=&2x(1-2x\phi(x))\geq 2x(1-\sqrt{\frac{2}{\pi}})>0.\n
\end{eqnarray}
Therefore $J''(x)/\phi(x)$ is an strictly increasing function of $x$. With $J''(0)<0$ and $J''(10)>0$, we have $J'(x)$ is first strictly decreasing then strictly increasing function of $x$. Since $J'(0)=1/2-1/\pi>0$ and $\lim_{x\rightarrow \infty}J'(x)=0$, we know $J(x)$ achieves its minimum at either $0$ or $\infty$. Since $J(0)=0$ and $\lim_{x\rightarrow \infty}J(x)=0$, we have
$$J(x)>0, \ \ \forall x>0.$$
This completes the proof of this lemma.

\subsection{Proof of Lemma \ref{lem:Fpropert1}}\label{sec:proofoflemmaFpropert1}
According to the definition of functions $F(x_b, x_\theta)$, $Q(x_a, x_b, \theta)$ and \eqref{eq:QFrelation} in Appendix \ref{sec:notation} we have
\begin{eqnarray}
F(x_b, x_\theta)&=&2\funcQ(0, x_b, x_\theta)\n\\
&=&\int\frac{e^{(y+x_{\theta})x_{b}}-e^{-(y+x_{\theta})x_{b}}}{e^{(y+x_{\theta})x_{b}}+e^{-(y+x_{\theta})x_{b}}}(y+x_{\theta})\phi(y)\dif y.\label{eq:Fpropert11}
\end{eqnarray}
To prove the concavity of this function we show that $$\frac{\partial^2 F(x_b, x_\theta)}{\partial x_b^2} \leq 0.$$ We have
\begin{eqnarray}
\frac{\partial F(x_b, x_\theta)}{\partial x_b}&=&\int\frac{4}{(e^{(y+x_{\theta})x_{b}}+e^{-(y+x_{\theta})x_{b}})^2}(y+x_{\theta})^2\phi(y)\dif y \geq 0, \n\\
\frac{\partial^2 F(x_b, x_\theta)}{\partial x_b^2}&=&\int\frac{-8(e^{(y+x_{\theta})x_{b}}-e^{-(y+x_{\theta})x_{b}})}{(e^{(y+x_{\theta})x_{b}}+e^{-(y+x_{\theta})x_{b}})^3}(y+x_{\theta})^3\phi(y)\dif y\n\\
&=&-\int\frac{8(e^{yx_{b}}-e^{-yx_{b}})}{(e^{yx_{b}}+e^{-yx_{b}})^3}y^3\phi(y-x_{\theta})\dif y\leq 0.\n
\end{eqnarray}
Therefore $F$ is increasing and strictly concave in $x_b>0$. Now we only need to calculate $F(0, x_\theta)$ and $F(x_\theta, x_\theta)$. From \eqref{eq:Fpropert11}, we have
\begin{eqnarray}
F(0, x_\theta)\ =\ \int\frac{e^{0}-e^{-0}}{e^{0}+e^{0}}(y+x_{\theta})\phi(y)\dif y\ =\ 0.\n
\end{eqnarray}
Using definition of $F$, we have
\begin{eqnarray}
F(x_\theta, x_\theta)&=&2Q(0, x_\theta, x_\theta)\n\\
&=&2\int\frac{e^{yx_{\theta}}}{e^{yx_{\theta}}+e^{-yx_{\theta}}}y\frac{1}{2\sqrt{2\pi}}(e^{-(y-x_{\theta})^2/2}+e^{-(y+x_{\theta})^2/2})\dif y\n\\
&=&\int e^{yx_{\theta}}y\frac{1}{\sqrt{2\pi}}e^{-(y^2+x_{\theta}^2)/2}\dif y\n\\
&=&x_{\theta}\n
\end{eqnarray}
This completes the proof of this lemma.

\subsection{Proof of  Lemma \ref{lemma:1a0conratepre}}\label{sec:proofoflemma1a0conratepre}
According to the definition of function $F$, we have
\begin{eqnarray}
\frac{\partial F(x_b, x_\theta)-x_\theta}{\partial x_b}|_{x_b=x_\theta}&=&\int\frac{4y^2}{(e^{yx_b}+e^{-yx_b})^2}\frac{1}{\sqrt{2\pi}}e^{-(y-x_\theta)^2/2}\dif y|_{x_b=x_\theta}\n\\
&=&\int\frac{2y^2}{e^{yx_\theta}+e^{-yx_\theta}}\frac{1}{\sqrt{2\pi}}e^{-(y^2+x_\theta^2)/2}\dif y\n\\
&\leq&e^{-\frac{x_\theta^2}{2}}.\n
\end{eqnarray}
We claim there exists $\de>0$ is a function of only $L_b, U_b, L_\theta, \|\stt\|$ such that $\forall |x_b-x_\theta|\in [0, \de], x_b\in [L_b, U_b], x_\theta\in [L_\theta, \|\stt\|], $
\begin{eqnarray}
\frac{\partial F(x_b, x_\theta)-x_\theta}{\partial x_b}\leq \frac{1+e^{-\frac{L_\theta^2}{2}}}{2}.\label{equ:bound2}
\end{eqnarray}
We prove it by contradiction. If not, for all $\de>0$, we have $b_{\de}\in [L_b, U_b], \theta_{\de}\in [L_\theta, \|\stt\|], |b_{\de}-\theta_{\de}|\in [0, \de]$ such that
\begin{eqnarray}
\frac{\partial F(x_b, x_\theta)-x_\theta}{\partial x_b}|_{x_b=b_\de, x_\theta=\theta_\de}>\frac{1+e^{-\frac{L_\theta^2}{2}}}{2}.\n
\end{eqnarray}
For any sequence $\{\de_i\}$ such that $\de_i\rightarrow 0$, there exists subsequence $\de_{i_j}$ such that $\{(b_{\de_{i_j}}, \theta_{\de_{i_j}})\}$ converge to the limits $(b^{\infty}, \theta^\infty)$. By compactness of the choice of $x_b, x_\theta$, we have $$b^\infty\in[L_b, U_b], \theta^\infty\in[L_\theta, \|\stt\|], |b^\infty-\theta^\infty|\in [0, \lim_{j\rightarrow \infty}\de_{i_j}=0].$$ By continuity of $\frac{\partial F(x_b, x_\theta)-x_\theta}{\partial x_b}$, we have
\begin{eqnarray}
\frac{1+e^{-\frac{L_\theta^2}{2}}}{2}&\leq& \lim_{j\rightarrow \infty}\frac{\partial F(x_b, x_\theta)-x_\theta}{\partial x_b}|_{x_b=b_{\de_{i_j}}, x_\theta=\theta_{\de_{i_j}}}\n\\
&=&\frac{\partial F(x_b, x_\theta)-x_\theta}{\partial x_b}|_{x_b=x_\theta=\theta^\infty}\n\\
&=&e^{-\frac{(\theta^\infty)^2}{2}}\ <\ \frac{1+e^{-\frac{L_\theta^2}{2}}}{2}.\n
\end{eqnarray}
Contradiction! Hence we have Eq.(\ref{equ:bound2}) holds and $\forall |x_b-x_\theta|\in [0, \de], x_b\in [L_b, U_b], x_\theta\in [L_\theta, \|\stt\|], $
$$|F(x_b, x_\theta)-x_\theta|\ \leq\  |F(x_\theta, x_\theta)|+|\frac{1+e^{-\frac{L_\theta^2}{2}}}{2}(x_b-x_\theta)|\ =\ \frac{1+e^{-\frac{L_\theta^2}{2}}}{2}|x_b-x_\theta|.$$
From Lemma \ref{lem:Fpropert1}, we have $$|F(x_b, x_\theta)-x_\theta|<|x_b-x_\theta|, \ \ \forall |x_b-x_\theta|\notin [0, \de), x_b\in [L_b, U_b], x_\theta\in [L_\theta, \|\stt\|].$$
Let
$$\kappa_b''=\max\{\frac{1+e^{-\frac{L_\theta^2}{2}}}{2}, \sup_{|x_b-x_\theta|\notin [0, \de), x_b\in [L_b, U_b], x_\theta\in [L_\theta, \|\stt\|]}\frac{|F(x_b, x_\theta)-x_\theta|}{|x_b-x_\theta|}\}, $$ by continuity of the function $\frac{|F(x_b, x_\theta)-x_\theta|}{|x_b-x_\theta|}$, we have $\kappa_b''\in (0, 1)$ is a function of only $L_b, U_b, L_\theta, \|\stt\|$ and
\begin{eqnarray}
|F(x_b, x_\theta)-x_\theta|\leq \kappa_b'' |x_b-x_\theta|, \ \ \forall x_b\in [L_b, U_b], x_\theta\in [L_\theta, \|\stt\|].\n
\end{eqnarray}
This completes the proof of this lemma.

\subsection{Cluster Points of Population EM}\label{sec:Laccumulationpoitnsproof}

\begin{lemma}\label{lemma:converge}
Any clustering point $(\va, \vb)$ of the estimates of the Population EM $\{(\pa{t}, \pb{t})\}_t$ satisfy the following equations:
\begin{eqnarray*}
\va&=&\frac{\vfuncq(1-2\vfuncp)}{2\vfuncp(1-\vfuncp)} , \\
\vb&=&\frac{\vfuncq}{2\vfuncp(1-\vfuncp)}, \\
\vfuncq
&=&\mathbb{E} \w_d(\Y-\va, \vb)\Y, \\
\vfuncp&=&\mathbb{E} \w_d(\Y-\va, \vb),
\end{eqnarray*}
where $\Y \sim 0.5 N(-\stt, I) + 0.5 N(\stt, I)$.
\end{lemma}
\begin{proof}
Here is a summary of our strategy to prove this result. We first prove that $\|\pa{t+1}-\pa{t}\| \rightarrow 0$ and $\|\pb{t+1} - \pb{t}\| \rightarrow 0$ as $t \rightarrow \infty$. Then we use the following simple argument to prove that in fact the clustering points must satisfy the above fixed point equations. Suppose that $(\va, \vb)$ is an accumulation point. Then there is a subsequence $\{(\pa{t_i}, \pb{t_i})\}_{i=1}^\infty$ that converges to $(\va, \vb)$. Since we have $\|\pa{t+1}-\pa{t}\| \rightarrow 0$ and $\|\pb{t+1} - \pb{t}\| \rightarrow 0$, we can simply argue that $\{(\pa{t_i+1}, \pb{t_i+1})\}_{i=1}^\infty$ also converges to $(\va, \vb)$. We know that
\begin{eqnarray*}
\pa{t_i+1}&=&\frac{\pq{t_i}(1-2\pp{t_i})}{2\pp{t_i}(1-\pp{t_i})} , \\
\pb{t_i+1}&=&\frac{\pq{t_i}}{2\pp{t_i}(1-\pp{t_i})}, \\
\pq{t_i+1} &=&\mathbb{E} \w_d(\Y-\pa{t}, \pb{t})\Y  \\
\pp{t_i+1}&=&\mathbb{E} \w_d(\Y-\pa{t}, \pb{t}).
\end{eqnarray*}
By taking the limit $i \rightarrow \infty$ from both sides of the above equations we obtain the fixed point equations. Hence, the rest of the section is devoted to the proof of $\|\pa{t+1}-\pa{t}\| \rightarrow 0$ and $\|\pb{t+1} - \pb{t}\| \rightarrow 0$. The technique we us to prove this claim was first developed in [2]. Since $\pa{t}=(\pm{1, t}+\pm{2, t})/2$ and $\pb{t}=(\pm{2, t}-\pm{1, t})/2$, we only need to prove that $\|\pm{1, t+1}-\pm{1, t}\| \rightarrow 0$ and $\|\pm{2, t+1} - \pm{2, t}\| \rightarrow 0$.

Define the following notion of distance between two parameter vectors:
\begin{eqnarray}
D(\veta, \vnu)&=&-\bbE f(z|\Y; \vnu) \sum_{z}\ln{(\frac{f(z|\Y ; \veta)}{f(z|\Y; \vnu)})}, \n
\end{eqnarray}
where $f(\cdot)$ indicates corresponding pdf.
Let $\pm{, t}$ is a shorthand for $(\pm{1, t}, \pm{2, t})$. As the first step of our proof we would like to show that $D(\pm{, t+1}, \pm{, t}) \rightarrow 0$. From \eqref{eq:populEMgeneric}, we have
\begin{eqnarray}
Q_{f} (\veta| \vnu) &=& \bbE \sum_{z}f(z|\Y;\vnu)\ln{(f(z, \Y;\veta))}\ =\ \bbE \ln{(f(\Y;\veta))}+\bbE \sum_{z}f(z|\Y;\vnu)\ln{(f(z|\Y;\veta))}\n\\
&=&\bbE \ln{(f(\Y;\veta))}-D(\veta, \vnu)+H(\vnu, \vnu)\n\\
&=&L(\veta)-D(\veta, \vnu)+H(\vnu, \vnu), \n
\end{eqnarray}
where
\begin{eqnarray}
L(\veta) &\triangleq& \bbE \ln{(f(\Y|\veta))}\ =\ \bbE \ln{(\frac{1}{2}\phi_d(\Y-\veta_1)+\frac{1}{2}\phi_d(\Y-\veta_2))}\n\\
 &\leq& -\frac{d}{2}\ln{2\pi}, \label{eq:likelihoodbound}\\
H(\veta, \vnu) &\triangleq& \bbE \sum_zf(z|y;\vnu)\ln{(f(z|y;\veta))}.
\end{eqnarray}
Hence,
\[
\pm{, t+1}\ =\ \text{argmax}_{\boldsymbol{\mu}'} Q_f(\boldsymbol{\mu}' | \pm{, t})\ =\ \text{argmax}_{\boldsymbol{\mu}'}\{L(\boldsymbol{\mu}')-D(\boldsymbol{\mu}', \pm{, t})\}.
\]
Note that every estimate of Population EM is obtained in a trade-off between maximizing the expected log-likelihood and minimizing the distance between the two consecutive estimates. First note that
\begin{equation}\label{eq:comparetwolike}
L(\pm{, t+1}) - D(\pm{, t+1}, \pm{, t})\ \geq\ L(\pm{, t}) - D(\pm{, t}, \pm{, t})\ =\ L(\pm{, t})
\end{equation}
Hence, $L(\pm{, t+1}) \geq L(\pm{, t}) + D(\pm{, t+1}, \pm{, t})$. Therefore, $\{L(\pm{, t})\}$ is a non-decreasing sequence. Since according to \eqref{eq:likelihoodbound}, $L(\boldsymbol{\mu})$ is upper bounded, thus $\{L(\pm{, t})\}_t$ converges.  Also, according to \eqref{eq:comparetwolike} we have
\begin{equation}
0\ \leq\ D(\pm{, t+1}, \pm{, t}) \leq L(\pm{, t+1}) -L(\pm{, t}) \rightarrow 0, \ \ \ {\rm as} \ \ t \rightarrow \infty.\n
\end{equation}
This implies that $$\{D(\pm{, t+1}, \pm{, t})\}\rightarrow 0, \ \text{as}\ t\rightarrow\infty.$$
Note that $D(\cdot, \cdot)$ is a measure of discrepancy between its two arguments. However, our goal is to show that the Euclidean distance between $\pm{, t}$ and $\pm{, t+1}$ goes to zero. The rest of the proof is devoted to this claim. Since,
\begin{eqnarray}
\pm{, t} &=& \frac{\mathbb{E} f(z| \Y; \pm{, t-1}) \Y}{\mathbb{E} f(z| \Y; \pm{, t-1})}, \n\\
\pm{, t+1} &=& \frac{\mathbb{E} f(z| \Y; \pm{, t}) \Y}{\mathbb{E} f(z| \Y; \pm{, t})}, \n
\end{eqnarray}
in order to prove $\|\pm{, t+1} - \pm{, t}\| \rightarrow 0$, we should show that $\forall z\in\{1, 0\}$
\[\|\mathbb{E} f(z| \Y; \pm{, t+1})\Y- \mathbb{E} f(z| \Y; \pm{, t}) \Y\| \rightarrow 0, \]
and
\[|\mathbb{E} f(z| \Y; \pm{, t+1})- \mathbb{E} f(z| \Y; \pm{, t})| \rightarrow 0.\]
If we define $\psi(x)=-\ln{x}+x-1$, then
\begin{eqnarray}
D(\veta, \vnu)= \bbE\sum_z\psi(\frac{f(z|\Y;\veta)}{f(z|\Y;\vnu)})f(z|\Y;\vnu), \label{equ:defD}
\end{eqnarray}
Since $\psi(x) >0$ for every value of $x>0$, the fact that $D(\pm{, t+1}, \pm{, t}) \rightarrow 0$ implies that $\forall z\in\{1, 0\}$
\begin{eqnarray}
\bbE\psi(\frac{f(z|\Y;\pm{, t+1})}{f(z|\Y;\pm{, t})})f(z|\Y; \pm{, t})&\rightarrow& 0, \ \text{as}\ t\rightarrow\infty.\label{equ:converge1}
\end{eqnarray}
Hence, for all $z\in\{1, 0\}$
\begin{eqnarray}
&&\bbE \psi(\frac{f(z|\Y;\pm{, t+1})}{f(z|\Y; \pm{, t})})f(z|\Y;\pm{, t})\n\\
&=&\bbE\left\{(f(z|\Y;\pm{, t+1})-f(z|\Y;\pm{, t}))-\left(\ln{f(z|\Y;\pm{, t+1})}-\ln{f(z|\Y;\pm{, t})}\right)f(z|\Y; \pm{, t})\right\}\n\\
&\overset{(i)}{=}&\bbE \frac{f(z|\Y;\pm{, t})}{2\xi^2}(f(z|\Y; \pm{, t+1})-f(z|\Y;\pm{, t}))^2\n\\
&\overset{(ii)}{\geq}&\bbE f(z|\Y;\pm{, t})(f(z|\Y;\pm{, t+1})-f(z|\Y; \pm{, t}))^2\n\\
&\geq&\bbE f(z|\Y;\pm{, t})(f(z|\Y; \pm{, t+1})-f(z|\Y;\pm{, t}))^2I(\|\Y\|<M), \ \ \forall t, M>0.\n
\end{eqnarray}
where Equality (i) is the result of the Taylor expansion on $\ln X$ and $\xi$ is a number between $f(z|\Y;\pm{, t+1})$ and $f(z|\Y;\pm{, t})$ and Inequality (ii) holds for the fact that $$f(z|\Y;\pm{, t+1}), f(z|\Y;\pm{, t})\in (0, 1), \ \ \forall z\in\{1, 0\}.$$
Hence, with \eqref{equ:converge1}, we have
\begin{eqnarray}
\bbE f(z|\Y;\pm{, t})(f(z|\Y;\pm{, t+1})-f(z|\Y;\pm{, t}))^2I(\|\Y\|<M)\rightarrow 0, \ \text{as}\ t\rightarrow\infty, \ \forall M>0, z\in\{1, 0\}.\n
\end{eqnarray}
According to Lemma \ref{lemma:upperboundEMestimates} $\{(\pa{t_n}, \pb{t_n})\}_{n=1}^{\infty}$ is in a compact set and hence so is $\{\pm{, t}\}$. Since $f(z|\y; \pm{, t})$ is a continuous function of $\y$ and $\pm{, t}$ with $f(z|\y;\pm{, t})>0$ and compactness of $\{\pm{, t}\}$, there exists a constant $c$ only depending on $M$ such that $$f(z|\y;\pm{, t})>c, \ \ \forall \|\y\|<M, z\in\{1, 0\}, t>0.$$
Therefore, for all $M>0, z\in\{1, 0\}$, we have
\begin{eqnarray}
\bbE (f(z|\Y; \pm{, t+1})-f(z|\Y; \pm{, t}))^2I(\|\Y\|<M)\rightarrow 0, \ \text{as}\ t\rightarrow\infty.\label{equ:converge3}
\end{eqnarray}
Also, for all $t\geq 0, z\in\{1, 0\}$, we have
\begin{eqnarray}
\bbE (f(z|\Y;\pm{, t+1})-f(z|\Y;\pm{, t}))^2I(\|y\|\geq M)\leq \bbE I(\|\Y\|\geq M) \rightarrow 0, \ \text{as}\ M\rightarrow\infty.\label{equ:converge4}
\end{eqnarray}
With \eqref{equ:converge3} and \eqref{equ:converge4}, we have
\begin{eqnarray}
\bbE (f(z|\Y;\pm{, t+1})-f(z|\Y;\pm{, t}))^2&\rightarrow& 0, \ \text{as}\ t\rightarrow\infty, \ \ \forall z\in\{1, 0\}.\n
\end{eqnarray}
Therefore for all $z\in\{1, 0\}$, as $t\rightarrow\infty$, we have,
\begin{eqnarray}
\|\bbE (f(z|\Y;\pm{, t+1})-f(z|\Y;\pm{, t}))\Y\|&\leq& \sqrt{\bbE(f(z|\Y;\pm{, t+1})-f(z|\Y;\pm{, t}))^2\bbE \|\Y\|^2}\rightarrow 0, \n\\
|\bbE f(z|\Y;\pm{, t+1})-f(z|\Y;\pm{, t})|&\leq& \sqrt{\bbE(f(z|\Y;\pm{, t+1})-f(z|\Y;\pm{, t}))^2}\rightarrow 0.\n
\end{eqnarray}
Hence with compactness on sequence $\{\pm{, t}\}$, we have
\begin{eqnarray}
\|\pm{1, t+2}-\pm{1, t+1}\|=\left\|\frac{\bbE f(z=0|\Y;\pm{, t+1})\Y}{\bbE f(z=0|\Y;\pm{, t+1})}-\frac{\bbE f(z=0|\Y;\pm{, t})\Y}{\bbE f(z=0|\Y;\pm{, t})} \right\|\rightarrow 0, \ \text{as}\ t\rightarrow\infty, \n
\end{eqnarray}
and
\begin{eqnarray}
\|\pm{2, t+2}-\pm{2, t+1}\|=\left\|\frac{\bbE f(z=1|\Y;\pm{, t+1})\Y}{\bbE f(z=1|\Y;\pm{, t+1})}-\frac{\bbE f(z=1|\Y;\pm{, t})\Y}{\bbE f(z=1|\Y;\pm{, t})} \right\|\rightarrow 0, \ \text{as}\ t\rightarrow\infty.\n
\end{eqnarray}
This completes the proof of this lemma.
\end{proof}

\end{document}